\documentclass[11pt, reqno]{amsart}
\usepackage[utf8]{inputenc}
\usepackage[english]{babel}
\usepackage{amsmath}
\usepackage{amsfonts}
\usepackage{amssymb}
\usepackage{amsthm}
\usepackage{enumitem}

\usepackage[dvipsnames]{xcolor}
\usepackage{hyperref}
\hypersetup{
    hypertexnames=false,
    colorlinks,
    linkcolor={red!50!black},
    citecolor={blue!50!black},
    urlcolor={blue!80!black}
}

\usepackage{graphicx}
\usepackage[all]{xy}
\usepackage[margin=2.7cm]{geometry}

\newcommand\mycolor[1]{}

\setlist[enumerate]{itemsep=0.3ex, topsep=0.3ex, label={\rm(\arabic*)}}
\setlist[itemize]{itemsep=0.3ex, topsep=0.3ex, leftmargin=4ex}

\newtheorem{theorem}[subsection]{Theorem}
\newtheorem{lemma}[subsection]{Lemma}

\newtheorem{corollary}[subsection]{Corollary}

\newtheorem{remark}[subsection]{Remark}

\newtheorem{definition}[subsection]{Definition}

\makeatletter
\@addtoreset{subsection}{section}
\@addtoreset{equation}{section}
\@addtoreset{figure}{section}
\@addtoreset{table}{section}
\makeatother

\makeatletter
\newcommand\testshape{family=\f@family; series=\f@series; shape=\f@shape.}
\def\myemphInternal#1{\if n\f@shape%
\begingroup\itshape #1\endgroup\/%
\else\begingroup\sf\itshape\small #1\endgroup%
\fi}
\def\myemph{\futurelet\testchar\MaybeOptArgmyemph}
\def\MaybeOptArgmyemph{\ifx[\testchar \let\next\OptArgmyemph
                 \else \let\next\NoOptArgmyemph \fi \next}
\def\OptArgmyemph[#1]#2{\index{#1}\myemphInternal{#2}}
\def\NoOptArgmyemph#1{\myemphInternal{#1}}
\makeatother

\newcommand\term[2][\empty]{\myemph[#1]{#2}}
\newcommand\monoArrow{\lhook\joinrel\rightarrow}

\newcommand\xmonoArrow[1]{\lhook\joinrel\xrightarrow{~#1~}}

\newcommand\epiArrow{\rightarrow\!\!\!\!\!\to}

\newcommand\xepiArrow[1]{\xrightarrow{#1}\!\!\!\!\!\to}

\newcommand\Aman{A}

\newcommand\Cman{C}

\newcommand\Fman{F}

\newcommand\Kman{K}
\newcommand\Lman{L}
\newcommand\Mman{M}

\newcommand\Pman{P}
\newcommand\Qman{Q}

\newcommand\Uman{U}

\newcommand\Xman{X}
\newcommand\Yman{Y}

\newcommand\cov[1]{\tilde{#1}} % 

\newcommand\tQman{\cov{\Qman}}

\newcommand\bC{\mathbb{C}}

\newcommand\bN{\mathbb{N}}

\newcommand\bR{\mathbb{R}}
\newcommand\bZ{\mathbb{Z}}
\newcommand\id{\mathrm{id}}          % identity map    
\newcommand\Int{\mathrm{Int}}        % interior 
\newcommand\myprod{\mathop{\prod}\limits}

\newcommand\GL{\mathrm{GL}}

\newcommand\UnitGroup{\mathbf{1}}

\newcommand\Aut{\mathrm{Aut}}       % automorphisms
\newcommand\Diff{\mathcal{D}}       % diffeomorphisms
\newcommand\Orb{\mathcal{O}}        % orbit
\newcommand\Stab{\mathcal{S}}       % stabilizer
\newcommand\DiffId{\Diff_{\id}}     % identity path components of diffeomorphisms
\newcommand\StabId{\Stab_{\id}}     % identity path components of stabilizer

\newcommand\Cinfty{\mathcal{C}^{\infty}}

\newcommand\Ci[2]{\mathcal{C}^{\infty}(#1,#2)}               % space of C^\infty maps
\newcommand\Cid[2]{\mathcal{C}_{\partial}^{\infty}(#1,#2)}   % space of C^\infty maps taking constant values of the boundary
\newcommand\Stabilizer[1]{\Stab(#1)}
\newcommand\StabilizerId[1]{\StabId(#1)}
\newcommand\StabilizerIsotId[1]{\Stab'(#1)}
\newcommand\Orbit[1]{\Orb(#1)}
\newcommand\OrbitComp[2]{\Orb_{#1}(#2)}
\newcommand\FolStabilizer[1]{\Delta(#1)}

\newcommand\FolStabilizerIsotId[1]{\Delta'(#1)}
\newcommand\FolStabilizerNbh[1]{\Delta_{\nb}(#1)}
\newcommand\GKR{\mathbf{G}}
\newcommand\GrpKR[1]{\GKR(#1)}
\newcommand\GrpKRIsotId[1]{\GKR'(#1)}

\newcommand\fixsymbol{\mathrm{fix}}%{\!\times}
\newcommand\invsymbol{}
\newcommand\nbsymbol{\mathrm{nb}}
\newcommand\folsymbol{*}%{\approx}

\newcommand\DiffInv[3][\empty]{\Diff_{\invsymbol}(#2,#3\ifx\empty #1\relax\else,#1\fi)}

\newcommand\DiffFix[3][\empty]{\Diff_{\fixsymbol}(#2,#3\ifx\empty #1\relax\else,#1\fi)}

\newcommand\DiffNb[3][\empty]{\Diff_{\nbsymbol}(#2,#3\ifx\empty #1\relax\else,#1\fi)}

\newcommand\DiffHFix[3][\empty]{\Diff^{0}_{\fixsymbol}(#2,#3\ifx\empty #1\relax\else,#1\fi)}

\newcommand\DiffHNb[3][\empty]{\Diff^{0}_{\nbsymbol}(#2,#3\ifx\empty #1\relax\else,#1\fi)}

\newcommand\DiffPlusFix[3][\empty]{\Diff^{+}_{\fixsymbol}(#2,#3\ifx\empty #1\relax\else,#1\fi)}

\newcommand\FDiff[2][\empty]{\Diff^{\folsymbol}(#2\ifx\empty #1\relax\else,#1\fi)}
\newcommand\FDiffFix[2][\empty]{\Diff^{\folsymbol}_{\fixsymbol}(#2\ifx\empty #1\relax\else,#1\fi)}
\newcommand\FDiffA[2][\empty]{\Diff^{=}(#2\ifx\empty #1\relax\else,#1\fi)}

\newcommand\VBAut[2][\empty]{\GL(#2\ifx\empty #1\relax\else,#1\fi)}

\newcommand\DiffLP{\Diff}  % leaf preserving diffeomorphisms
\newcommand\DiffLPInv[3][\empty]{\DiffLP_{inv}(#2,#3\ifx\empty#1\relax\else,#1\fi)}

\newcommand\DiffLPFix[3][\empty]{\DiffLP_{fix}(#2,#3\ifx\empty#1\relax\else,#1\fi)}

\newcommand\DiffLPNb[3][\empty]{\DiffLP_{nb}(#2,#3\ifx\empty#1\relax\else,#1\fi)}

\newcommand\func{f}
\newcommand\gfunc{g}
\newcommand\dif{h}
\newcommand\gdif{g}
\newcommand\qdif{q}

\newcommand\px{x}

\newcommand\pw{w}

\newcommand\Circle{S^1}

\newcommand\PrjPlane{\mathbb{RP}^2}

\newcommand{\comment}[1]{}

\newcommand{\DD}[2]{D_{#1}^{#2}}
\newcommand{\EE}[2]{E_{#1}^{#2}}

\newcommand\KerSAct{\Stabilizer{\func;\PlCompSet}}

\newcommand\ael{a}
\newcommand\bel{b}
\newcommand\cel{c}
\newcommand\del{d}
\newcommand\gel{g}

\newcommand\vel{v}

\newcommand\mel{{\color{Mulberry}m}}
\newcommand\nel{{\color{Mulberry}n}}

\newcommand\aeli[1]{\ael_{#1}}
\newcommand\beli[1]{\bel_{#1}}
\newcommand\celi[1]{\cel_{#1}}
\newcommand\deli[1]{\del_{#1}}

\newcommand\qhom{q}
\newcommand\smd[3]{#1\rtimes_{#2}#3}

\newcommand\zhom{\psi}

\newcommand{\nb}{\mathrm{nb}}
\newcommand\DiffNbh{\Diff_{{\nb}}}     % diffeomorphisms

\newcommand\StabilizerNbh[1]{\Stab_{{\nb}}(#1)}  % elements of stabilizer of #1 fixed on some neighborhood
\newcommand\FSp[1]{\mathcal{F}(#1,\Pman)}
\newcommand\FSpR[1]{\mathcal{F}(#1,\bR)}
\newcommand\FSpS[1]{\mathcal{F}(#1,S^1)}

\newcommand\CrComp{K}

\newcommand\YYi[1]{\Yman_{#1}}

\newcommand{\classGroups}{\mathcal{G}}

\newcommand\MBand{\mathbb{M}}%{\text{\sf\"{M}}} %\!\text{\"{o}}}
\newcommand\hMBand{\PrjPlane} %\!\text{\"{o}}}
\newcommand\bdpt{x^{*}}
\newcommand\CompSet{\mathbf{Y}}
\newcommand\PlCompSet{\hat{\CompSet}}

\newcommand\FST[1]{\FolStabilizer{\func|_{#1}, \partial#1}}
\newcommand\PFST[1]{\pi_0\FST{#1}}

\newcommand\ST[1]{\Stabilizer{\func|_{#1}, \partial#1}}
\newcommand\PST[1]{\pi_0\ST{#1}}

\newcommand\GST[1]{\GrpKR{\func|_{#1}, \partial#1}}
\newcommand\GSTB[1]{\GrpKR{\func, \partial#1}}

\newcommand\STN[1]{\StabilizerNbh{\func|_{#1}, \partial#1}}
\newcommand\PSTN[1]{\pi_0\STN{#1}}

\newcommand\STB[1]{\Stabilizer{\func, \partial#1}}
\newcommand\PSTB[1]{\pi_0\STB{#1}}

\newcommand\FSTB[1]{\FolStabilizer{\func, \partial#1}}
\newcommand\PFSTB[1]{\pi_0\FSTB{#1}}

\newcommand\Agrp{{\color{red}A}}
\newcommand\Bgrp{{\color{red}B}}
\newcommand\Cgrp{{\color{red}C}}

\newcommand\Ggrp{{\color{red}G}}
\newcommand\Hgrp{{\color{red}H}}
\newcommand\Kgrp{{\color{red}K}}
\newcommand\Lgrp{{\color{red}L}}

\newcommand\Pgrp{{\color{red}P}}
\newcommand\Qgrp{{\color{red}Q}}

\newcommand\WrGZ[2]{#1\wr_{#2}\bZ}
\newcommand\WrGZZ[3]{#1\wr_{#2,#3}\bZ^2}
\newcommand\WrGHZ[4]{\left(#1,#2\right)\wr_{#3,#4}\bZ}

\newcommand\WrGZm[2]{#1\wr\bZ_{#2}}
\newcommand\WrGZZmn[3]{#1\wr(\bZ_{#2}\times\bZ_{#3})}
\newcommand\WrGHZm[4]{\left(#1,#2\right)\wr_{#3,#4}\bZ_{2#4}}

\newcommand\ahom{\varphi}
\newcommand\bhom{\psi}

\newcommand\fSing{\Sigma_{\func}}
\newcommand\KRGraphf{\Gamma_{\func}}

\newcommand\regU[1]{U_{#1}}

\newcommand\sgOrd{{\color{red}b}}   % order of Z_b     
\newcommand\cntOne{{\color{red}d}}  % number of disks of type T1  
\newcommand\cntTwo{{\color{red}e}}  % number of disks of type T2

\newcommand\sg[1]{\widehat{#1}}
\newcommand\spd[1]{$#1$-special decomposition}

\newcommand\exd[9]{
 #1 \ar@{^(->}[r]\ar@{^(->}[d] & 
 #2 \ar@{^(->}[d]\ar@{->>}[r]  &
 #3 \ar@{^(->}[d] \\
 #4 \ar@{^(->}[r]\ar@{->>}[d] & 
 #5 \ar@{->>}[d]\ar@{->>}[r]  &
 #6 \ar@{->>}[d] \\
 #7 \ar@{^(->}[r] & 
 #8 \ar@{->>}[r]  &
 #9 
}

\newcommand\QQi[1]{\Qman_{#1}}
\newcommand\CCi[1]{\Cman_{#1}}

\newcommand\Amatri[1]{A_{#1}}

\newcommand\trace{\mathrm{tr}}

\newcommand\pdif{\mathbf{h}}
\newcommand\cell{e}
\newcommand\cplus[1]{c^{+}_{#1}}
\newcommand\cmin[1]{c^{-}_{#1}}

\newcommand\tcell{\widetilde{e}}
\newcommand\spdif{\widetilde{\pdif}}

\newcommand\PrjCWPart{\Xi}
\newcommand\SphCWPart{\widetilde{\Xi}}

\newcommand\kdif{k}

\newcommand\orig{\mathsf{0}}
\newcommand\CoRn{\Cinfty_{\orig}(\bR^{n})}

\newcommand\JidR[1]{J_{\bR}(#1)}
\newcommand\JidC[1]{J_{\bC}(#1)}

\newcommand\milnorNum[2]{\mu_{#1}(#2)}
\newcommand\bx{\mathbf{x}}
\newcommand\bz{\mathbf{z}}

\newcommand\holoCn{\mathcal{O}(\bC^n)}
\newcommand\maxId{\mathfrak{m}}

\newcommand\DDX[1]{\mathsf{D}^{#1}}
\newcommand\EEX[1]{\mathsf{E}^{#1}}
\newcommand\DDLast{\mathsf{D}^{\sgOrd-1}}
\newcommand\EELast{\mathsf{E}^{\mel-1}}

\title[Deformational symmetries of functions on non-orientable surfaces]
{Deformational symmetries of smooth functions on non-orientable surfaces}

\author{Iryna Kuznietsova}
\address{Department of Algebra and Topology, Institute of Mathematics of NAS of Ukraine, Teresh\-chenkivska str. 3, Kyiv, 01601, Ukraine}
\email{kuznietsova@imath.kiev.ua}

\author{Sergiy Maksymenko}
\address{Department of Algebra and Topology, Institute of Mathematics of NAS of Ukraine, Teresh\-chenkivska str. 3, Kyiv, 01601, Ukraine}
\curraddr{}

\email{maks@imath.kiev.ua}


\thanks{This work was supported by a grant from the Simons Foundation (1030291, 1290607, I.V.K. and S.I.M.)}

\begin{document}

\begin{abstract}Given a compact surface $M$, consider the natural right action of the group of diffeomorphisms $\mathcal{D}(M)$ of $M$ on $\mathcal{C}^{\infty}(M,\mathbb{R})$ defined by the rule: $(f,h)\mapsto f\circ h$ for $f\in \mathcal{C}^{\infty}(M,\mathbb{R})$ and $h\in\mathcal{D}(M)$. 
Denote by $\mathcal{F}(M)$ the subset of $\mathcal{C}^{\infty}(M,\mathbb{R})$ consisting of function $\func:M\to\mathbb{R}$ taking constant values on connected components of $\partial{M}$, having no critical points on $\partial{M}$, and such that at each of its critical points $z$ the function $f$ is $\mathcal{C}^{\infty}$ equivalent to some homogenenous polynomial without multiple factors.
In particular, $\mathcal{F}(M)$ contains all Morse maps.
Let also 
$\mathcal{O}(f) = \{ f\circ h \mid h\in\mathcal{D}(M) \}$ be the orbit of $f$.
Previously it was computed the algebraic structure of $\pi_1\mathcal{O}(f)$ for all $f\in\mathcal{F}(M)$, where $M$ is any orientable compact surface distinct from $2$-sphere.

\quad 
In the present paper we compute the group $\pi_0\mathcal{S}(f,\partial\mathbb{M})$, where $\mathbb{M}$ is a M\"obius band, and $\mathcal{S}(f,\partial\mathbb{M}) = \{ h\in\mathcal{D}(\mathbb{M}) \mid f\circ h = f, \ h|_{\partial \mathbb{M}} = \mathrm{id}_{\mathbb{M}}\}$ is the subgroup of the corresponding stabilizer of $f$ consisting of diffeomorphisms fixed on the boundary $\partial \mathbb{M}$.
As a consequence we obtain an explicit algebraic description of $\pi_1\mathcal{O}(f)$ for all non-orientable surfaces distinct from Klein bottle and projective plane.
\end{abstract}
\maketitle

% !TeX encoding = UTF-8
% !TeX spellcheck = en_US

\section{Introduction}\label{sect:intro}
The present paper continues a series of works by many authors~\cite{
    Maksymenko:AGAG:2006,
    Maksymenko:TrMath:2008,
    Maksymenko:UMZ:ENG:2012,
    MaksymenkoFeshchenko:UMZ:ENG:2014,
    MaksymenkoFeshchenko:MS:2015,
    MaksymenkoFeshchenko:MFAT:2015,
    Maksymenko:TA:2020,
    Feshchenko:Zb:2015,
    Kudryavtseva:ConComp:VMU:2012, Kudryavtseva:SpecMF:VMU:2012, Kudryavtseva:MathNotes:2012, Kudryavtseva:MatSb:2013}
and others devoted to the study of the natural right action
\[
    \Ci{\Mman}{\Pman} \times \Diff(\Mman) \to  \Ci{\Mman}{\Pman},
    \qquad
    (\func,\dif) \mapsto \func\circ\dif,
\]
of the diffeomorphism group $\Diff(\Mman)$ of a compact surface $\Mman$ on the space $\Ci{\Mman}{\Pman}$ of smooth maps $\func:\Mman\to\Pman$, where $\Pman$ is either the real line $\bR$ or the circle $\Circle$.

For a closed subset $\Xman$ denote by $\Diff(\Mman,\Xman)$ the subgroup of $\Diff(\Mman)$ consisting of diffeomorphisms fixed on $\Xman$, and for every $\func\in\Ci{\Mman}{\Pman}$ let
\begin{align*}
    &\Stabilizer{\func,\Xman} = \{ \dif\in\Diff(\Mman,\Xman) \mid \func\circ\dif = \func \},    &
    &\Orbit{\func,\Xman} = \{ \func\circ\dif \mid  \dif\in\Diff(\Mman,\Xman) \},
\end{align*}
be respectively the \term{stabilizer} and the \term{orbit} of $\func\in\Ci{\Mman}{\Pman}$ with respect to that action. If $\Xman=\varnothing$, then we will omit it from notation.
It will be convenient to say that the diffeomorphisms from $\Stabilizer{\func}$ \term{preserve $\func$}.

Endow the spaces $\Diff(\Mman,\Xman)$ and $\Ci{\Mman}{\Pman}$ with Whitney $C^\infty$-topologies, and their subspaces $\mathcal{S}(f,\Xman)$, $\mathcal{O}(f,\Xman)$ with the induced ones.
Let also
\begin{itemize}
\item $\DiffId(\Mman,\Xman)$ be the identity path component of $\Diff(\Mman,\Xman)$ consisting of diffeomorphisms isotopic to $\id_{\Mman}$ rel.\ $\Xman$,
\item $\StabilizerId{\func,\Xman}$ be the identity path component of the stabilizer $\Stabilizer{\func,\Xman}$, and
\item $\OrbitComp{\func}{\func,\Xman}$ be the path component of the orbit $\Orbit{\func,\Xman}$ containing $\func$.
\end{itemize}
Finally, denote by
\[
    \StabilizerIsotId{\func,\Xman} := \Stabilizer{\func} \cap \DiffId(\Mman,\Xman)
\]
the subgroup of $\Stabilizer{\func,\Xman}$ consisting of $\func$-preserving diffeomorphisms isotopic to $\id_{\Mman}$ rel.\ $\Xman$, but such isotopies are not required to preserve $\func$.
If $\Xman=\varnothing$, then we will omit it from notation.

Note that if $(\Mman,\Xman)$ is either a $(D^2,\partial D^2)$ or $(\Circle\times[0;1], \Circle\times0)$ or $(\MBand,\partial\MBand)$, where $\MBand$ is a M\"obius band, then $\Diff(\Mman,\Xman)$ is well known to be contractible (and in particular path connected), whence $\Diff(\Mman,\Xman)=\DiffId(\Mman,\Xman)$ and therefore in those cases for every $\func\in\Ci{\Mman}{\Pman}$ we have that
\begin{equation}\label{equ:STf_STprf_disk_cyl_mb}
\StabilizerIsotId{\func,\Xman} := 
    \Stabilizer{\func} \cap \DiffId(\Mman,\Xman) = 
    \Stabilizer{\func} \cap \Diff(\Mman,\Xman)   = 
    \Stabilizer{\func,\Xman}.
\end{equation}
We will study the homotopy types of stabilizers and orbits for maps $\func:\Mman\to\Pman$ belonging to the following subset $\FSp{\Mman}\subset\Ci{\Mman}{\Pman}$.
\begin{definition}\label{def:classF}
Denote by $\Cid{\Mman}{\Pman}$ the subset of $\Ci{\Mman}{\Pman}$ consisting of maps $\func$ such that
\begin{enumerate}[label={\rm(B)}, leftmargin=*]
\item\label{axiom:Bd}
$\func$ takes constant values on connected components of $\partial\Mman$ and has no critical points therein.
\end{enumerate}
Furthermore, let $\FSp{\Mman}$ be the subset of $\Cid{\Mman}{\Pman}$ consisting of maps $\func$ such that
\begin{enumerate}[label={\rm(H)}, leftmargin=*]
\item\label{axiom:Hm}
for every critical point $z$ of $\func$ the germ of $\func$ at $z$ is $\Cinfty$ equivalent to some non-zero homogeneous polynomial $\gfunc_z\colon \bR^2 \to \bR$  without multiple factors.
\end{enumerate}
\end{definition}
The space $\FSp{\Mman}$ is rather ``massive'' and (due to Morse lemma and axiom~\ref{axiom:Hm}) contains all maps with non-degenerate critical points.
Moreover, axiom~\ref{axiom:Hm} (namely, absence of multiple factors) implies that for each $\func\in\FSp{\Mman}$ its critical points are isolated, and this further implies that $\FSp{\Mman}$ also contains ``up to a homeomorphism'' every $\func\in\Cid{\Mman}{\Pman}$ with isolated critical points, see Remark~\ref{rem:isol_sing}.

Our main result explicitly computes the fundamental group $\pi_1\Orbit{\func}$ for every $\func\in\FSp{\MBand}$ on the M\"obius band $\MBand$, (Theorem~\ref{th:class:Mobius_band}).
As a consequence this allows to \term{explicitly} describe the homotopy types of path components $\OrbitComp{\func}{\func}$ of  orbits for \term{all} maps $\func\in\FSp{\Mman}$ on \term{all} non-orientable surfaces $\Mman$ distinct from Klein bottle and projective plane, (Theorem~\ref{th:class:nonoriented_case}).
For orientable surfaces distinct from $2$-torus and $2$-sphere such a description was done in S.~Maksymenko~\cite{Maksymenko:TA:2020}, while the case of $2$-torus was studied in a series of papers by S.~Maksymenko and B.~Feshchenko~\cite{MaksymenkoFeshchenko:UMZ:ENG:2014, MaksymenkoFeshchenko:MS:2015, MaksymenkoFeshchenko:MFAT:2015, Feshchenko:Zb:2015, Feshchenko:PIGC:2019, Feshchenko:PIGC:2021}.

\subsection{Homotopy type of orbits}
For $k\geq0$ let $T^k = (\Circle)^k$ be the \term{$k$-torus}, i.e.\ the product of $k$ circles.
Say that a map $\func\in\FSp{\Mman}$ is \term{generic Morse}, if all critical points of $\func$ are non-degenerate, and $\func$ takes distinct values at distinct critical points.
By~\cite[Theorem~1.5]{Maksymenko:AGAG:2006} for every \term{generic} Morse map $\func:\Mman\to\Pman$ (except for few types of such maps) we have the following homotopy equivalences
\[
    \OrbitComp{\func}{\func} \simeq
    \begin{cases}
        T^k\times SO(3), & \text{if $\Mman=S^2$ or $\PrjPlane$}, \\
        T^k,             & \text{otherwise.}
    \end{cases}
\]
Moreover, the number $k$ depends on $\func$ and is expressed via the numbers of critical points of $\func$ of distinct indices.
See also~\cite[Theorem~5.3.1]{Maksymenko:ProcIMNANU:2020} for an explicit construction of the inclusion $T^k\subset\OrbitComp{\func}{\func}$ being a homotopy equivalence in the second case.

The above result was extended by E.~Kudryavtseva~\cite{Kudryavtseva:ConComp:VMU:2012, Kudryavtseva:SpecMF:VMU:2012, Kudryavtseva:MathNotes:2012, Kudryavtseva:MatSb:2013} to arbitrary Morse functions $\func:\Mman\to\bR$ on orientable surfaces, and in~\cite{Kudryavtseva:DAN:2016} to functions with singularities of type $A_{\mu}$.
She proved that for every such $\func$ (again except for several specific cases) there exists a free action of a certain finite group $G_{\func}$ on a $k$-torus $T^k$ for some $k\geq0$ such that we have the following homotopy equivalences:
\[
    \OrbitComp{\func}{\func} \simeq
    \begin{cases}
        (T^k/G_{\func})\times SO(3), & \text{if $\Mman=S^2$ or $\PrjPlane$}, \\
        T^k/G_{\func},             & \text{otherwise.}
    \end{cases}
\]

That group $G_{\func}$ is exactly the group $\GrpKRIsotId{\func,\varnothing}:=\StabilizerIsotId{\func,\varnothing}/\FolStabilizerIsotId{\func,\varnothing}$ defined below by~\eqref{equ:Delta_G}.
For Morse maps it can also be regarded as the group of automorphisms of the Kronrod-Reeb graph of $\func$ induced by diffeomorphisms from $\StabilizerIsotId{\func}$.
It is easy to show that $G_{\func}$ is trivial for generic Morse map.
However, the algebraic structure of $\pi_1\OrbitComp{\func}{\func}$ and $G_{\func}$ was not clear.
S.~Maksymenko~\cite{Maksymenko:TA:2020} explicitly computed $\pi_1\OrbitComp{\func}{\func}$ and $G_{\func}$ for all $\func\in\FSp{\Mman}$, where $\Mman$ is an orientable surface distinct from $2$-torus and $2$-sphere,
S.~Maksymenko and B.~Feshchenko~\cite{MaksymenkoFeshchenko:UMZ:ENG:2014, MaksymenkoFeshchenko:MS:2015, MaksymenkoFeshchenko:MFAT:2015, Feshchenko:Zb:2015, Feshchenko:PIGC:2019, Feshchenko:PIGC:2021} completely described the case when $\Mman$ is a $2$-torus, and also in~\cite{Maksymenko:WreathProd:2022} explicit models for the actions of $G_{\func}$ on $T^k$ are given.
In the present paper those computations will be extended to all non-orientable surfaces distinct from the Klein bottle and projective plane.

Our starting point is the following statement collecting several particular results from the mentioned above papers.
\begin{theorem}\label{th:prelim_statements}
Let $\Mman$ be a connected compact surface, $\Yman$ be a possibly empty collection of boundary components of $\Mman$, and $\func\in\FSp{\Mman}$.
\begin{enumerate}[wide, label={\rm(\alph*)}]
\item\label{enum:th:prelim_statements:Of_OfY}
Then $\OrbitComp{\func}{\func,\Yman} = \OrbitComp{\func}{\func}$, i.e.\ the path components of $\func$ in those orbits coincide as sets, though it is not true in general that $\Orbit{\func,\Yman} = \Orbit{\func}$.

\item
Suppose that either $\Yman\not=\varnothing$, or $\Yman=\varnothing$ but the Euler characteristic $\chi(\Mman) < 0$.
Then
\begin{enumerate}[wide, label={\rm(b\arabic*)}]
\item\label{enum:th:prelim_statements:Orbit_aspherical}
$\DiffId(\Mman,\partial\Yman)$ is contractible, $\OrbitComp{\func}{\func,\Yman}$ is \term{aspherical}, and we have an \term{isomorphism}
\begin{equation}\label{equ:j_piOf_pi0Stprf}
\partial:\pi_1\OrbitComp{\func}{\func,\Yman} \to \pi_0\StabilizerIsotId{\func,\Yman};
\end{equation}

\item\label{enum:th:prelim_statements:Sprf_product}
there exists a compact subsurface $\Xman\subset\Mman$ such that $\overline{\Mman\setminus\Xman}$ is a disjoint union of compact subsurfaces $\Lman_1,\ldots,\Lman_{p}$, where each $\Lman_{i}$ is either a $2$-disk or a cylinder or a M\"obius band, $\func|_{\Lman_{i}} \in \FSp{\Lman_{i}}$, and the inclusions
\begin{equation}\label{equ:stab_inclusions}
    \xymatrix@C=7ex{
        \ \StabilizerIsotId{\func,\Yman} \ &
        \ \StabilizerIsotId{\func,\Yman\cup\Xman} \ \ar@{^(->}[r]^-{\alpha} \ar@{_(->}[l] &
        \ \prod\limits_{i=1}^{p}\StabilizerIsotId{\func|_{\Lman_{i}},\partial\Lman_{i}} \
    }
\end{equation}
are homotopy equivalences, where $\alpha(\dif) = (\dif|_{\Lman_{1}}, \ldots, \dif|_{\Lman_{p}})$.
\end{enumerate}
\end{enumerate}

In particular, we get isomorphisms:
\begin{equation}\label{equ:expt_pi1Of__pi0Sprf}
    \pi_1\OrbitComp{\func}{\func}       \ \cong \
    \pi_1\OrbitComp{\func}{\func,\Yman} \ \cong \
    \pi_0\StabilizerIsotId{\func,\Yman} \ \cong \
    \prod\limits_{i=1}^{p}\pi_0\StabilizerIsotId{\func|_{\Lman_i},\partial\Lman_i}.
\end{equation}
\end{theorem}
\begin{proof}[Remarks to the proof]
Statement~\ref{enum:th:prelim_statements:Of_OfY} proved in~\cite[Corollary~2.1]{Maksymenko:UMZ:ENG:2012}, while statement~\ref{enum:th:prelim_statements:Sprf_product} is established in~\cite[Theorem~1.7]{Maksymenko:MFAT:2010} (in fact, the inclusions~\eqref{equ:stab_inclusions} even yield isomorphisms of the corresponding \term{Bieberbach sequences}, see~\eqref{equ:iso_bbb} below for details).

Contractibility of $\DiffId(\Mman,\Yman)$ with $\Yman\not=\varnothing$ is well-known, e.g.~\cite{EarleEells:JGD:1969,EarleSchatz:DG:1970, Gramain:ASENS:1973}.
Further, by statement~\ref{enum:lm:maps_of_fin_codim:DIdY} of Theorem~\ref{th:maps_of_fin_codim} below, the map $p:\DiffId(\Mman,\Yman)\to\OrbitComp{\func}{\func,\Yman}$, $p(\dif)=\func\circ\dif$, is a locally trivial fibration with fiber $\StabilizerIsotId{\func,\Yman}$.
Hence, we get the following part of the short exact sequence of that fibration:
\[
   \pi_1\DiffId(\Mman,\Yman) \to
   \pi_1\OrbitComp{\func}{\func,\Yman} \xrightarrow{~\partial~}
   \pi_0 \StabilizerIsotId{\func,\Yman} \to
   \pi_0\DiffId(\Mman,\Yman).
\]
Since $\DiffId(\Mman,\Yman)$ is contractible, the first and last terms are zero, which implies that $\partial$ is an isomorphism.
\end{proof}

Thus, by~\ref{enum:th:prelim_statements:Orbit_aspherical}, the group $\pi_0\StabilizerIsotId{\func,\Yman}$ completely determines the homotopy type of $\OrbitComp{\func}{\func}$.
Moreover, due to~\ref{enum:th:prelim_statements:Sprf_product}, the computation of $\pi_0\StabilizerIsotId{\func,\Yman}$ reduces to the computation of groups $\pi_0\StabilizerIsotId{\gfunc,\partial\Lman}$ for functions $\gfunc\in\FSp{\Lman}$, where $\Lman$ is either a $2$-disk or a cylinder or a M\"obius band.

The structure of such groups for disks and cylinders is described in~\cite{Maksymenko:TA:2020}.
Moreover, it is shown in~\cite[Theorem~1.4]{MaksymenkoKuznietsova:PIGC:2019}, see Lemma~\ref{lm:unique_cr_level} below, that every map $\func\in\FSp{\MBand}$ allows to additionally decompose the M\"obius band $\MBand$ into one cylinder $\Yman_0$, several $2$-disks $\Yman_1,\ldots,\Yman_k$, and a certain subsurface $\Xman$.
Our main result (Theorem~\ref{th:class:Mobius_band}) shows that $\pi_1\Orbit{\func}$ can be expressed via the groups $\pi_1\Orbit{\func|_{\Yman_i}}$ in a certain explicit algebraic way.

\section{Main result}
In what follows $\UnitGroup$ denotes the unit group, and the symbols $\monoArrow$ and $\epiArrow$ denote respectively \term{monomorphism} and \term{epimorphism}.
In particular, $\Agrp\monoArrow\Bgrp\epiArrow\Cgrp$ is a short exact sequence of groups.

First we will need to introduce several types of semidirect products corresponding to certain non-effective actions of the group $\bZ$.
All of them are particular cases of the following construction.

Let $\Ggrp$ be a group with unit $e$ and $\ahom\in\Aut(\Ggrp)$ be any automorphism of $\Ggrp$.
Then one can define the homomorphism $\widehat{\ahom}:\bZ\to\Aut(\Ggrp)$, $\widehat{\ahom}(t) = \ahom^{t}$, $t\in\bZ$.
Hence, we also have the corresponding semidirect product $\smd{\Ggrp}{\ahom}{\bZ}$ which, by definition, is the Cartesian product $\Ggrp\times\bZ$ of sets with the following multiplication:
\begin{equation}\label{equ:G_rtimes_Z:multiplication}
    (\ael,k)(\bel,l) = (\ael\, \ahom^{k}(\bel), k+l)
\end{equation}
for all $\ael,\bel\in\Ggrp$ and $k,l\in\bZ$.

Furthermore, if $\ahom$ is periodic of some order $\mel\geq1$, then $\widehat{\ahom}$ reduces to a monomorphism $\overline{\ahom}:\bZ_{\mel} \to \Aut(\Ggrp)$, $\overline{\ahom}(t) = \ahom^{t}$, $t\in\bZ_{\mel}$, and we also have the semidirect product $\smd{\Ggrp}{\ahom}{\bZ_{\mel}}$, which again by definition, is a Cartesian product $\Ggrp\times\bZ_{\mel}$ with the multiplication given by same formula~\eqref{equ:G_rtimes_Z:multiplication}, but now $k,l$ should be taken modulo $\mel$.

Evidently, there is the following short exact sequence
\[
    \{e\} \times \mel\bZ     \xmonoArrow{~~~~~~~~~~}
    \smd{\Ggrp}{\ahom}{\bZ}  \xepiArrow{~(\gel,\,k) \,\mapsto \,(\gel, \, k\bmod\mel)~}
    \smd{\Ggrp}{\ahom}{\bZ_{\mel}}.
\]

\subsection{Groups $\WrGZ{\Ggrp}{\mel}$ and $\WrGZm{\Ggrp}{\mel}$}\label{sect:grp:GwrmZ}
Given $\mel\in\bN$ define the following automorphism cyclically shifting coordinates:  
\[ 
    \ahom:\Ggrp^{\mel} \to \Ggrp^{\mel}, 
    \qquad
    \ahom(\aeli{0},\ldots,\aeli{\mel-1}) = (\aeli{1}, \aeli{2},\ldots,\aeli{\mel-1}, \aeli{0})
\]
for all $(\aeli0,\ldots,\aeli{\mel-1}) \in \Ggrp^{\mel}$.
Let also
\begin{align*}
    \WrGZ{\Ggrp}{\mel}  &:= \smd{\Ggrp^{\mel}}{\ahom}{\bZ}, &
    \WrGZm{\Ggrp}{\mel} &:= \smd{\Ggrp^{\mel}}{\ahom}{\bZ_{\mel}}
\end{align*}
to be the corresponding semidirect products.
Thus, $\WrGZ{\Ggrp}{\mel}$ is a Cartesian product $\Ggrp^{\mel} \times \bZ$ with the following multiplication:
\[
    (\aeli0,\ldots,\aeli{\mel-1}; k) \cdot
    (\beli{0},\ldots,\beli{\mel-1}; l) =
    (\aeli0 \beli{k}, \aeli1 \beli{k+1}, \ldots, \aeli{\mel-1} \beli{k-1}; k+l),
\]
where the indices are taken modulo $\mel$, $\aeli{i},\beli{i}\in\Ggrp$, $k,l\in\bZ$.
On the other hand, $\WrGZm{\Ggrp}{\mel}$ is the set $\Ggrp^{\mel} \times \bZ_{\mel}$ but with the multiplication given by the same formulas, but now $k,l\in\bZ_{\mel}$.

Evidently, $\WrGZm{\Ggrp}{\mel}$ is the standard wreath product of groups $\Ggrp$ and $\bZ_{\mel}$, while $\WrGZ{\Ggrp}{\mel}$ is also the wreath product of $\Ggrp$ and $\bZ$ with respect to the \term{non-effective} action of $\bZ$ on the set $\{0,1,\ldots,\mel-1\}$ by cyclic shifts to the left.

\subsection{Groups $\WrGZZ{\Ggrp}{\mel}{\nel}$ and $\WrGZZmn{\Ggrp}{\mel}{\nel}$}\label{sect:grp:GwrmnZ}
More generally, let $\mel,\nel\in\bN$ be two numbers.
Then the elements of the $\mel\nel$-power $\Ggrp^{\mel\nel}$ of $\Ggrp$ can be regarded as $(\mel\times\nel)$-matrices $\{\gel_{i,j}\}_{i=0,\ldots\mel-1,\, j=0,\ldots,\nel-1}$ whose entries are elements of $\Ggrp$.
Then there is a natural non-effective action of $\bZ^2$ on $\Ggrp^{\mel\nel}$ by cyclic shifts of rows and columns, i.e.
\[
    (k,l) \cdot \{\gel_{i,j}\} = \{\gel_{i+k \bmod \mel, \, j+l \bmod \nel}\}.
\]
This action reduces to an effective action of $\bZ_{\mel}\times\bZ_{\nel}$.
Let
\begin{align*}
    &\WrGZZ{\Ggrp}{\mel}{\nel}   := \smd{\Ggrp^{\mel\nel}}{}{\bZ^2},
    &\WrGZZmn{\Ggrp}{\mel}{\nel} := \smd{\Ggrp^{\mel\nel}}{}{(\bZ_{\mel}\times\bZ_{\nel})}
\end{align*}
be the corresponding semidirect products.
Then $\WrGZZmn{\Ggrp}{\mel}{\nel} $ is the standard wreath product of groups $\Ggrp$ and $\bZ_{\mel}\times\bZ_{\nel}$, while $\WrGZZ{\Ggrp}{\mel}{\nel}$ is also the wreath product of $\Ggrp$ and $\bZ^2$ with respect to the action of $\bZ^2$ on the set $\{0,1,\ldots,\mel-1\}\times\{0,1,\ldots,\nel-1\}$ by independent cyclic shifts of coordinates.

\subsection{Groups $\WrGHZ{\Ggrp}{\Hgrp}{\gamma}{\mel}$ and $\WrGHZm{\Ggrp}{\Hgrp}{\gamma}{\mel}$}\label{sect:grp:GHwr_gamma_m_Z}
Let $\Hgrp$ be another group, $\mel\in\bN$, and $\gamma\colon \Hgrp\to \Hgrp$ be an automorphism such that $\gamma^{2}=\id_{\Hgrp}$, so $\gamma$ is either the identity automorphism or has order $2$.
Define the automorphism $\bhom\colon \Ggrp^{2\mel}\times\Hgrp^{\mel}\to \Ggrp^{2\mel}\times\Hgrp^{\mel}$ by the following formula
\begin{equation}\label{equ:Aut_in_WrGH_gamma_m}
    \bhom(\aeli0,\ldots,\aeli{2\mel-1}; \, \beli{0},\ldots, \beli{\mel-1})
    =
    (\aeli{1},\ldots, \aeli{2\mel-1}, \aeli{0}; \,  \beli{1},\ldots, \beli{\mel-1}, \gamma(\beli{0})).
\end{equation}
Evidently, $\bhom$ is of order $2\mel$.
Let
\begin{align*}
    &\WrGHZ{\Ggrp}{\Hgrp}{\gamma}{\mel} := \smd{(\Ggrp^{2\mel}\times\Hgrp^{\mel})}{\bhom}{\bZ}, \\
    &\WrGHZm{\Ggrp}{\Hgrp}{\gamma}{\mel} := \smd{(\Ggrp^{2\mel}\times\Hgrp^{\mel})}{\bhom}{\bZ_{2\mel}},
\end{align*}
be the corresponding semidirect products associated with $\bhom$.
Thus, $\WrGHZ{\Ggrp}{\Hgrp}{\gamma}{\mel}$ is the Cartesian product $\Ggrp^{2\mel}\times\Hgrp^{\mel}\times\bZ$ with the following operation.
Let $v = (\aeli{0},\ldots,\aeli{2\mel-1}; \,  \beli{0},\ldots,\beli{\mel-1};  k)$ and
$w= (\celi{0},\ldots,\celi{2\mel-1}; \,  \deli{0},\ldots,\deli{\mel-1};  l) \in \Ggrp^{2\mel}\times\Hgrp^{\mel}\times\bZ$.
Denote $k' = k\bmod 2\mel$.
Then
{\small\begin{equation}\label{equ:mult_in_GHZ}
v\,w \!=\!
\begin{cases}
    \bigl(
        \aeli{0}\celi{k}, \aeli{1}\celi{k+1}, ..., \aeli{2\mel-1}\celi{k-1}; \\
    \quad \beli{0}\deli{k}, \beli{1}\deli{k+1},   ..., \beli{\mel-k-1}\deli{\mel-1},
        \beli{\mel-k}\gamma(\deli{0}),        ..., \beli{\mel-1}\gamma(\deli{k-1}); 
        k\!+\!l
    \bigr), \\
    \hfill  0\leq k' < \mel, \\[1mm]
%%%%%%%%%%%%%
    \bigl(
        \aeli{0}\celi{k}, \aeli{1}\celi{k+1}, ..., \aeli{2\mel-1}\celi{k-1}; \\
    \quad \beli{0}\gamma(\deli{k}), \beli{1}\gamma(\deli{k+1}), ..., \beli{\mel-k-1}\gamma(\deli{\mel-1}),
        \beli{\mel-k} \deli{0},        ..., \beli{\mel-1} \deli{k-1}; 
        k\!+\!l
    \bigr), \\
    \hfill \mel\leq k' < 2\mel,
\end{cases}
\end{equation}}%
where the indices of $\aeli{*}$ and $\celi{*}$ are taken modulo $2\mel$, while the indices of $\beli{*}$ and $\deli{*}$ are taken modulo $\mel$.

Again the multiplication in $\WrGHZm{\Ggrp}{\Hgrp}{\gamma}{\mel}$ is given by the same formulas~\eqref{equ:mult_in_GHZ}, but one should take $k,l\in\bZ_{2\mel}$.

Evidently, for every $\mel\in\bN$ we have the following isomorphisms:
\begin{align*}
     \WrGZ{\UnitGroup}{1} &\cong \bZ, &
     \WrGZ{\Ggrp}{1}      &\cong \Ggrp\times\bZ, &
      \WrGZZ{\Ggrp}{1}{1} &\cong \Ggrp\times\bZ^2, \\
     \WrGZ{\Ggrp}{2\mel}  &\cong \WrGHZ{\Ggrp}{\UnitGroup}{\id_{\UnitGroup}}{\mel}, &
     \WrGZZ{\Ggrp}{\mel}{1} & \cong (\WrGZ{\Ggrp}{\mel}) \times\bZ.
\end{align*}

\begin{definition}\label{def:classG}
Let $\classGroups$ be the minimal class of groups (considered up to isomorphism) satisfying the following conditions:
\begin{enumerate}[label={\rm\arabic*)}]
\item\label{enum:def:classG:1} $\UnitGroup\in \classGroups$;
\item\label{enum:def:classG:x} if $\Ggrp,\Hgrp \in \classGroups$, then $\Ggrp\times\Hgrp \in \classGroups$;
\item\label{enum:def:classG:wr} if $\Ggrp\in \classGroups$ and $\mel\geq 1$, then $\WrGZ{\Ggrp}{\mel}\in\classGroups$.
\end{enumerate}
\end{definition}
It is easy to see, \cite[Lemma~2.6]{Maksymenko:TA:2020}, that $\classGroups$ consists of groups $\Ggrp$ which can be obtained from the unit group $\UnitGroup$ by finitely many operations of direct products $\times$ and wreath products $\WrGZ{\cdot}{\bullet}$.
For example, the following groups belong to $\classGroups$:
\begin{align*}
  &\bZ = \WrGZ{\UnitGroup}{1}, &
  &\bZ^n, (n\geq1), &
  &\WrGZ{ \bigl((\WrGZ{\bZ^3}{5}) \times (\WrGZ{\bZ^{17}}{2})\bigr) }{11}.
\end{align*}
Notice that such an expression of a group $\Ggrp\in\classGroups$ is not unique, e.g.\ $\WrGZ{\UnitGroup}{1} \cong \WrGZ{(\UnitGroup\times\UnitGroup)}{1}$.

The following theorem collects some known information about the structure of $\pi_1\Orbit{\func}$ for maps $\func\in\FSp{\Mman}$ on orientable surfaces $\Mman$ distinct from $S^2$.

\begin{theorem}[{\rm\cite{Maksymenko:TA:2020, KuznietsovaSoroka:UMJ:2021, Feshchenko:PIGC:2019}}]
\label{th:class:oriented_case}
Let $\Mman$ be a compact orientable surface distinct from $2$-sphere and $\func\in\FSp{\Mman}$.
\begin{enumerate}[leftmargin=*, topsep=1ex, itemsep=1ex]
\item\label{enum:th:class:oriented_case:not_torus}
If $\Mman$ is also distinct from $2$-torus, then there exists $\Ggrp\in\classGroups$ such that
\begin{equation}\label{equ:pi1Of_notT2}
    \pi_1\Orbit{\func} \cong \Ggrp.
\end{equation}

\item\label{enum:th:class:oriented_case:torus}
Suppose $\Mman=T^2$ is a $2$-torus.
If the Kronrod-Reeb graph of $\func$ is a tree then there exist $\Ggrp\in\classGroups$ and $\mel,\nel\in\bN$ such that
\begin{equation}\label{equ:pi1Of_T2_tree}
    \pi_1\Orbit{\func} \cong \WrGZZ{\Ggrp}{\mel}{\nel}.
\end{equation}
Otherwise, the Kronrod-Reeb graph of $\func$ contains a unique cycle and there exist $\Ggrp\in\classGroups$ and $\mel\in\bN$ such that
\begin{equation}\label{equ:pi1Of_T2_cycle}
    \pi_1\Orbit{\func} \cong \WrGZ{\Ggrp}{\mel}.
\end{equation}
The latter holds e.g.\ if $\Pman=\Circle$ and $\func:T^2\to\Circle$ is not null homotopic.
\end{enumerate}
Conversely, for every group as in the r.h.s.\ of~\eqref{equ:pi1Of_notT2}-\eqref{equ:pi1Of_T2_cycle} there exists $\func\in\FSp{\Mman}$ on the corresponding surface $\Mman$ such that the corresponding isomorphism holds.
Moreover, one can also assume that $\func$ is Morse.
\end{theorem}
\begin{proof}[Remarks to the proof]
It is shown in~\cite[Theorem~5.10]{Maksymenko:TA:2020} that if $\Mman\not=T^2$ and $S^2$, then for each $\func\in\FSp{\Mman}$ it is possible to explicitly express $\pi_1\Orbit{\func}=\pi_0\Stabilizer{\func,\partial\Mman}$ in terms of operations $\times$ and $\WrGZ{\cdot}{\mel}$.
Conversely, it is proved in~\cite[Theorem~1.2]{KuznietsovaSoroka:UMJ:2021} that for any finite combination of such operations giving rise a group $\Ggrp$ one can construct some $\func\in\FSp{\Mman}$ with $\Ggrp \cong \pi_1\Orbit{\func}$.
In addition, if $\Yman_1,\ldots,\Yman_k$ is some collection of boundary components of $\Mman$, then (except for few cases) one can even assume that $\func$ takes a local minimum of each $\Yman_i$, and a local maximum on each boundary components from $\partial\Mman\setminus\cup_{i=1}^{k}\Yman_i$.

Similar considerations were done for the case of torus in~\cite{MaksymenkoFeshchenko:UMZ:ENG:2014, MaksymenkoFeshchenko:MS:2015, MaksymenkoFeshchenko:MFAT:2015, Feshchenko:Zb:2015, Feshchenko:PIGC:2019, Feshchenko:PIGC:2021}.
See especially~\cite[Theorem~3.2]{Feshchenko:PIGC:2019} for the general statement on the structure of $\pi_1\Orbit{\func}$ and $\pi_0\StabilizerIsotId{\func}$ and several other groups related with $\func\in\FSp{T^2}$.
\end{proof}

Our main result is the following theorem which will be proved in Section~\ref{sect:proof:th:class:Mobius_band}.
\begin{theorem}\label{th:class:Mobius_band}
Let $\MBand$ be the M\"obius band.
Then for every $\func\in\FSp{\MBand}$ either
\begin{itemize}[label=$- $, leftmargin=*]
\item there exist groups $\Agrp, \Ggrp, \Hgrp \in \classGroups$, an automorphism $\gamma:\Hgrp\to\Hgrp$ with $\gamma^{2}=\id_{\Hgrp}$, and $\mel\geq1$ such that
\begin{equation}\label{equ:pi0Sprf_A_GHZ__MBand}
    \pi_1\Orbit{\func} \ \cong \ \Agrp \ \times  \ \WrGHZ{\Ggrp}{\Hgrp}{\gamma}{\mel},
\end{equation}
\item
or there exist groups $\Agrp, \Ggrp \in \classGroups$ and an \term{odd} number $\sgOrd\geq1$ such that
\begin{equation}\label{equ:pi0Sprf_A_GZ__MBand_1}
    \pi_1\Orbit{\func} \ \cong \ \Agrp \ \times  \ \WrGZ{\Ggrp}{\sgOrd}.
\end{equation}
\end{itemize}
In particular, in the second case $\pi_1\Orbit{\func}\in\classGroups$.

Conversely, for every such tuple $(\Agrp, \Ggrp,\sgOrd)$ or $(\Agrp, \Ggrp,\Hgrp, \gamma, \mel)$ there exists $\func\in\FSp{\MBand}$ such that we have the corresponding isomorphism~\eqref{equ:pi0Sprf_A_GHZ__MBand} or~\eqref{equ:pi0Sprf_A_GZ__MBand_1}, and one can assume that $\func$ is Morse.
\end{theorem}

As a consequence we get the following statement:
\begin{theorem}\label{th:class:nonoriented_case}
Let $\Mman$ be a non-orientable compact surface (possibly with boundary) of non-orientable genus $g\geq2$, i.e.\ distinct from Klein bottle and projective plane $\PrjPlane$.
Then for every $\func\in\FSp{\Mman}$ there exist $k\leq g$, groups $\Agrp, \Ggrp_1, \Hgrp_1, \ldots, \Ggrp_k, \Hgrp_k \in \classGroups$, and for each $j=1,\ldots,k$ an automorphism $\gamma_j:\Hgrp_j\to\Hgrp_j$ with $\gamma_j^{2}=\id_{\Hgrp_j}$ and $\mel_j\in\bN$ such that
\begin{equation}\label{equ:pi0Sprf_A_GHZ_non_orientable}
    \pi_1\Orbit{\func} \ \cong \ \Agrp  \times   \prod_{j=1}^{k}\WrGHZ{\Ggrp_j}{\Hgrp_j}{\gamma_j}{\mel_j}.
\end{equation}
This formally includes the case $k=0$, when $\pi_1\Orbit{\func} \cong \Agrp \in \classGroups$.
\end{theorem}
\begin{proof}
Let $\func\in\FSp{\Mman}$.
Then by Theorem~\ref{th:prelim_statements}\ref{enum:th:prelim_statements:Sprf_product}, $\pi_1\Orbit{\func}\cong \prod\limits_{j=1}^{p}\pi_1\Orbit{\func_j}$, where each $\Lman_j$ is either a $2$-disk or a cylinder or a M\"obius band and $\func_j = \func|_{\Lman_j} \in\FSp{\Lman_j}$.
One can assume that
\begin{itemize}[label={$- $}, leftmargin=*]
\item
for some $k<p$ we have that $\Lman_1,\ldots,\Lman_{k}$ are M\"obius bands such that $\pi_1\Orbit{\func|_{\Lman_j}}$ is of type~\eqref{equ:pi0Sprf_A_GHZ__MBand},
\item
while all others $\Lman_{k+1},\ldots,\Lman_{p}$ are $2$-disks, cylinders, or M\"obius bands, so $\pi_1\Orbit{\func|_{\Lman_j}}\in\classGroups$ for all these subsurfaces.
\end{itemize}
Recall that the non-orientable genus of $\Mman$ can be defined as the maximal number of embedded and mutually disjoint M\"obius bands in $\Mman$.
Therefore, since $\Lman_j$ are mutually disjoint, we should have that $k\leq g$.

Now by Theorem~\ref{th:class:Mobius_band} and statement~\ref{enum:th:class:oriented_case:not_torus} of Theorem~\ref{th:class:oriented_case} there exist groups $\Agrp_i, \Ggrp_j,\Hgrp_j \in\classGroups$, automorhisms $\gamma_j:\Hgrp_j\to\Hgrp_j$ with $\gamma_j^2=\id_{\Hgrp_j}$ and numbers $\mel_j\in\bN$, where $i=1,\ldots,p$ and $j=1,\ldots,k$, such that
\[
    \pi_1\Orbit{\func_j} \cong
    \begin{cases}
        \Agrp_j \ \times  \ \WrGHZ{\Ggrp_j}{\Hgrp_j}{\gamma_j}{\mel_j}, & j=1,\ldots,k,\\
        \Agrp_j, & j=k+1,\ldots,p.
    \end{cases}
\]
Denote $\Agrp = \myprod_{i=1}^{p}\Agrp_i$.
Since the class $\classGroups$ is closed under direct products, see property~\ref{enum:def:classG:x} of Definition~\ref{def:classG}, $\Agrp\in\classGroups$ as well, and therefore $\pi_1\Orbit{\func} \cong  \Agrp \times \myprod_{j=1}^{k}\WrGHZ{\Ggrp_j}{\Hgrp_j}{\gamma_j}{\mel_j}$, which agrees with~\eqref{equ:pi0Sprf_A_GHZ_non_orientable}.
\end{proof}

\begin{remark}\rm
The question of realization of groups of the form~\eqref{equ:pi0Sprf_A_GHZ_non_orientable} as $\pi_1\OrbitComp{\func}{\func}$ for some $\func\in\FSp{\Mman}$, where $\Mman$ is a non-orientable surface distinct from Klein bottle and projective plane, is more delicate, and will be studied in another paper.
\end{remark}

\subsection{Structure of the paper}
Section~\ref{sect:alg_preliminaries} contains two results characterizing groups $\WrGZ{\Ggrp}{\mel}$ and $\WrGHZ{\Ggrp}{\Hgrp}{\gamma}{\mel}$, see Lemmas~\ref{lm:3x3_G_m:char} and~\ref{lm:3x3_G_H_gamma_m:char}.

In Section~\ref{sect:preliminaries} we present preliminary information about actions of diffeomorphisms on spaces of functions on surfaces.
In particular, we discuss the codimension (Milnor number) of a singularity used in Lemma~\ref{th:prelim_statements}, see Theorem~\ref{th:maps_of_fin_codim}.

Section~\ref{sect:func_on_surf} describes several constructions related with smooth functions on surfaces.

In Section~\ref{sect:decomp_B_in_Yi} we recall the result from~\cite{MaksymenkoKuznietsova:PIGC:2019},
see Lemma~\ref{lm:unique_cr_level}, claiming that every $\func\in\FSp{\MBand}$ yields a certain ``almost CW'' partition of $\MBand$ which is preserved by elements of $\Stabilizer{\func}$.
We also study the action of $\Stabilizer{\func}$ on elements of that partition.

In Section~\ref{sect:proof:th:class:Mobius_band} we start to prove Theorem~\ref{th:class:Mobius_band} reducing it to the case when the group $\Agrp$ is trivial.
We also formulate Theorem~\ref{th:g_eta} and in Section~\ref{sect:decuct:th:class:Mobius_band} deduce from it Theorem~\ref{th:class:Mobius_band}, and, in particular, establish Lemmas~\ref{lm:all_disks_T1}-\ref{lm:there_are_disks_T1_and_T2} giving a more detailed information on the structure of $\pi_1\Orbit{\func}$.

Section~\ref{sect:proof:th:g_eta} contains the proof of Theorem~\ref{th:g_eta} itself.

\section{Algebraic preliminaries}\label{sect:alg_preliminaries}
\subsection{Exact $(3\times3)$-diagrams and epimorphisms onto $\bZ$}
The following lemmas are easy and well known.
We leave them for the reader.

\begin{lemma}\label{lm:3x3-diagram}
Let $\Bgrp$ be a group and $\Agrp, \Lgrp \triangleleft \Bgrp$ two normal subgroups, and $\Kgrp =\Agrp\cap\Lgrp$.
Then we have the following $(3\times3)$-diagram in which all rows and columns are short exact sequences:
\begin{equation}\label{equ:3x3-diagram}
\begin{gathered}
    \xymatrix@C=1em@R=1em{
        \exd
        {\Agrp\cap\Lgrp}
        {\Lgrp}
        {\Lgrp / \Kgrp}
        {\Agrp}
        {\Bgrp}
        {\Bgrp/\Agrp}
        {\Agrp / \Kgrp}
        {\Bgrp/\Lgrp}
        {\Bgrp/(\Agrp\Lgrp)}
}
\end{gathered}
\end{equation}
\end{lemma}
Such diagrams will be called \term{exact $(3\times3)$-diagram}.

\begin{lemma}\label{lm:char_B_rtimes_Z}
Let $\Bgrp$ be a group, $\eta:\Bgrp\to\bZ$ be an epimorphism with kernel $\Lgrp$, and $\gel\in\Bgrp$ be any element such that $\eta(\gel)=1$.
Denote by $\phi:\Lgrp\to\Lgrp$, $\phi(\vel) = \gel^{-1} \vel \gel$, the automorphism of $\Lgrp$ induced by conjugation of $\Bgrp$ with $\gel$.
Then the map $\theta:\smd{\Lgrp}{\phi}{\bZ} \to \Bgrp$, $\theta(\vel,k)= \vel \gel^{k}$, is an isomorphism inducing an isomorphism of the following short exact sequences:
\begin{equation}\label{equ:epi_Z_split}
\begin{gathered}
    \xymatrix@R=1.5em{
    \Lgrp \ar@{^(->}[r]            & \smd{\Lgrp}{\phi}{\bZ} \ar@{->>}[r] \ar[d]_{\theta} & \bZ \\
    \Lgrp \ar@{^(->}[r] \ar@{=}[u] & \Bgrp\ar@{->>}[r]^-{\eta}                           & \bZ \ar@{=}[u]
    }
\end{gathered}
\end{equation}
\end{lemma}

\begin{lemma}\label{lm:smd_iso_Z}
Let $\Lgrp$ and $\Lgrp'$ be two groups, $\phi\in\Aut(\Lgrp)$ and $\phi'\in\Aut(\Lgrp')$ some of their automorphisms, and $\qhom:\Lgrp \to \Lgrp'$ be a homomorphism.
Then the following conditions are equivalent:
\begin{enumerate}[label={\rm(\alph*)}]
\item $\qhom\circ\phi = \phi'\circ\qhom$;
\item the map $\zhom:\smd{\Lgrp}{\phi}{\bZ}\to\smd{\Lgrp'}{\phi'}{\bZ}$ defined by $\zhom(\ael,k) = (\qhom(\ael),k)$ is a homomorphism.
\end{enumerate}
In other words, the following diagrams are mutually commutative or non-commutative:
\begin{align*}
&
\xymatrix{
    \Lgrp  \ar[r]^-{\phi}  \ar[d]_-{\qhom} & \Lgrp \ar[d]^-{\qhom}\\
    \Lgrp' \ar[r]^-{\phi'}                 & \Lgrp'
}
&&
\xymatrix{
    \Lgrp\ar@{^(->}[r] \ar[d]^-{\qhom} &  \smd{\Lgrp}{\phi}{\bZ}   \ar[d]^-{\zhom}\ar[r] & \bZ \ar@{=}[d]\\
    \Lgrp'\ar@{^(->}[r]                &  \smd{\Lgrp'}{\phi'}{\bZ} \ar[r]                & \bZ
    }
\end{align*}
Under above conditions $\qhom$ is an isomorphism iff so is $\zhom$.
\end{lemma}

Also recall that a group $\Lgrp$ is a \term{direct product of its subgroups $\Ggrp_1,\ldots,\Ggrp_n < \Lgrp$} if the map
\[
    \delta: \Ggrp_1 \times\cdots\times\Ggrp_n \to\Lgrp,
    \qquad
    \delta(\gel_1,\ldots,\gel_n) = \gel_1\cdots\gel_n,
\]
is an \term{isomorphism} of groups.
Notice that in this definition some groups are allowed to be unit groups.

\subsection{Characterization of groups $\WrGZ{\Ggrp}{\mel}$ and $\WrGHZ{\Ggrp}{\Hgrp}{\gamma}{\mel}$}

Let $\Ggrp$ be a group with unit $e$ and $\mel\geq1$.
In Section~\ref{sect:grp:GwrmZ} we defined the semidirect product $\WrGZ{\Ggrp}{\mel}$ corresponding to the non-effective action of $\bZ$ on $\Ggrp^{\mel}$ by cyclic left shifts of coordinates.
Notice that $\WrGZ{\Ggrp}{\mel}$ contains the element $\gdif = (0,\ldots,0, 1)$ and the following subgroups
\[
    \Ggrp_i:= e\times \cdots \times e \times \Ggrp \times e\times \cdots \times e \times 0, \quad i=0,\ldots,\mel-1,
\]
where the multiple $\Ggrp$ stands at position $i$.
One easily checks that $\gdif^{\mel}=(0,\ldots,0, \mel)$ commutes with each $\Ggrp_i$, $\gdif \Ggrp_i \gdif^{-1} = \Ggrp_{i+1\bmod\mel}$ for all $i=0,\ldots,\mel-1$, and the kernel $\Ggrp^{\mel} \times 0$ of the natural epimorphism $\eta:\WrGZ{\Ggrp}{\mel}\to\bZ$, $\eta(\gel_0\ldots,\gel_{\mel-1},k)=k$, splits into a direct product of subgroups $\Ggrp_i$.
As the following lemma shows these properties characterize the group $\WrGZ{\Ggrp}{\mel}$.

\begin{lemma}[{\cite[Lemma~2.3]{Maksymenko:TA:2020}}]\label{lm:3x3_G_m:char}
Let $\eta:\Bgrp\to\bZ$ be an epimorphism with kernel $\Lgrp:=\ker(\eta)$.
Suppose there exist $\mel\geq1$, $\gel\in\Bgrp$, and a subgroup $\Ggrp$ of $\Lgrp$ such that
\begin{enumerate}[label={\rm(\alph*)}]
\item $\eta(\gel)=1$ and $\gel^{\mel}$ commutes with $\Lgrp$;
\item $\Lgrp$ is a direct product of subgroups $\Ggrp_i := \gel^{i}\Ggrp\gel^{-i}$, $i=0,\ldots,\mel-1$.
\end{enumerate}
Then the following map $\theta\colon \WrGZ{\Ggrp}{\mel}\to\Bgrp$, $\theta(\aeli{0},\aeli{1},\ldots,\aeli{\mel-1}; \, k) = \Bigl( \myprod_{i=0}^{\mel-1} \gel^{i} \aeli{i} \gel^{-i} \Bigr) \gel^k$,
is an isomorphism which also yields an isomorphism of the following short exact sequences:
\begin{equation}\label{equ:iso_GmZ__B}
\begin{gathered}
    \xymatrix@R=3ex{
        \Ggrp^{\mel}\times 0   \ar@{^(->}[r] \ar[d]                  &
        \WrGZ{\Ggrp}{\mel}     \ar[d]^-{\theta}\ar[r]                &
        \bZ                                       \ar@{=}[d]
        \\
        \Lgrp                                     \ar@{^(->}[r]      &
        \Bgrp                                     \ar[r]^-{\eta}     &
        \bZ
    }
\end{gathered}
\end{equation}
Let also $\Agrp$ be a normal subgroup of $\Bgrp$.
Denote $\Kgrp:=\Agrp\cap\Lgrp$, $\Pgrp := \Agrp\cap\Ggrp$, and suppose that
\begin{enumerate}[label={\rm(\alph*)}, resume]
\item $\eta(\Agrp)=\mel\bZ$, $\gel^{\mel}\in\Agrp$,
\item $\Kgrp$ is generated by subgroups $\Pgrp_i := \gel^{i}\Pgrp\gel^{-i}$, $i=0,\ldots,\mel-1$.
\end{enumerate}
Then $\theta(\Pgrp^{\mel}\times\mel\bZ)=\Agrp$, whence $\theta$ induces an isomorphism of the following exact $(3\times3)$-diagrams:
\begin{equation}\label{equ:diagram:eta_GwrmZ}
	\begin{gathered}
		\xymatrix@C=1.5em@R=1.5em{
        \exd
        {\Pgrp^{\mel}\times0}
        {\Ggrp^{\mel}\times0}
        {(\Ggrp/\Pgrp)^{\mel}\times0}
        {\Pgrp^{\mel}\times\mel\bZ}
        {\WrGZ{\Ggrp}{\mel}}
        {\WrGZm{(\Ggrp/\Pgrp)}{\mel}}
        {\mel\bZ}
        {\bZ}
        {\bZ_{\mel}}
		}
        \
        \xymatrix{ \\ \cong \\ }
        \
        \xymatrix@C=1.5em@R=1.5em{
            \exd
            {\Kgrp}
            {\Lgrp}
            {\Lgrp/\Kgrp}
            {\Agrp}
            {\Bgrp}
            {\Bgrp/\Agrp}
            {\mel\bZ}
            {\bZ}
            {\bZ_{\mel}}
        }
	\end{gathered}
\end{equation}
and that isomorphism is the identity on the lower row.
\end{lemma}

The following lemma gives a similar characterization of groups $\WrGHZ{\Ggrp}{\Hgrp}{\gamma}{\mel}$.

\begin{lemma}\label{lm:3x3_G_H_gamma_m:char}
Let $\eta:\Bgrp\to\bZ$ be an epimorphism with kernel $\Lgrp:=\ker(\eta)$.
Let also $\gel\in\Bgrp$ and $\xi:\Lgrp\to\Lgrp$, $\xi(l) = \gel^{-1} l \gel$, be the conjugation by $\gel$ automorphism of $\Lgrp$.
Suppose further that there exist $\mel\geq1$, and two subgroups $\Ggrp$ and $\Hgrp$ of $\Lgrp$ such that
\begin{enumerate}[label={\rm(\alph*)}]
\item\label{enum:lm:3x3_G_H_gamma_m:char:g2m_commutes_with_L} 
$\eta(\gel)=1$, and $\gel^{2\mel}$ commutes with $\Lgrp$, i.e.\ $\xi^{2\mel}=\id_{\Lman}$;
\item\label{enum:lm:3x3_G_H_gamma_m:char:gm_H_H} $\xi^{\mel}(\Hgrp) = \Hgrp$, so we have an automorphism $\gamma = \xi^{\mel}|_{\Hgrp}:\Hgrp\to\Hgrp$ with $\gamma^2=\id_{\Hgrp}$;
\item\label{enum:lm:3x3_G_H_gamma_m:char:L_GixHj} $\Lgrp$ is a direct product of subgroups%
\footnote{The signs of powers of $\gel$ are chosen so that $\xi$ will act on the tuple of subgroups
$(\Ggrp_0,\Ggrp_1,\ldots,\Ggrp_{2\mel-1})$ by cyclically shifting them to the left.
Similar observations hold for $\Hgrp_0,\ldots,\Hgrp_{\mel-1}$.
Note that this agrees with the behavior of $\bhom$ from Section~\ref{sect:grp:GHwr_gamma_m_Z}.}
\begin{align*}
 \Ggrp_i &:= \xi^{-i}(\Ggrp) \equiv \gel^{i}\Ggrp\gel^{-i}, \ (i=0,\ldots,2\mel-1), \\
 \Hgrp_j &:= \xi^{-j}(\Hgrp) \equiv \gel^{j}\Hgrp\gel^{-j}, \ (j=0,\ldots,\mel-1).
\end{align*}
\end{enumerate}
Then the following map $\theta\colon \WrGHZ{\Ggrp}{\Hgrp}{\gamma}{\mel}\to\Bgrp$,
\begin{align}\label{equ:theta_iso}
\theta(\aeli{0},\aeli{1},\ldots,\aeli{2\mel-1};\,
        \beli{0},\beli{1},\ldots,\beli{\mel-1}; \, k) =
        \Bigl( \prod_{i=0}^{2\mel-1} \xi^{-i} (\aeli{i})  \cdot \prod_{j=0}^{\mel-1} \xi^{-j} (\beli{j}) \Bigr) \gel^k,
\end{align}
is an isomorphism which also yields an isomorphism of the following short exact sequences:
\begin{equation}\label{equ:iso_G2mHmZ__B}
\begin{gathered}
    \xymatrix@R=3ex{
        \Ggrp^{2\mel}\times\Hgrp^{\mel}\times 0   \ar@{^(->}[r] \ar[d]     &
        \WrGHZ{\Ggrp}{\Hgrp}{\gamma}{\mel}        \ar[d]^-{\theta}\ar[rr]  &&
        \bZ                                       \ar@{=}[d]
        \\
        \Lgrp                                     \ar@{^(->}[r]           &
        \Bgrp                                     \ar[rr]^-{\eta}         &&
        \bZ
    }
\end{gathered}
\end{equation}
Let also $\Agrp$ be a normal subgroup of $\Bgrp$.
Denote $\Kgrp:=\Agrp\cap\Lgrp$, $\Pgrp := \Agrp\cap\Ggrp$, $\Qgrp := \Agrp\cap\Hgrp$, and suppose that
\begin{enumerate}[label={\rm(\alph*)}, resume]
\item\label{enum:lm:3x3_G_H_gamma_m:char:b_2m} $\eta(\Agrp)=2\mel\bZ$, $\gel^{2\mel}\in\Agrp$;
\item\label{enum:lm:3x3_G_H_gamma_m:char:K_gen_Pi_Qj} $\Kgrp$ is generated by subgroups
\begin{equation}\label{equ:subgroups_Pi_Qj}
\begin{gathered}
    \Pgrp_i := \xi^{i}(\Pgrp) \equiv \gel^{-i}\Pgrp\gel^{i} = \Agrp\cap\Ggrp_i,  \ i=0,\ldots,2\mel-1, \\
    \Qgrp_j := \xi^{j}(\Qgrp) \equiv \gel^{-j}\Qgrp\gel^{j} = \Agrp\cap\Hgrp_j,  \ j=0,\ldots,\mel-1.
\end{gathered}
\end{equation}
\end{enumerate}
Then $\theta(\Pgrp^{\mel}\times\Qgrp^{\mel}\times2\mel\bZ)=\Agrp$, whence $\theta$ also induces an isomorphism of the following exact $(3\times3)$-diagrams:
{\small\[
		\xymatrix@C=1em@R=1.5em{
            \exd
			{\Pgrp^{2\mel}\!\times\!\Qgrp^{\mel}\!\times\! 0}
            {\Ggrp^{2\mel}\!\times\!\Hgrp^{\mel}\!\times\! 0}
            {(\Ggrp/\Pgrp)^{2\mel}\!\times\! (\Hgrp/\Qgrp)^{\mel} \!\times\! 0}
			{\Pgrp^{2\mel}\!\times\!\Qgrp^{\mel}\!\times\! 2\mel\bZ}
            {\WrGHZ{\Ggrp}{\Hgrp}{\gamma}{\mel}}
            {\WrGHZm{\Ggrp/\Pgrp}{\Hgrp/\Qgrp}{\gamma}{\mel}}
			{2\mel\bZ}
            {\bZ}
            {\bZ_{2\mel}}
		}
        \xymatrix{ \\ \cong \\ }
        \xymatrix@C=1em@R=1.6em{
            \exd
            {\Kgrp}
            {\Lgrp}
            {\Lgrp/\Kgrp}
            {\Agrp}
            {\Bgrp}
            {\Bgrp/\Agrp}
            {2\mel\bZ}
            {\bZ}
            {\bZ_{2\mel}}
        }
\]}%
and that isomorphism is the identity on the lower row.
\end{lemma}
\begin{proof}
First note that due to~\ref{enum:lm:3x3_G_H_gamma_m:char:L_GixHj} the product in round brackets in~\eqref{equ:theta_iso} does not depend on their order, and therefore the map $\theta$ is well-defined.

Recall further that $\WrGHZ{\Ggrp}{\Hgrp}{\gamma}{\mel}$ is the semidirect product $\smd{(\Ggrp^{2\mel}\times\Hgrp^{\mel})}{\bhom}{\bZ}$ corresponding to the isomorphism $\bhom$ given by~\eqref{equ:Aut_in_WrGH_gamma_m}.
Also, by Lemma~\ref{lm:char_B_rtimes_Z}, $\Bgrp$ is isomorphic to $\smd{\Lgrp}{\xi}{\bZ}$.
Hence, due to Lemma~\ref{lm:smd_iso_Z}, for the proof that $\theta$ is an isomorphism, it suffices to check that $\theta\circ\bhom = \xi\circ\theta$ on $\Lgrp$.
Indeed, since $\gamma=\xi^{\mel}$ we have that
\begin{align*}
    &\theta\circ\bhom(\aeli{0},\ldots,\aeli{2\mel-1},\beli{0},\ldots, \beli{\mel-1}) =
    \theta(\aeli{1},\ldots, \aeli{2\mel-1}, \aeli{0}; \,   
           \beli{1},\ldots, \beli{\mel-1},  \gamma(\beli{0})) = \\
    &\quad =
    \prod_{i=1}^{2\mel-1} \xi^{-(i-1)} (\aeli{i})  \ \cdot \
    \xi^{-(2\mel-1)}(\aeli{0}) \ \cdot \
    \prod_{j=1}^{\mel-1} \xi^{-(j-1)} (\beli{j})   \ \cdot \
    \xi^{-(\mel-1)+\mel}(\beli{0}) = \\
    &\quad =
    \prod_{i=0}^{2\mel-1} \xi^{-i+1} (\aeli{i})  \ \cdot \
    \prod_{j=0}^{\mel-1} \xi^{-j+1} (\beli{j}) = \xi \Bigl(
           \prod_{i=0}^{2\mel-1} \xi^{-i} (\aeli{i})  \ \cdot \
           \prod_{j=0}^{\mel-1} \xi^{-j} (\beli{j})
        \Bigr) = \\
    &\quad = \xi\circ\theta(\aeli{0},\ldots,\aeli{2\mel-1},\beli{0},\ldots, \beli{\mel-1}).
\end{align*}

Suppose~\ref{enum:lm:3x3_G_H_gamma_m:char:b_2m} and~\ref{enum:lm:3x3_G_H_gamma_m:char:K_gen_Pi_Qj} hold.
Then $\Agrp$ splits into the direct product of subgroups $\Kgrp \times \langle\gel^{2\mel}\rangle$.
Moreover, by~\ref{enum:lm:3x3_G_H_gamma_m:char:L_GixHj}, $\Kgrp$ is a direct product $\myprod_{i=0}^{2\mel-1}\Pgrp_i \times \myprod_{j=0}^{\mel-1}\Qgrp_j$ of its subgroups~\eqref{equ:subgroups_Pi_Qj}.
It now follows from~\eqref{equ:theta_iso} that $\theta(\Pgrp^{\mel}\times\Qgrp^{\mel}\times2\mel\bZ)$ is exactly
$\myprod_{i=0}^{2\mel-1}\Pgrp_i \times \myprod_{j=0}^{\mel-1}\Qgrp_j \times \langle\gel^{2\mel}\rangle = \Agrp$.
\end{proof}

\newcommand\mycaption[3]{\caption{#1, \\
\centerline{$#2$}}
}

\section{Codimension of singularities}\label{sect:preliminaries}
In this Section we give a short and elementary proof of finiteness of Milnor numbers of homogeneous polynomials $\gfunc\in\bR[x,y]$ without multiple factors, (Lemma~\ref{lm:axH_fin_codim}), which is a principal property allowing to compute the homotopy types of orbits of maps from $\FSp{\Mman}$.
In fact, that statement follows from general results of algebraic geometry, see Lemma~\ref{lm:char_codim_complex}, and is used in mentioned above papers, e.g.\ \cite{Maksymenko:TA:2020}, however the authors did not find its explicit proof in available literature.

We will also discuss structure of level-sets of isolated critical points on surfaces.
These results will not be directly used in the proofs of main results and may be skipped on first reading.

\subsection{Maps of finite codimension}\label{sect:finite_codimension}
Let $\CoRn$ be the algebra of germs at the origin $\orig\in\bR^n$ of $\Cinfty$ functions $\func:\bR^n\to\bR$.
For $\func\in\CoRn$ denote by $\JidR{\func}$ the ideal in $\CoRn$ generated by germs of partial derivatives of $\func$.
Then the real codimension of $\JidR{\func}$ in $\CoRn$:
\begin{equation}\label{equ:MilnorNumber}
    \milnorNum{\bR}{\func}:= \dim_{\bR} \bigl( \CoRn / \JidR{\func} \bigr)
\end{equation}
is called the \term{codimension} or \term{Milnor number} or \term{multiplicity} of $\func$ at the critical point $\orig$ with respect to the algebra $\CoRn$.

The well-known Tougeron theorem~\cite{Tougeron:AIF:1968} claims that \term{if $\orig$ is a critical point of a finite codimension $k$ of a $\Cinfty$ germ $\func:(\bR^{n},\orig)\to(\bR,\orig)$, then $\func$ is $\Cinfty$ equivalent to its Taylor polynomial of order $k+1$}.

It is easy to see that $\milnorNum{\bR}{\func}$ does not depend on local coordinates at $\orig\in\bR^n$, and therefore it can be defined for $\Cinfty$ functions on manifolds and even for $\Cinfty$ circle-valued maps.

Now let $\Mman$ be a smooth compact manifold and $\Pman$ be either $\bR$ or $\Circle$.
Say that a map $\func\in\Cid{\Mman}{\Pman}$ is of \term{finite codimension} if it has only finitely many critical points and at each of them $\func$ has finite codimension~\eqref{equ:MilnorNumber}.

\begin{theorem}[{\rm\cite{Sergeraert:ASENS:1972, Maksymenko:AGAG:2006}}]
\label{th:maps_of_fin_codim}
Let $\func\in\Cid{\Mman}{\Pman}$ be a map of finite codimension and $\Yman$ be any collection of boundary components of $\Mman$.
Then
\begin{enumerate}[label={\rm(\alph*)}, leftmargin=*]
\item\label{enum:lm:maps_of_fin_codim:Of}
the corresponding orbit $\Orbit{\func,\Yman}$ is a Fr\'{e}chet submanifold of finite codimension of the Fr\'{e}chet manifold $\Cid{\Mman}{\Pman}$,
\item\label{enum:lm:maps_of_fin_codim:DY}
the natural map $\nu:\Diff(\Mman,\Yman)\to \Orbit{\func,\Yman}$, $\dif \mapsto \func\circ\dif$, is a principal locally trivial fibration with fiber $\Stabilizer{\func,\Yman}$;
\item\label{enum:lm:maps_of_fin_codim:DIdY}
$\nu(\DiffId(\Mman,\Yman))=\OrbitComp{\func}{\func,\Yman}$ and the restriction map $\nu:\DiffId(\Mman,\Yman)\to \OrbitComp{\func}{\func,\Yman}$, $\dif \mapsto \func\circ\dif$, is also principal locally trivial fibration with fiber $\StabilizerIsotId{\func,\Yman} := \Stabilizer{\func,\Yman} \cap \DiffId(\Mman,\Yman)$.
\end{enumerate}
\end{theorem}
\begin{proof}[Remarks to the proof]
Note that~\ref{enum:lm:maps_of_fin_codim:DIdY} directly follows from~\ref{enum:lm:maps_of_fin_codim:DY}.
For the statements~\ref{enum:lm:maps_of_fin_codim:Of} and~\ref{enum:lm:maps_of_fin_codim:DY} consider first the so-called \term{left-right} action
\[
    \Diff(\bR)\times\Ci{\Mman}{\bR}\times\Diff(\Mman) \to \Ci{\Mman}{\bR},
    \qquad
    (\phi,\func,\dif) \mapsto \phi\circ\func\circ\dif,
\]
of the group $\Diff(\bR)\times\Diff(\Mman)$ on $\Ci{\Mman}{\bR}$.
F.~Sergeraert~\cite{Sergeraert:ASENS:1972} proved that if $\func\in\Ci{\Mman}{\bR}$ is a smooth function \term{of finite codimension}, then the corresponding orbit
\[
    \widetilde{\mathcal{O}}(\func)
      = \{
            \phi\circ\func\circ\dif \mid (\phi,\dif)\in\Diff(\bR)\times\Diff(\Mman)
        \}
\]
is a Fr\'{e}chet submanifold of finite codimension of the Fr\'{e}chet space $\Ci{\Mman}{\bR}$, while
the natural map $\nu:\Diff(\bR)\times\Diff(\Mman)\to \widetilde{\mathcal{O}}(\func)$, $(\phi,\dif) \mapsto \phi\circ\func\circ\dif$, is a principal locally trivial fibration.

In~\cite[Theorem~11.7]{Maksymenko:AGAG:2006} such a result was extended to ``tame actions of tame Lie groups on tame Fr\'{e}chet manifolds'', where \term{tameness} is understood in the sense of R.~Hamilton~\cite{Hamilton:BAMS:1982}.
In particular, this implied that~\ref{enum:lm:maps_of_fin_codim:Of} and~\ref{enum:lm:maps_of_fin_codim:DY} hold for $\func$ being a Morse map and $\Yman=\varnothing$, see~\cite[Section 11.30]{Maksymenko:AGAG:2006}.

However, almost literally the same arguments show that the same result holds for the above map $\nu:\Diff(\Mman,\Yman)\to\Orbit{\func,\Yman}$, where $\Yman$ is any (possibly empty) collection of boundary components of $\Mman$ and $\func:\Mman\to\Pman$ is a map of finite codimension.
One should just
\begin{itemize}[label=$- $, leftmargin=*]
\item
write in Eq.~(2) of~\cite[Section~11.30]{Maksymenko:AGAG:2006} $g(x)$ as 
\[ g(x) = \sum_{i=1}^{k} c_i \lambda_i(x) + \sum \beta_i(x) \gfunc'_{x_i}(x),\]
where $c_i\in\bR$ are some constants, and $\lambda_i$ some $\Cinfty$ functions which span the (finite-dimensional) complement to the Jacobi ideal of $\func$ at $z_i$;
\item
note that in Eq.~(3) of~\cite[Section~11.30]{Maksymenko:AGAG:2006} the vector field $H_j$ vanishes on the boundary component $B_i$, and thus belongs to the tangent space at $\id_{\Mman}$ of the group $\Diff(\Mman,\Yman)$ not only of $\Diff(\Mman)$.
\qedhere
\end{itemize}
\end{proof}

\subsection{Finiteness of Milnor number}
Denote by $\holoCn$ the $\bC$-algebra of germs of analytic maps $\bC^n\to\bC$ at $\orig\in\bC^2$, i.e.\ maps represented by series covering on some neighborhood of $\orig\in\bC^n$.
In particular, we can regard the ring of polynomials $\bC[z_1,\ldots,z_n]$ as a subalgebra of $\holoCn$.
Let also 
\[\maxId = \{ \func\in \holoCn\mid \func(\orig)=0\}\]
be the unique maximal ideal in $\holoCn$.
Note that for each $\func\in\holoCn$ one can define its Jacobi ideal $\JidC{\func} = (\func'_{z_1},\ldots,\func'_{z_n})$ in $\holoCn$ generated by partial derivatives of $\func$.
Its codimension, $\milnorNum{\bC}{\func}:=\dim_{\bC}\bigl(\holoCn/\JidC{\func} \bigr)$, is called the \term{Milnor number} of $\func$ with respect to $\holoCn$.

\begin{lemma}\label{lm:char_codim_complex}
For $\func\in\holoCn$ the following conditions are equivalent:
\begin{enumerate}[label={\rm$(\mu\arabic*)$}]
\item\label{enum:miln:zim_in_Jid} there exists $m\geq1$ such that $z_i^m\in\JidC{\gfunc}$ for each $i=1,\ldots,n$;
\item\label{enum:miln:monoms} there exists $p\geq1$ such that $z_1^{a_1}\cdots z_n^{a_n}\in\JidC{\gfunc}$ if $a_1+\cdots+a_n \geq p$;
\item\label{enum:miln:mp} there exists $p\geq1$ such that $\maxId^{p} \subset \JidC{\gfunc}$;
\item\label{enum:miln:isol} $0\in\bC^n$ is an isolated critical point of $\func$;
\item\label{enum:miln:finite} $\milnorNum{\bC}{\func} < \infty$.
\end{enumerate}
\end{lemma}
\begin{proof}
The equivalence \ref{enum:miln:zim_in_Jid}$\Leftrightarrow$\ref{enum:miln:monoms}$\Leftrightarrow$\ref{enum:miln:mp} is evident.

\ref{enum:miln:mp}$\Rightarrow$\ref{enum:miln:finite}.
Note that for each $p\geq1$ the vector space $\holoCn/\maxId^{p}$ over $\bC$ is generated by \term{finitely many} monomials $\{z_1^{a_1}\cdots z_n^{a_n} \mid 0\leq a_1+\cdots a_n < p\}$.
Hence, if $\maxId^{p} \subset \JidC{\gfunc}$, then
\[
    \milnorNum{\bC}{\func} := \dim_{\bC}\bigl(\holoCn/\JidC{\func} \bigr)
                           \leq \dim_{\bC}\bigl(\holoCn/\maxId^{p}\bigr) < \infty.
\]

The equivalence~\ref{enum:miln:isol}$\Leftrightarrow$\ref{enum:miln:finite} is a principal non-trivial statement which can be found in~\cite[Lemma~2.3]{GreuelLossenShustin:Sing:2007}, which in turn is based on~\cite[Proposition~1.70]{GreuelLossenShustin:Sing:2007}.

\ref{enum:miln:isol}$\Rightarrow$\ref{enum:miln:mp}
This implication follows in fact from the proof of the implication (d)$\Rightarrow$(b) of that~\cite[Proposition~1.70]{GreuelLossenShustin:Sing:2007}.
Namely, in that proof it is actually shown that isolateness of $\orig$ implies that $\maxId^{p} \subset \JidC{\gfunc}$ for some $p\geq1$, i.e.\ that~\ref{enum:miln:isol} implies~\ref{enum:miln:mp}.
\end{proof}

\begin{corollary}\label{cor:fin_codim_real}
Let $\func\in\CoRn$ be an analytic germ, i.e.\ a series with real coefficients converging on some neighborhood of $\orig\in\bR^n$; for instance $\func$ can be a polynomial.
Regard $\func$ as complex series with real coefficients, i.e.\ as an element of $\holoCn$.
Then either of the conditions~\ref{enum:miln:zim_in_Jid}-\ref{enum:miln:finite} implies that $\milnorNum{\bR}{\func}<\infty$.
\end{corollary}
\begin{proof}
Let $\phi\in\CoRn$.
Then by Hadamard lemma there are $\Cinfty$ germs $\alpha_1,\ldots,\alpha_n\in\CoRn$ and $b_0=\phi(\orig)$ such that $\phi(\bx) = b_0 + \sum_{j=1}^{n} \beta_j(\bx)x_j$, where $\bx=(x_1,\ldots,x_n)\in\bR^n$.
Applying the same statement to each $\beta_j$ and so on we get that for each $p$ we have the identity
\begin{equation}\label{equ:phi_decomp}
\phi(x,y) = \sum_{0\leq a_1+\cdots+a_n < p} b_{ij} x_1^{a_1}\cdots x_{n}^{a_n} +
\sum_{a_1+\cdots+a_n  = p} \beta_{a_1,\ldots,a_n}(\bx) x_1^{a_1}\cdots x_{n}^{a_n}
\end{equation}
for some $b_{ij}\in\bR$ and $\beta_{a_1,\ldots,a_n}\in\CoRn$.
Hence, as in the proof of \ref{enum:miln:mp}$\Rightarrow$\ref{enum:miln:finite} in Lemma~\ref{lm:char_codim_complex}, in order to prove that $\milnorNum{\bR}{\func}<\infty$, \term{it suffices to show that there exists $m\geq1$ such that $x_i^m\in\JidR{\gfunc}$ for all $i=1,\ldots,n$}.
Then for some large $p$, the second term in~\eqref{equ:phi_decomp} will belong to $\JidR{\gfunc}$, whence the vector space $\CoRn/\JidR{\gfunc}$ over $\bR$ will be generated by a finite set of monomials of degree $< p$.

Regard $\bR^{n}$ as a subspace of $\bC^{n}$ of real coordinates.
By assumption $\func$, as an element of $\holoCn$, satisfies condition~\ref{enum:miln:zim_in_Jid}, so there exist $m\geq1$ and $\tau_1,\ldots,\tau_n\in\holoCn$ such that
\begin{equation}\label{equ:zim_decomp}
    z_i^m = \tau_1(\bz) \func'_{z_1}(\bz) + \cdots +\tau_n(\bz) \func'_{z_n}(\bz),
\end{equation}
for all $i=1,\ldots,n$, where $\bz = (z_1,\ldots,z_n)\in\bC^n$.
Let $\tau_j = \alpha_j + i \beta_j$ be the decomposition of $\tau_j$ into the real and imaginary parts.
Then $\alpha_j$ and $\beta_j$ are converging series near $\orig$ in $\bC^n$ with real coefficients, and $\alpha_j|_{\bR^{n}}, \beta_j|_{\bR^{n}} \in \CoRn$.
Since $\func$ has real coefficients, we also have that $\func'_{z_j}(\bx) = \func'_{x_i}(\bx)$ for all $j$.
Now taking the real parts of both sides of~\eqref{equ:zim_decomp} and restricting them to the real subspace $\bR^n \subset \bC^{n}$, i.e.\ substituting $\bx$ instead of $\bz$, we get that 
\[ 
    x_i^m = \alpha_j(\bx) \func'_{x_1}(\bx) + \cdots +\alpha_n(\bx) \func'_{x_n}(\bx) \in \JidR{\func}.
    \qedhere
\]
\end{proof}

\subsection{Homogeneous polynomials in two variables}
Let $\gfunc:\bR^2\to\bR$ be a real homogeneous polynomial with $\deg\gfunc\geq2$.
Then it directly follows from fundamental theorem of algebra that $\gfunc$ splits into a product of finitely many linear and irreducible over $\bR$ quadratic factors:
\begin{equation}\label{equ:g_prod_L_Q}
    \gfunc(x,y) = \prod_{i=1}^{p} (a_i x + b_i y) \cdot \prod_{j=1}^{q}(c_j x^2 + 2d_j xy + e_j y^2).
\end{equation}
Evidently,
\begin{enumerate}[leftmargin=*]
\item
$\gfunc^{-1}(0)$ is a union of straight lines $\{ a_i x + b_i y = 0 \}$ corresponding to the linear factors;

\item
if $\gfunc$ has two proportional linear factors, then all points of the corresponding line are critical for $\gfunc$;

\item
hence, $\gfunc$ has no multiple (non-proportional) \term{linear} factors if and only if the origin $\orig\in\bR^2$ is an isolated critical point;
in this case $\gfunc^{-1}(0)$ consists of $2p$ rays starting from the origin;

\item
the following conditions are equivalent due to Lemma~\ref{lm:char_codim_complex}:
\begin{enumerate}[label={\rm(4\alph*)}]
\item\label{enum:g:no_multi_factors} $\gfunc$ has no multiple factors at all;
\item the partial derivatives $\gfunc'_x$ and $\gfunc'_x$ of $\gfunc$ have no common factors;
\item the origin $\orig\in\bC^2$ is an isolated critical point of $\gfunc:\bC^2\to\bC$ as a complex polynomial with real coefficients;
\item
$\milnorNum{\bC}{\gfunc}<\infty$.
\end{enumerate}
Moreover, by Corollary~\ref{cor:fin_codim_real}, they also imply that $\milnorNum{\bR}{\gfunc}<\infty$.
\end{enumerate}
For the completeness and for the convenience of the reader we present a short explicit and not based on Lemma~\ref{lm:char_codim_complex} proof that~\ref{enum:g:no_multi_factors} implies $\milnorNum{\bR}{\gfunc}<\infty$.

\begin{lemma}\label{lm:axH_fin_codim}
Suppose $\gfunc\in\bR[x,y]$ is a homogeneous polynomial without multiple factors and $\deg\gfunc\geq2$.
Then $x^m,y^m\in\JidR{\func}$ for some large $m\geq1$, and therefore $\milnorNum{\bR}{\gfunc} < \infty$.

In particular, if $\Mman$ is a compact surface, then each $\func\in\FSp{\Mman}$ is of finite codimension, and Theorem~\ref{th:maps_of_fin_codim} holds for $\func$.
\end{lemma}
\begin{proof}
Put $k = \deg\gfunc - 1$.
To simplify notation also redenote $A := \gfunc'_x$ and $B := \gfunc'_y$.
Then $A$ and $B$ are homogeneous polynomials of degree $k$.
As mentioned above, the assumption that $\gfunc$ has no multiple factors implies that $A$ and $B$ are relatively prime in $\bR[x,y]$.
We will find homogeneous polynomials $P,Q\in \bR[x,y]$ such that $A P + B Q = x^m$ for some $m\geq1$.
This will mean that $x^m\in\JidR{\func}$.
The proof that $y^m\in\JidR{\func}$ for some $m$ is similar.

Consider the following polynomials $\alpha(t) := A(1,t)$, $\beta(t) := B(1,t) \in \bR[t]$.
Since $A$ and $B$ are homogeneous of degree $k$, we have that
\begin{align*}
    A(x,y) &= \alpha(y/x) x^k, &
    B(x,y) &= \beta(y/x) x^k.
\end{align*}
Moreover, $\mathrm{gcd}(\alpha(t),\beta(t)) = 1$ in $\bR[t]$, since $\mathrm{gcd}(A,B)=1$ in $\bR[x,y]$.
In particular, using Euclid division algorithm, one can find polynomials $p,q\in\bR[t]$ such that
\begin{align}\label{equ:aipi_biqi_1}
    \alpha(t) p(t) + \beta(t) q(t) = 1.
\end{align}
This implies that $\deg (\alpha p) = \deg(\beta q)$, and we will denote this common degree by $m$.
It follows that $P:=p(y/x)x^{\deg p}$ and $Q:=q(y/x)x^{\deg q}$ are homogeneous polynomials in $\bR[x,y]$.
Multiplying~\eqref{equ:aipi_biqi_1} by $x^{m}$ and substituting $t=y/x$ we get:
\[
    \bigl(\alpha(y/x) x^{\deg\alpha}\bigr) \, \bigl( p(y/x) x^{\deg p} \bigr) \ + \
    \bigl(\beta(y/x)  x^{\deg\beta} \bigr) \, \bigl( q(y/x) x^{\deg q} \bigr) \ = \ x^{m},
\]
i.e.\ $A P + B Q = x^{m} \in\JidR{\func}$.

Now, let $\Mman$ be a compact surface and $\func\in\FSp{\Mman}$.
Then $\func$ has finitely many critical points and at each of them its germ is $\Cinfty$ equivalent to a homogeneous polynomial without multiple factors.
As just proved the Milnor number of $\func$ at each critical point is finite.
Therefore, $\func$ is also of finite codimension.
\end{proof}

\begin{remark}\label{rem:isol_sing}\rm
Let $\Mman$ be a $2$-dimensional manifold and $\func\in\Ci{\Mman}{\Pman}$.
Suppose $\pw\in\Int{\Mman}$ is an isolated critical point of $\func$.
Then there are germs of \term{homeomorphisms} $\dif:(\bC,0) \to (\Mman,\pw)$ and $\phi:(\bR,0) \to (\bR,0)$ such that
\begin{equation}\label{equ:isol_sing}
\phi \circ \func \circ \dif (z) =
\begin{cases}
|z|^2, & \text{if $0\in\bC$ is a \term{local extreme} of $\func$, \cite[Lemma~3]{Dancer:2:JRAM:1987}}, \\
Re(z^p), & \text{for some $p\geq1$ otherwise, \cite{ChurchTimourian:PJM:1973, Prishlyak:TA:2002}}.
\end{cases}
\end{equation}
Moreover, one can make $\dif$ to be $\Cinfty$ out of an arbitrary small neighborhood of $\pw$.
Figure~\ref{fig:isol_crit_pt} shows examples of level sets of smooth functions near isolated critical points.
\begin{figure}[htbp!]
	\centering
	\footnotesize
	\begin{tabular}{ccccccc}
		\includegraphics[height=1.2cm]{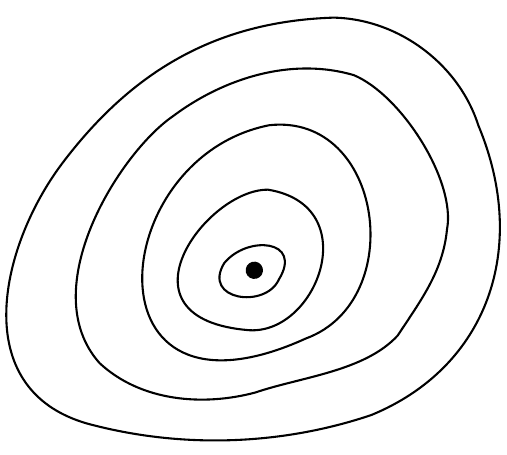}  & \ &
		\includegraphics[height=1.2cm]{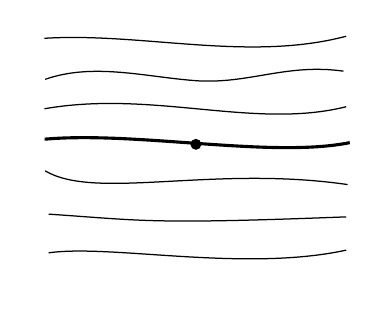}    & \ &
		\includegraphics[height=1.2cm]{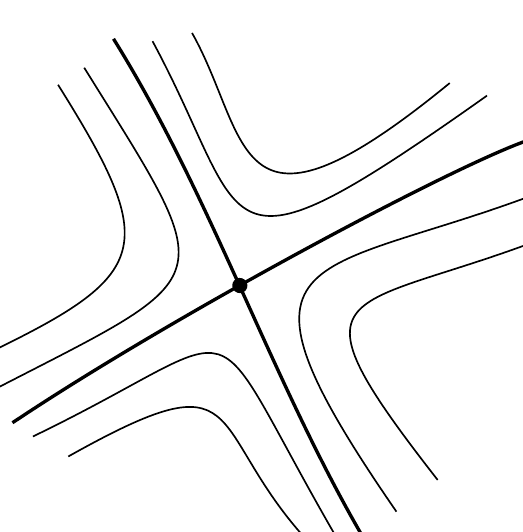}    & \ &
		\includegraphics[height=1.2cm]{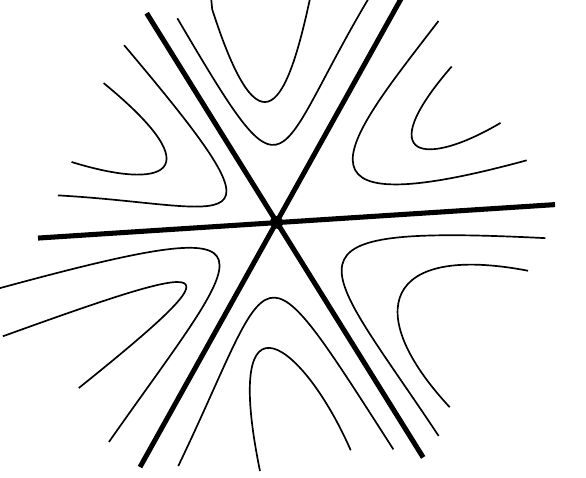}
	\end{tabular}
	\caption{Topological structure of level-sets of isolated singularities}
	\label{fig:isol_crit_pt}
\end{figure}
Note that every function in~\eqref{equ:isol_sing} is a homogeneous polynomial without multiple factors.
Hence, for any compact surface $\Mman$ the class $\FSp{\Mman}$ is not only ``massive'' but also contains ``up to homeomorphism'' each map $\func\in\Cid{\Mman}{\Pman}$ with isolated critical points.

It is worth to mention that a topological structure of functions with isolated critical points on compact $1$-manifolds was studied e.g. in~\cite{Arnold:UMN:1992}, and for $2$-manifolds e.g. in~\cite{Prishlyak:UMJ:2000, HladyshPrishlyak:JMPAG:2019}.
\end{remark}

\section{Functions on surfaces}\label{sect:func_on_surf}
Let $\Mman$ be a compact surface, $\func\in\FSp{\Mman}$, and $\fSing$ be the set of all critical points of $\func$.
Then every connected component of each level set of $\func$ will be called a \term{contour} of $\func$.
A contour $\Xman$ of $\func$ is \term{critical} if it contains critical points of $\func$, and \term{regular} otherwise.
Evidently, each regular contour is a submanifold of $\Mman$ diffeomorphic to the circle.

On the other hand, due to Axiom~\ref{axiom:Hm}, see also Figure~\ref{fig:isol_crit_pt}, a critical contour has a structure of a $1$-dimensional CW-complex whose $0$-cells are critical points of $\func$ belonging to $\Xman$, while the connected components of $\Xman\setminus\fSing$ being one-dimensional submanifolds of $\Mman$ are $1$-cells of $\Xman$.

Denote by $\KRGraphf$ the partition of $\Mman$ into contours of $\func$, and let $p:\Mman\to\KRGraphf$ be the quotient map associating to each $x\in\Mman$ the contour of $\func$ containing $x$.
Endow $\KRGraphf$ with the corresponding quotient topology with respect to $p$.
Then it is well known that $\KRGraphf$ has a structure of a finite $1$-dimensional CW-complex and is called \term{Kronrod-Reeb graph} of $\func$, \cite{Reeb:CR:1946, Adelson-Welsky-Kronrode:DAN:1945, Kronrod:UMN:1950}.

A subset $\Xman\subset\Mman$ will be called \term{$\func$-adapted} if it is a union of (possibly uncountably many) contours of $\func$.
In other words, $\Xman = p^{-1}(p(\Xman))$.

For instance, a compact $1$-dimensional submanifold $\Xman$ of $\Mman$ is $\func$-adapted whenever it is a union of finitely many regular contours.
Also, a compact $2$-dimensional submanifold $\Xman$ of $\Mman$ is $\func$-adapted iff $\partial\Xman$ is $\func$-adapted, i.e.\ every boundary component of $\Xman$ is a contour.
For example, if $\func\in\FSpR{\Mman}$ and $a<b$ are two regular values of $\func$, then $\func^{-1}\bigl([a;b]\bigr)$ is $\func$-adapted.

Let $\Xman\subset\Mman$ be a connected $\func$-adapted submanifold.
Then by an \term{$\func$-regular neighborhood} of $\Xman$ we will mean an arbitrary connected $\func$-adapted subsurface $\regU{\Xman}$ such that
$\regU{\Xman}$ is a neighborhood of $\Xman$ and $\regU{\Xman}\setminus\Xman$ contains no critical points of $\func$.

More generally, suppose $\Xman=\mathop{\cup}\limits_{i=1}^{k}\Xman_i$ is a disjoint union of connected $\func$-adapted submanifolds $\Xman_i$.
For every $i=1,\ldots,k$ choose any $\func$-regular neighborhood $\regU{\Xman_i}$ of $\Xman_i$ such that $\regU{\Xman_i}\cap\regU{\Xman_j}=\varnothing$ for $i\not=j$.
Then their union $\regU{\Xman} = \mathop{\cup}\limits_{i=1}^{k}\regU{\Xman_i}$, will also be called an \term{$\func$-regular neighborhood} of $\Xman$.

Denote by $\FolStabilizer{\func}$ the subgroup of $\Stabilizer{\func}$ consisting of diffeomorphisms $\dif$ having the following properties:
\begin{enumerate}[leftmargin=*]
\item
$\dif$ leaves invariant every connected component of every level set of $\func$, i.e.\ it preserves the elements of  $\KRGraphf$;
\item
if $z$ is a \term{degenerate local extreme}, and thus $\dif(z)=z$, then the tangent map $T_{z}\dif:T_{z}\Mman\to T_{z}\Mman$ is the identity $\id_{T_{z}\Mman}$.
\end{enumerate}

Evidently, $\FolStabilizer{\func}$ is a normal subgroup of $\Stabilizer{\func}$.
Moreover, it follows from~\cite[Lemma~6.2]{Maksymenko:MFAT:2009}, describing local linear symmetries of homogeneous polynomials, that the quotient $\Stabilizer{\func}/\FolStabilizer{\func}$ is finite.
In fact, it can be regarded as a group of automorphisms of an ``enhanced'' Kronrod-Reeb graph of $\func$ induced by elements of $\Stabilizer{\func}$, see~\cite[Section~4]{Maksymenko:TA:2020}.
That graph is obtained from $\KRGraphf$ by adding a certain finite number edges to each vertex corresponding to a \term{degenerate local extreme} of $\func$, so that every $\dif\in\Stabilizer{\func}$ ``cyclically rotates'' those new edges.
In particular, if $\func$ is Morse, so it has no degenerated critical points, then that ``enhanced'' graph coincides with $\KRGraphf$, and in that case $\Stabilizer{\func}/\FolStabilizer{\func}$ is the group of automorphisms of $\KRGraphf$ induced by $\Stabilizer{\func}$.
However, we will not use that interpretation in the present paper.

For a closed subset $\Xman\subset\Mman$ denote
\begin{equation}\label{equ:Delta_G}
\begin{aligned}
    \FolStabilizer{\func,\Xman}       &:= \FolStabilizer{\func}\cap\Diff(\Mman,\Xman), &
    \GrpKR{\func,\Xman}               &:= \Stabilizer{\func,\Xman}/\FolStabilizer{\func,\Xman}, \\
    \FolStabilizerIsotId{\func,\Xman} &:= \FolStabilizer{\func}\cap\DiffId(\Mman,\Xman), &
    \GrpKRIsotId{\func,\Xman}         &:= \StabilizerIsotId{\func,\Xman}/\FolStabilizerIsotId{\func,\Xman}.
\end{aligned}
\end{equation}
Then we have the following short exact sequences
\begin{align}
    \mathbf{b}(\func,\Xman) &: \FolStabilizer{\func,\Xman}
    \monoArrow
    \Stabilizer{\func,\Xman}
    \epiArrow
    \GrpKR{\func,\Xman}, \notag
    \\
    \label{equ:Bieberbach_seq}
    \mathbf{b}'(\func,\Xman) &: \FolStabilizerIsotId{\func,\Xman}
    \monoArrow
    \StabilizerIsotId{\func,\Xman}
    \epiArrow
    \GrpKRIsotId{\func,\Xman}.
\end{align}
The second sequence~\eqref{equ:Bieberbach_seq} will play the most important role, and we will call it the \term{Bieberbach (short exact) sequence} of $(\func,\Xman)$, see~\cite{Maksymenko:TA:2020}.
For instance, \cite[Theorem~5.4]{Maksymenko:TA:2020}, in the statement~\ref{enum:th:prelim_statements:Sprf_product} of Theorem~\ref{th:prelim_statements} the inclusions~\eqref{equ:stab_inclusions} yield in fact isomorphisms of the sequences:
\begin{equation}\label{equ:iso_bbb}
    \mathbf{b}'(\func,\Yman)          \ \cong \
    \mathbf{b}'(\func,\Yman\cup\Xman) \ \cong \
    \prod\limits_{i=1}^{p}\mathbf{b}'(\func|_{\Lman_{i}},\partial\Lman_{i}).
\end{equation}
Our results will also be given in terms of such sequences, see proofs of Lemmas~\ref{lm:all_disks_T1}-\ref{lm:there_are_disks_T1_and_T2}.

Finally, denote by $\DiffNbh(\Mman,\Xman)$ the group of diffeomorphisms of $\Mman$ fixed on some neighborhood of $\Xman$ (so each $\dif\in\DiffNbh(\Mman,\Xman)$ is fixed on some neighborhood $U_h$ of $X$ which depends on $h$), and put
\begin{align*}
    &\FolStabilizerNbh{\func,\Xman} :=\FolStabilizer{\func} \cap\DiffNbh(\Mman,\Xman), &
    &\StabilizerNbh{\func,\Xman} :=\Stabilizer{\func} \cap\DiffNbh(\Mman,\Xman).
\end{align*}
\begin{lemma}[{\rm\cite[Corollary~7.2]{Maksymenko:TA:2020}}]\label{lm:Snb_S}
Let $\Xman$ be an $\func$-adapted submanifold, and $\regU{\Xman}$ be an $\func$-adapted neighborhood of $\Xman$.
Then the following inclusions of pairs are homotopy equivalences:
\begin{align*}
    \bigl( \Stabilizer{\func,\regU{\Xman}}, \FolStabilizer{\func,\regU{\Xman}} \bigr)
    \subset
    \bigl( \StabilizerNbh{\func,\Xman},     \FolStabilizerNbh{\func,\Xman}  \bigr)
    \subset
    \bigl( \Stabilizer{\func,\Xman},        \FolStabilizer{\func,\Xman} \bigr).
\end{align*}
\end{lemma}
\begin{proof}[Remarks to the proof]
Though this statement is established in~\cite[Corollary~7.2]{Maksymenko:TA:2020} for orientable surfaces, the proof actually uses the fact that $\func$-regular neighborhoods of regular contours of $\func$ (being either connected components of $\Xman$ or of $\partial\Xman$) are always cylinders, and therefore they are orientable, see~\cite[Corollary~6.3]{MaksymenkoKuznietsova:PIGC:2019}.
\end{proof}

We will also need the following two simple statements which we leave to the reader.
The first lemma is a straightforward consequence of definitions, while the second one easily follows from axiom~\ref{axiom:Bd}.
\begin{lemma}\label{lm:h_in_FST}
Let $\func\in\FSp{\Mman}$, $\dif\in\Stabilizer{\func}$, and $\{\YYi{i}\}_{i\in\Lambda}$ be a family of $\func$-adapted subsurfaces invariant under $\dif$ and such that $\Mman=\cup_{i\in\Lambda}\YYi{i}$ (these subsurfaces may intersect each other).
Then $\dif\in\FolStabilizer{\func}$ if and only if $\dif|_{\YYi{i}} \in\FolStabilizer{\func|_{\YYi{i}}}$ for each $i\in\Lambda$.
\end{lemma}

\begin{lemma}\label{lm:h_Sidf__SidfX}
Let $\func\in\FSp{\Mman}$, $\Xman$ be any collection of boundary components of $\partial\Mman$, and $\regU{\Xman}$ be any $\func$-regular neighborhood of $\Xman$.
Let also $\dif\in\StabilizerId{\func}$ be a diffeomorphism fixed near $\Xman$.
Then $\dif$ is isotopic in $\StabilizerNbh{\func,\Xman}$ to a diffeomorphism supported in $\regU{\Xman}$.
\end{lemma}

\section{Functions on M\"obius band}\label{sect:decomp_B_in_Yi}
\subsection{Reduction to the case $\Pman=\bR$}
Let $\MBand$ be a M\"obius band and $\func\in\FSpS{\MBand}$ be a map into $\Circle$.
Since $\func$ takes constant value on $\partial\MBand$, it follows that $\func$ is null-homotopic.
Let $p:\bR\to\Circle$, $p(t) = e^{2\pi i t}$, be the universal covering map and $\sg{\func}:\MBand\to\bR$ be any  lifting of $\func$, so $\func = \sg{\func}\circ p$.
Then, see~\cite[Lemma~5.3(3)]{Maksymenko:TA:2020}, $\sg{\func}\in\FSpR{\MBand}$, $\func$ and $\sg{\func}$ have the same critical points and the same partitions into level sets, which also implies that $\Stabilizer{\func}=\Stabilizer{\sg{\func}}$.
Hence, for studying the stabilizer of $\func\in\FSpS{\MBand}$ we can replace $\func$ with $\sg{\func}$.

Therefore in what follows we will assume that $\func\in\FSpR{\MBand}$.

\subsection{Special $\func$-decomposition of $\MBand$}
The following lemma provides an additional (to statement~\ref{enum:th:prelim_statements:Sprf_product} of Theorem~\ref{th:prelim_statements}) decomposition of a M\"obius band $\MBand$ with respect to $\func$.

\begin{lemma}[{\rm\cite[Theorem~1.4]{MaksymenkoKuznietsova:PIGC:2019}}]
\label{lm:unique_cr_level}
Every $\func\in\FSpR{\MBand}$ has a unique critical contour $\CrComp$ with the following property:
there exists an $\func$-regular neighborhood $\regU{\CrComp} \subset \Int{\MBand}$ of $\CrComp$ invariant under $\STB{\MBand}$ and such that
\begin{enumerate}[label={\rm(\alph*)}, itemsep=0.5ex, topsep=0.5ex, leftmargin=*]
\item\label{enum:lm:unique_cr_level:y0-cyl}
the unique connected component $\YYi{0}$ of $\overline{\MBand\setminus\regU{\CrComp}}$ containing $\partial\MBand$ is a cylinder;
\item\label{enum:lm:unique_cr_level:yi-disks}
all other components $\YYi{1}, \ldots, \YYi{n}$ of $\overline{\MBand\setminus\regU{\CrComp}}$ are $2$-disks, see Figure~\ref{fig:cr_level_decomp}.
\end{enumerate}
Properties~\ref{enum:lm:unique_cr_level:y0-cyl} and~\ref{enum:lm:unique_cr_level:yi-disks} then hold for arbitrary $\STB{\MBand}$-invariant $\func$-regular neighborhood of $\CrComp$.
\end{lemma}
\begin{figure}[htbp]
\includegraphics[height=2.5cm]{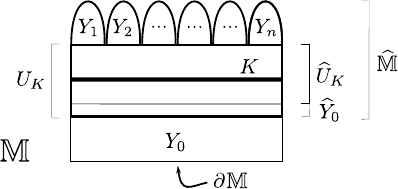}
\caption{Schematic decomposition of a M\"obius band associated with $\func\in\FSpR{\MBand}$}
\label{fig:cr_level_decomp}
\end{figure}

We will call $\CrComp$ a \term{special} contour of $\func$, and $\xi(\func):=\{\regU{\CrComp}, \YYi{0},\YYi{1},\ldots,\YYi{n}\}$ a \term{\spd{\func}} of $\MBand$ (associated with the $\STB{\MBand}$-invariant $\func$-regular neighborhood $\regU{\CrComp}$ of $\CrComp$).

For $i=0,1\ldots,n$, let $\CCi{i}$ be the connected component of $\MBand\setminus\CrComp$ containing $\YYi{i}$, and $\QQi{i}:= \overline{\CCi{i}}\setminus\fSing$.
One easily checks the following statements.
\begin{enumerate}[label={\rm(\roman*)}, leftmargin=*]
\item\label{enum:Qi_prop:Q0}
$\QQi{0}$ is diffeomorphic to $(\Circle\times[0;1]) \setminus (\Fman\times1)$, where $\Fman$ is a finite set.
So it is a closed cylinder out of finitely many points removed from one of its boundary components.
Another boundary component corresponds thus to $\partial\MBand$, see Figure~\ref{fig:cyl_cover}.

\item\label{enum:Qi_prop:Qi}
On the other hand, each $\QQi{i}$, $i=1,\ldots,n$, is a closed $2$-disk out of finitely many points removed from its boundary.

\item\label{enum:Qi_prop:bdedges}
Moreover, for each $i=0,\ldots,n$, the intersection $\QQi{i}\cap\CrComp$ is a finite collection of arcs (edges of $\CrComp$) which also constitute a cycle.
We will call them \term{boundary edges of $\QQi{i}$}.

\item\label{enum:Qi_prop:2comps}
Each edge of $\CrComp$ belongs to exactly two distinct components $\QQi{i}$ and $\QQi{j}$ one of which is ``upper'' and another one ``lower'' in the sense that the images $\func(\CCi{i})$ and $\func(\CCi{j})$ are contained in distinct components of $\bR\setminus\func(\CrComp)$.
\end{enumerate}

\subsection{Quasi-cones}\label{sect:quasi-cones}
Let $C\Circle = \bigl(\Circle\times[0;1]\bigr)/\bigl(\Circle\times0\bigr)$ be a cone over a circle, i.e.\ just a $2$-disk.
The equivalence class $[\Circle\times0]$ is called the \term{vertex} of the cone $C\Circle$.
Let also $F_1\sqcup\cdots\sqcup F_k\subset \Circle\times1$ be a finite collection of mutually disjoint finite subsets.
Shrink each $F_i$ into a single point.
Then the corresponding quotient space $\Lman:=C\Circle / (F_1\sqcup\cdots\sqcup F_k)$ will be called a \term{quasi-cone} over $\Circle$.
Let $p:C\Circle \to \Lman$ be the respective quotient map.
Then $p(\Circle\times1)$ will be called the \term{base} of the quasi-cone $\Lman$ and denoted by $b\Lman$.

One easily checks that if $\dif:\Lman_1\to\Lman_2$ is a homeomorphism between quasi-cones over $\Circle$, then it lifts to a unique  homeomorphism $\hat{\dif}:C\Circle\to C\Circle$ such that $\dif\circ p_1 = p_2\circ \hat{\dif}$, where $p_i:C\Circle\to\Lman_i$, $i=1,2$, is the corresponding quotient map.
Moreover, as $\hat{\dif}(\Circle\times 1) = \Circle\times 1$, we have a homeomorphism $\hat{\dif}_1:\Circle\to\Circle$ such that $\hat{\dif}(\px,1) = (\hat{\dif}_1(\px),1)$ for all $\px\in\Circle$.
Hence, one can also define another homeomorphism $\hat{\dif}':C\Circle \to C\Circle$, $\hat{\dif}'(\px,t) = (\hat{\dif}_1(\px),t)$, called a \term{cone over $\hat{\dif}_1$}.
It coincides with $\dif$ on the bases and sends the vertex of $\Lman_1$ to the vertex of $\Lman_2$.
Moreover, $\hat{\dif}'$ induces a homeomorphism $\dif':\Lman_1\to\Lman_2$ which coincides with $\dif$ on the base $b\Lman_1$, and is evidently determined by $\dif:b\Lman_1\to b\Lman_2$.
We will call $\dif'$ the ``\term{cone change}'' of $\dif$.

\subsection{CW-decompositions of $S^2$ and $\PrjPlane$ associated with $\func$}\label{sect:CW_decomp_RP2_of_f}
Let $\hMBand = \MBand/\partial\MBand$ be a surface obtained by shrinking the boundary $\partial\MBand$ into one point which we will denote by $\bdpt$.
Then $\hMBand$ is a projective plane, and we get a CW-partition $\PrjCWPart$ of $\hMBand$ whose
\begin{itemize}
\item $0$-cells are critical points of $\func$ belonging to $\CrComp$;
\item $1$-cells are connected components of $\CrComp\setminus\fSing$;
\item $2$-cells are connected components $\CCi{0}/\partial\MBand, \,\CCi{1},\,\ldots,\,\CCi{n}$ of $\hMBand\setminus\CrComp$.
\end{itemize}
Thus, we can regard $\hMBand$ as a one point compactification of $\Int{\MBand}$, and so that $\hMBand = \Int{\MBand} \sqcup \{\bdpt\}$.
Denote by $c_i$, $i=0,1,2$, the total number of $i$-cells of $\PrjCWPart$.
In particular, $c_2=n+1$.

Moreover, let $p:S^2\to\hMBand$ be the universal cover and $\xi:S^2\to S^2$ be the reversing orientation involution without fixed points generating the group $\bZ_2$ of covering transformations.
Then $S^2$ has a CW-structure $\SphCWPart$ whose $i$-cells are connected components of the inverses of $i$-cells of $\PrjCWPart$.
In fact, for each $i$-cell $\cell\in\PrjCWPart$ its inverse image $p^{-1}(e)$ is a disjoint union of two $i$-cells $\tcell_{0}$ and $\tcell_{1}$ such that $\xi$ exchanges them.
We will call them \term{liftings} of $\cell$.

Notice that it follows from~\ref{enum:Qi_prop:Q0} and~\ref{enum:Qi_prop:Qi} above that the closure of each $2$-cell of $\PrjCWPart$ is a quasi-cone over $\Circle$.
We can also assume that $x^{*}$ is the vertex of $\CCi{0}/\partial\MBand$.

\subsection{Lefschetz number}
Due to Lemma~\ref{lm:unique_cr_level} each $\dif\in\STB{\MBand}$ yields a cellular (i.e.\ sending cells to cells) homeomorphism $\pdif:\hMBand\to\hMBand$ which fixes $\bdpt$ and coincides with $\dif$ on $\Int{\MBand}$.
Thus, if $\cell$ is an $i$-cell of $\PrjCWPart$, then $\pdif(\cell)$ is also an $i$-cell of $\PrjCWPart$.
Then for $i=1,2$ one can take to account whether the restriction map $\pdif|_{\cell}: \cell \to \pdif(\cell)$ preserves or reverses chosen orientations.
For the case $i=0$ we will always assume that $\pdif|_{\cell}: \cell \to \pdif(\cell)$ preserves the orientation.

We will say that $\cell$ is $\pdif^{+}$-invariant (resp. $\pdif^{-}$-invariant) if $\pdif(\cell)=\cell$ and the restriction map $\pdif:\cell\to \cell$ preserves (resp. reverses) orientation.
Denote by $\cplus{i}(\pdif)$ and $\cmin{i}(\pdif)$ the number of $\pdif^{+}$-invariant (resp. $\pdif^{-}$-invariant) cells of $\PrjCWPart$.
Then $\cmin{0}(\pdif)=0$.

Now fix some orientation of each cell.
This allows to define the group $C_i(\hMBand,\bZ)$, $i=0,1,2$, of integral chains of the CW-partition $\PrjCWPart$, being just a free abelian group whose basis consists of $i$-cells.
Let also $\pdif_i:C_i(\hMBand,\bZ)\to C_i(\hMBand,\bZ)$ be the induced automorphism of $C_i(\hMBand,\bZ)$.
Since $\pdif$ is cellular, $\pdif_i$ is given by some non-degenerate $(c_i\times c_i)$-matrix $\Amatri{i}$ and each column and each row of $\Amatri{i}$ contains only one non-zero being either $+1$ or $-1$.
Notice also that non-zero diagonal elements of $\Amatri{i}$ correspond to invariant $i$-cells under $\pdif$.
Moreover, such an element equals $+1$ (resp.\ $-1$) is $\pdif$ preserves (resp. reverses) orientation of the invariant cell.
In particular, due to the above convention, $\Amatri{0}$ may consist of $0$ and $1$.
This implies that $\trace(\Amatri{i}) = \cplus{i}(\pdif) - \cmin{i}(\pdif)$, $i=0,1,2$.

It is also well known that the following number $L(\pdif) = \trace(\Amatri{0}) - \trace(\Amatri{1}) + \trace(\Amatri{2})$, called the \term{Lefschetz number} of $\pdif$, depends in fact only on the homotopy class of $\pdif$.
In particular, since $\dif$ is isotopic to $\id_{\MBand}$ relatively $\partial\MBand$, it follows that $\pdif$ is isotopic to $\id_{\hMBand}$, whence
\[
    L(\pdif) = L(\id_{\hMBand}) = c_0 - c_1 + c_2 = \chi(\hMBand) = 1.
\]
Thus, we finally get the following identity:
\begin{equation}\label{equ:Lefschetz_number_of_h}
    1  =  c_0 - c_1 + c_2
       = \cplus{0}(\pdif) -
         \bigl( \cplus{1}(\pdif) - \cmin{1}(\pdif) \bigr) +
         \bigl( \cplus{2}(\pdif) - \cmin{2}(\pdif) \bigr).
\end{equation}
Also note that since $\pdif$ is homotopic to $\id_{\hMBand}$ there exists a unique lifting $\spdif:S^2\to S^2$ such that
\begin{itemize}
\item $\spdif$ is isotopic to $\id_{S^2}$ and thus preserves orientation of $S^2$;
\item $\spdif$ is also $\SphCWPart$-cellular.
\end{itemize}

\begin{lemma}\label{lm:inv_cells__pdif__spdif}
Let $\cell\in\PrjCWPart$ be an $i$-cell with $i=1,2$, and $\tcell_0,\tcell_1\in\SphCWPart$ be its liftings.
If $\cell$ is $\pdif^{+}$-invariant, then both $\tcell_{0}$ and $\tcell_{1}$ are $\spdif^{+}$-invaraint.
If $\cell$ is $\pdif^{-}$-invariant, then $\spdif$ exchanges $\tcell_{0}$ and $\tcell_{1}$, so they are not invariant under $\spdif$.
In particular,
\begin{align*}
\cplus{i}(\spdif) &= 2 \cplus{i}(\pdif), &
\cmin{i}(\spdif)  &= 0
\end{align*}
and therefore we have the following identity:
\begin{equation}\label{equ:Lefschetz_number_of_spdif}
    L(\spdif)
        = \chi(S^2)
        = 2
        = \cplus{0}(\spdif) - 2\cplus{1}(\pdif) + 2\cplus{2}(\pdif).
\end{equation}
\end{lemma}
\begin{proof}
The statements about $\spdif^{\pm}$-invariantness follow from the fact that $\xi$ exchanges $\tcell_{0}$ and $\tcell_{1}$, and $\xi:\tcell_{k} \to \tcell_{1-k}$, $k=0,1$, reverses orientation.
The last formula is a variant of the Lefschetz formula~\eqref{equ:Lefschetz_number_of_h} for $\spdif$.
We leave the details for the reader.
\end{proof}

\subsection{Epimorphism $\eta\colon\PSTB{\MBand}\to\bZ$}
Let $\tQman_{0} = (\bR\times[0;1]) \setminus (\bZ\times1)$, and $J_k = (k;k+1)\times 1$ for $k\in\bZ$.
Then we have the universal covering map $p:\tQman_{0}\to\QQi{0}$, and hence for each $\dif\in\STB{\MBand}$ there exists a unique lifting $\pdif:\tQman_{0}\to\tQman_{0}$ of the restriction $\dif|_{\QQi{0}}$ fixed on $\bR\times0$.
Evidently, $\pdif$ should shift the intervals $J_k$, so there exists unique integer $\sg{\eta}(\dif)\in\bZ$ such that $\pdif(J_k)=J_{k+\sg{\eta}(\dif)}$ for all $k\in\bZ$.
One easily checks that the correspondence $\dif\mapsto\sg{\eta}(\dif)$ is a well-defined homomorphism $\sg{\eta}:\STB{\MBand} \to \bZ$.

\begin{figure}[htbp!]
\includegraphics[width=0.8\textwidth]{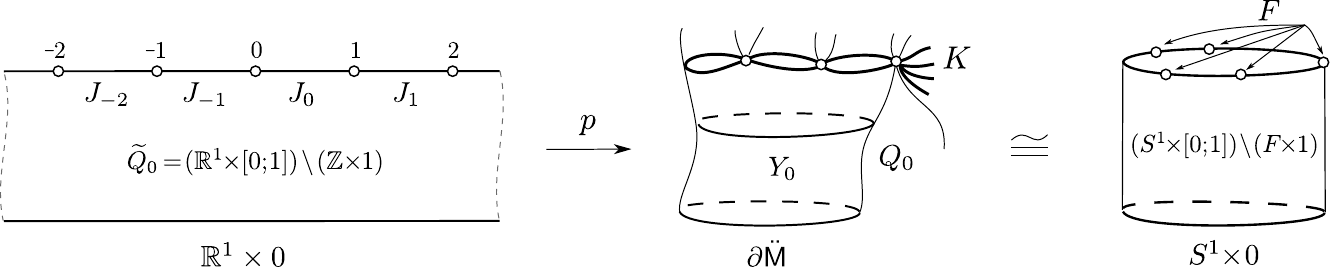}
\caption{}\label{fig:cyl_cover}
\end{figure}

Denote also by $c$ the total number of intervals in $\QQi{0}\cap\CrComp$ which is the same as the number of points in $\Fman$.
Moreover, let $\tau\in\STB{\MBand}$ be a Dehn twist along $\partial\MBand$ supported in the cylinder $\YYi{0} \subset \Cman_0$.
Thus, by definition, $\tau$ is a ``one rotation along $\partial\MBand$'', which means that $\sg{\eta}(\tau) = \pm c$.
Replacing $\tau$ with $\tau^{-1}$ we can assume that $\sg{\eta}(\tau) = c$.
In particular, $\sg{\eta}$ is a non-trivial homomorphism, and thus $\sg{\eta}(\STB{\MBand}) = a\bZ$ for some $a\geq1$.

Hence, we get an epimorphism
\[
    \eta:\STB{\MBand}\to\bZ,
    \qquad
    \eta(\dif) = \sg{\eta}(\dif)/a.
\]
Notice also that the value $\eta(\dif)$ depends only on the isotopy class of $\dif$ in $\PSTB{\MBand}$, and therefore $\eta$ factors to the desired epimorphism $\eta:\PSTB{\MBand}\to\bZ$, which (for simplicity) we will denote by the same letter $\eta$.
Put
\begin{equation}\label{equ:eta_tau_b}
    \sgOrd := \eta(\tau) = c/a.
\end{equation}

Thus, roughly speaking, every $\dif\in\STB{\MBand}$ cyclically shifts the boundary edges of $\QQi{0}$, so that all shifts are integer multiples of $a$, and the period of that shift is $c$.

\subsection{Action of $\STB{\MBand}$ on disks $\YYi{i}$}
Let $\CompSet=\{\YYi{1}, \ldots, \YYi{n}\}$ be the family of $2$-disks from the \spd{\func}\ of $\MBand$.
Since $\mathop{\cup}\limits_{k=1}^{n}\YYi{k}$ is invariant with respect to $\STB{\MBand}$, we have a natural action of $\STB{\MBand}$ on $\CompSet$ by permutations.

Let us fix some orientation of each $\YYi{k}$, $k=1,\ldots,n$, and put $\PlCompSet = \CompSet \times\{\pm 1\}$.
Then the action of $\STB{\MBand}$ on $\CompSet$ extends further to an action on $\PlCompSet$ defined by the following rule:
\begin{enumerate}[label={$(*)$}, leftmargin=*]
\item\label{stab-action}
\em if $\dif\in\STB{\MBand}$ and $\YYi{k}\in\CompSet$, then $\dif(\YYi{k},+1)=(\dif(\YYi{k}),\delta)$ and $\dif(\YYi{k},-1)=(\dif(\YYi{k}),-\delta)$, where
\begin{equation*}
	\delta=
	\begin{cases}
		+1, & \text{if the restriction $\dif|_{\YYi{k}}: \YYi{k} \to \dif(\YYi{k})$ preserves orientation}, \\
		-1, & \text{otherwise}.
	\end{cases}
\end{equation*}
\end{enumerate}
Let also $\KerSAct$ be the kernel of non-effectiveness of that action of $\STB{\MBand}$ on $\PlCompSet$, i.e.\ $\KerSAct$ consists of diffeomorphisms $\dif$ such that $\dif(\YYi{i})=\YYi{i}$, for all $i=1,\ldots,n$, and $\dif$ also preserves orientation of $\YYi{i}$.

\begin{lemma}\label{lm:charKerSAct}
Let $\dif\in\STB{\MBand}$.
Then the following statements are equivalent:
\begin{enumerate}[label={\rm(\alph*)}, leftmargin=*]
\item\label{enum:lm:charKerSAct:h_in_ker} $\dif\in\KerSAct$, in other words $\dif$ preserves each $\YYi{k}$ with its orientation;
\item\label{enum:lm:charKerSAct:h_pres_Yi} $\dif$ preserves each connected component of $\MBand\setminus\CrComp$ with its orientation, i.e.\ $\cplus{2}(\pdif)=n+1$ and $\cmin{2}(\pdif)=0$;
\item\label{enum:lm:charKerSAct:h_pres_one_Yi} $\dif$ preserves at least one connected component of $\MBand\setminus\CrComp$ distinct from $\CCi{0}$ with its orientation, i.e.\ $\cplus{2}(\pdif)\geq2$;
\item\label{enum:lm:charKerSAct:h_pres_all} $\dif$ each preserves each cell of $\PrjCWPart$ with its orientation;
\item\label{enum:lm:charKerSAct:h_pres_edges} $\dif$ preserves each edge of $\CrComp$ with its orientation, i.e.\ $\cplus{1}(\pdif)=c_1$ and $\cmin{1}(\pdif)=0$;
\item\label{enum:lm:charKerSAct:eta_h__b}
$\dif\in\eta^{-1}(b\bZ)$, in other words $\dif$ preserves each boundary edge of $\QQi{0}$ with its orientation, see~\eqref{equ:eta_tau_b};
\item\label{enum:lm:charKerSAct:h_pres_one_edge} $\dif$ preserves at least one edge of $\CrComp$ with its orientation, i.e.\ $\cplus{1}(\pdif)>0$.
\end{enumerate}
In particular,
\begin{itemize}[leftmargin=*, label=$- $]
\item
\ref{enum:lm:charKerSAct:eta_h__b} implies that
$\frac{\STB{\MBand}}{\KerSAct} \cong \frac{\eta(\STB{\MBand})}{\eta(\KerSAct)} = \bZ / \sgOrd\bZ \cong \bZ_{\sgOrd}$;
\item the equivalence~\ref{enum:lm:charKerSAct:h_in_ker}$\Leftrightarrow$\ref{enum:lm:charKerSAct:h_pres_Yi} implies that the action of $\STB{\MBand}/\KerSAct \cong\bZ_{\sgOrd}$ on $\PlCompSet$ is \term{free};
\item
\ref{enum:lm:charKerSAct:h_pres_all} implies that two diffeomorphisms $\dif_1,\dif_2\in\STB{\MBand}$ induce the same permutation of cells of $\PrjCWPart$ if and only if $\dif_1^{-1}\circ\dif_2 \in\KerSAct$, i.e.\ they belong to the same adjacent class $\STB{\MBand}/\KerSAct$.
\end{itemize}
\end{lemma}
\begin{proof}
The equivalence~\ref{enum:lm:charKerSAct:h_in_ker}$\Leftrightarrow$\ref{enum:lm:charKerSAct:h_pres_Yi} follows from the equality $\dif(\CrComp)=\CrComp$, while the implications \ref{enum:lm:charKerSAct:h_pres_Yi}$\Rightarrow$\ref{enum:lm:charKerSAct:h_pres_one_Yi} and \ref{enum:lm:charKerSAct:h_pres_all}$\Rightarrow$\ref{enum:lm:charKerSAct:h_pres_edges}$\Rightarrow$\ref{enum:lm:charKerSAct:eta_h__b}$\Rightarrow$\ref{enum:lm:charKerSAct:h_pres_one_edge} are evident.

The implication~\ref{enum:lm:charKerSAct:h_pres_one_edge}$\Rightarrow$\ref{enum:lm:charKerSAct:h_pres_edges} can be found in~\cite[Claim~7.1.1]{Maksymenko:AGAG:2006}.
Let us recall the arguments.
Since $\dif$ preserves some edge of $\CrComp$ with its orientation, $\dif$ also preserves the closure of this edge with its orientation, i.e.\ $\dif$ fixes the vertices on the boundary of this edge.
But $\dif$ also preserves the cyclic order of edges incident to both of its ends, and therefore $\dif$ should preserve all these edges with their orientations.
It now follows from the connectedness of $\CrComp$ that $\dif$ preserves all edges of $\CrComp$ with their orientations.

\ref{enum:lm:charKerSAct:h_pres_edges}$\Rightarrow$\ref{enum:lm:charKerSAct:h_pres_Yi}
Let $\CCi{i}$ be a connected component of $\MBand\setminus\CrComp$, and $e \subset\CrComp$ be a boundary edge of $\QQi{i} = \overline{\CCi{i}}\setminus\fSing$.
Let also $\CCi{j}$ be another component of $\MBand\setminus\CrComp$ such that $e\subset\QQi{j}$.
By~\ref{enum:lm:charKerSAct:h_pres_edges}, $\dif(e)=e$ whence $\dif$ should leave invariant $\QQi{i} \cup \QQi{j}$, and therefore the set $\CCi{i} \cup \CCi{j} = (\QQi{i} \cup \QQi{j})\setminus\CrComp$.
As noted above in~\ref{enum:Qi_prop:2comps}, one can assume that $\func(\CCi{i}) \subset \bigl(-\infty;\func(\CrComp)\bigr)$ and $\func(\CCi{j}) \subset \bigl(\func(\CrComp);+\infty\bigr)$.
Since $\func\circ\dif=\func$, so $\dif$ preserves level sets of $\func$, we see that $\dif$ can not interchange $\CCi{i}$ with $\CCi{j}$, and thus $\dif(\CCi{i})=\CCi{i}$.
Moreover, by~\ref{enum:lm:charKerSAct:h_pres_edges}, $\dif$ also preserves orientation of $e$, whence it should also preserve orientation of $\CCi{i}$.

\ref{enum:lm:charKerSAct:h_pres_Yi}$\Rightarrow$\ref{enum:lm:charKerSAct:h_pres_one_edge}
Suppose $\cplus{2}(\pdif) = c_2 = n+1$ and thus $\cmin{2}(\pdif)=0$.
Then by formula~\eqref{equ:Lefschetz_number_of_h} we have that
\[
1 = \cplus{0}(\pdif) -  \bigl( \cplus{1}(\pdif)- \cmin{1}(\pdif) \bigr) +  n+1,
\]
whence $\cplus{1}(\pdif) = \cplus{0}(\pdif)  + \cmin{1}(\pdif) + n > 0$.

\ref{enum:lm:charKerSAct:h_pres_one_Yi}$\Rightarrow$\ref{enum:lm:charKerSAct:h_pres_one_edge}
Suppose $\cplus{2}(\pdif)\geq2$ but~\ref{enum:lm:charKerSAct:h_pres_one_edge} fails, so $\cplus{1}(\pdif)=0$.
Then by Lemma~\ref{lm:inv_cells__pdif__spdif}, we have that $\cplus{2}(\spdif)=2\cplus{2}(\pdif)\geq4$, $\cplus{1}(\spdif)=2\cplus{1}(\pdif) = 0$, whence by~\eqref{equ:Lefschetz_number_of_spdif}
\[
2 = \cplus{0}(\spdif) - \cplus{1}(\spdif) + \cplus{2}(\spdif)
    = \cplus{0}(\spdif) + \cplus{2}(\spdif)
    \geq  \cplus{2}(\spdif) \geq 4,
\]
which is impossible.
Hence, \ref{enum:lm:charKerSAct:h_pres_one_edge} should hold.

\ref{enum:lm:charKerSAct:h_pres_Yi}\&\ref{enum:lm:charKerSAct:h_pres_edges}\,$\Rightarrow$\,\ref{enum:lm:charKerSAct:h_pres_all}
By assumption, all $1$- and $2$-cells are $\dif^{+}$-invaraint.
As each $0$-cell $v$ of $\CrComp$ belongs to the closure of some $1$-cell $e$, and $\dif$ preserves orientation of $e$, we see that $\dif$ also fixes $v$.
\end{proof}

\subsection{Disks of types \ref{dt1} and \ref{dt2}}
By Lemma~\ref{lm:charKerSAct} we have an effective free action of the quotient $\frac{\STB{\MBand}}{\KerSAct} \cong \frac{\PSTB{\MBand}}{\pi_0\KerSAct} \cong  \bZ_{\sgOrd}$ on $\PlCompSet$, so we can regard it as a cyclic subgroup of the permutation group of $\PlCompSet$.

Let $\gdif\in\PSTB{\MBand}$ be any element with $\eta(\gdif)=1$.
Then its class in the quotient $\bZ_{\sgOrd}$ is a generator of that group, so it is represented by some bijection $\sg{\gdif}:\PlCompSet\to\PlCompSet$.

Hence, $\sg{\gdif}^{\sgOrd}=\id_{\PlCompSet}$, $\sg{\gdif}^k$ for $k=1,\ldots,\sgOrd-1$ has no fixed points, and every orbit of the action of $\bZ_{\sgOrd}$ consists of $\sgOrd$ elements, so, in particular, $\sgOrd$ divides $2n$ (the order of $\PlCompSet$).

\begin{lemma}\label{lm:action_g_on_hY}
\begin{enumerate}[label={\rm(\alph*)}, leftmargin=*]
\item\label{enum:lm:action_g_on_hY:star}
If $\sg{\gdif}(\YYi{i},1)=(\YYi{k},\delta)$ for some $i,k\in\{1,\ldots,n\}$ and $\delta\in\{\pm1\}$, then $\sg{\gdif}(\YYi{i},-1)=(\YYi{k},-\delta)$.

\item\label{enum:lm:action_g_on_hY:disk_types}
Let $(\YYi{i_0},\delta_{i_0})\in\PlCompSet$, and $(\YYi{i_k},\delta_{i_k}) = \sg{\gdif}^k(\YYi{i_0},\delta_{i_0})$, $k=0,\ldots,\sgOrd-1$, be all elements of its orbit.
Then exactly one of the following two possibilities holds:
\begin{enumerate}[leftmargin=*, label={\rm(T\arabic*)}]
\item\label{dt1}
either all disks $\YYi{i_0},\ldots,\YYi{i_{\sgOrd-1}}$ are mutually distinct,
\item\label{dt2}
$\sgOrd$ is even, the disks $\YYi{i_0},\ldots,\YYi{i_{\mel-1}}$ are mutually distinct, where $\mel=\sgOrd/2$, and $(\YYi{i_{\mel}},\delta_{i_{\mel}}) = (\YYi{i_0},-\delta_{i_0})$, so, due to~\ref{enum:lm:action_g_on_hY:star},
\[
\begin{matrix}
    (\YYi{i_0}, \delta_{i_0}),  & (\YYi{i_1}, \delta_{i_1}),  & \ldots, & (\YYi{i_{\mel-1}}, \delta_{i_{\mel-1}}), \\
    (\YYi{i_0},-\delta_{i_0}),  & (\YYi{i_1},-\delta_{i_1}),  & \ldots, & (\YYi{i_{\mel-1}},-\delta_{i_{\mel-1}}),
\end{matrix}
\]
are consecutive elements of the orbit of $(\YYi{i_0},\delta_{i_0})$.
\end{enumerate}

\item\label{enum:lm:action_g_on_hY:conj_by_g}
Suppose $\gdif(\YYi{i})=\YYi{j}$ for some $i,j$.
Then we have the following isomorphism:
\[
    \PST{\YYi{i}} \to \PST{\YYi{j}},
    \qquad
    t \mapsto \gdif t \gdif^{-1},
\]
which sends $\PFST{\YYi{i}}$ onto $\PFST{\YYi{j}}$.
\end{enumerate}
\end{lemma}
\begin{proof}
Statement~\ref{enum:lm:action_g_on_hY:star} is just the same as~\ref{stab-action}, and~\ref{enum:lm:action_g_on_hY:conj_by_g} is evident.

\ref{enum:lm:action_g_on_hY:disk_types}
Suppose~\ref{dt1} fails, so $\gdif^{k}(\YYi{i_0})=\YYi{i_k}=\YYi{i_l}=\gdif^{l}(\YYi{i_0})$ for some $k<l \in\{0,1,\ldots,\sgOrd-1\}$.
Denote $\mel = l-k$.
Then $\YYi{i_0} = \gdif^{\mel}(\YYi{i_0}) = \YYi{i_u}$.
Since the action of $\sg{\gdif}$ on $\PlCompSet$ is free and $\mel<\sgOrd$, we should have that $(\YYi{i_{0}},\delta_{i_{0}})\not=(\YYi{i_{\mel}},\delta_{i_{\mel}})=(\YYi{i_{0}},\delta_{i_{\mel}})$, whence $\delta_{i_0} = - \delta_{i_{\mel}}$.
But then
\[
    \sg{\gdif}^{2\mel}(\YYi{i_0},\delta_{i_0}) =
    \sg{\gdif}^{\mel}(\YYi{i_0},-\delta_{i_0})
    \stackrel{\ref{enum:lm:action_g_on_hY:star}}{=} (\YYi{i_0},\delta_{i_0}),
\]
i.e.\ $\sg{\gdif}^{2\mel}$ has a fixed element $(\YYi{i_0},\delta_{i_0})$, and therefore $2\mel$ must divide $\sgOrd$.
Since $\mel<\sgOrd$, we should have that $2\mel=\sgOrd$.
\end{proof}

Thus, in the case~\ref{dt1}, if $\dif(\Yman_{i_k})=\Yman_{i_k}$ for some $\dif\in\STB{\MBand}$ and $k=0,\ldots,\sgOrd-1$, then $\dif$ must preserve orientation of $\Yman_{i_k}$.
On the other hand, in the case~\ref{dt2} for each $k$ there exists $\dif\in\STB{\MBand}$ such that $\dif(\Yman_{i_k})=\Yman_{i_k}$ and $\dif$ reverses orientation of $\Yman_{i_k}$.

It will be convenient to say that a pair $(\Yman_{i_0},\delta_{0})$ as well as its orbit is of type~\ref{dt1} or~\ref{dt2} depending on the corresponding cases.
Moreover, since $\sg{\gdif}(\Yman_{i_0},\delta_{0})$ and $(\Yman_{i_0},-\delta_{0})$ have the same type, we can even say that the disk $\Yman_{i_0}$ itself has the corresponding type.

\begin{lemma}\label{lm:b_is_odd}
Suppose all disks $\YYi{1},\ldots,\YYi{n}$ are of type~\ref{dt1}.
There exists a $\PrjCWPart$-cellular homeomorphism $\qdif:\hMBand\to\hMBand$ of order $\sgOrd$ such that each of its iterations $\qdif^{i}$, $i=1,\ldots,\sgOrd-1$, has a unique fixed point $x^{*}$.
Hence, $\qdif$ yields a free action of $\bZ_{\sgOrd}$ on $\hMBand\setminus x^{*} = \Int{\MBand}$, and this implies that $\sgOrd$ must be odd.
\end{lemma}
\begin{proof}
{\bf Construction of $\qdif$.}
Let $\gdif\in\STB{\MBand}$ be any generator of the group $\bZ_{\sgOrd} = \STB{\MBand}/\KerSAct$.
Then $\gdif^{i}\not\in\KerSAct$ for $i=1,\ldots,\sgOrd-1$, whence by
Lemma~\ref{lm:charKerSAct}\ref{enum:lm:charKerSAct:h_pres_all}, $\gdif^{i}$ has no $(\gdif^{i})^{+}$-invariant cells except for $\CCi{0}$.

We will define $\qdif$ so that it will interchange the cells of $\PrjCWPart$ in the same way as $\gdif$.
Let us also mention that by Lemma~\ref{lm:charKerSAct}, all representatives of the same adjacent class in $\STB{\MBand}/\KerSAct$ induce the same permutations of cells of $\PrjCWPart$ and map those cells with the same orientations.
In particular, $\qdif$ will actually depend only on the class $[\gdif]\in\bZ_{\sgOrd}$.
The proof is close to the construction described in~\cite[Theorem~2.2]{Feshchenko:MFAT:2016}, and we will just sketch the arguments.
\begin{enumerate}[label={\rm(\arabic*)}, start=0, leftmargin=*]
\item
For every $0$-cell $\px$ of CW-partition $\PrjCWPart$, that is a critical point of $\func$ belonging to $\CrComp$, we set $\qdif(\px) = \gdif(\px)$.
\item
Further, choose a metric on $\CrComp$ such that each edge has length $1$.
Now, if $e$ is an oriented edge of $\CrComp$ and $e' = \gdif(e)$ is its image, then we define $\qdif|_{e}: e \to e'$ to be a unique isometry which preserves (reverses) orientations in accordance with $\gdif:e\to e'$.
This gives a free isometric action of $\qdif$ on $\CrComp$.
\item
Finally, if $\Lman_i$ and $\Lman_j$ are closures of $2$-cells of $\PrjCWPart$ and $\gdif(\Lman_i)=\Lman_j$, then we define $\qdif:\Lman_i\to\Lman_j$ as the cone change $\qdif:=\gdif'$ of $\gdif$, see Section~\ref{sect:quasi-cones}.
We can also assume that $x^{*} = \hMBand\setminus\Int{\MBand}$ is the vertex of $\CCi{0}$, and therefore $\qdif(x^{*})=x^{*}$.
\end{enumerate}
Due to Lemma~\ref{lm:charKerSAct}\ref{enum:lm:charKerSAct:h_pres_all}, each iteration $\qdif^{i}$, $i=1,\ldots,\sgOrd-1$, of $\qdif$ has a unique fixed point $x^{*}$.
Moreover, since $\qdif$ is a cone change, the fixed point $x^{*}$ has a $\qdif$-invariant small $2$-disk neighborhood $\Uman$, whence the M\"obius band $\MBand:=\overline{\hMBand\setminus\Uman}$ is an invariant under $\qdif$.
Then the action of $\bZ_{\sgOrd}$ on $\MBand$ is free, and so the induced quotient map $p:\MBand\to\MBand/\qdif$ is a $\sgOrd$-sheeted covering.

{\bf Proof that $\sgOrd$ is odd.}
Suppose, in general, that we have a $\sgOrd$-sheeted covering map $p:\MBand\to\Aman$.
Then $\Aman$ must be a non-orientable surface with one boundary component, while its fundamental group $\pi_1\Aman$ contains a free abelian subgroup $p(\pi_1\MBand) \cong \bZ$ of finite index $\sgOrd$.
Hence, $\Aman$ is a M\"obius band as well, and therefore it has an orientable double cover $\lambda:\Circle\times[0;1]\to\Aman$.
Now if $\sgOrd$ is even, then
\[
    p(\pi_1\MBand) = \sgOrd\bZ \subset 2\bZ = \lambda\bigl(\pi_1(\Circle\times[0;1])\bigr) \subset \pi_1\Aman,
\]
and by theorem on existence of lifting, there exists a map $\sg{p}:\MBand\to\Circle\times[0;1]$ such that $p = \lambda\circ\sg{p}$.
Since $\lambda$ and $p$ are coverings, $\sg{p}$ must be a covering as well, which is impossible.
Hence, $\sgOrd$ should be odd.
\end{proof}

\section{Proof of Theorem~\ref{th:class:Mobius_band}}\label{sect:proof:th:class:Mobius_band}
Let $\func\in\FSpR{\MBand}$.
We should prove that there exist groups $\Agrp, \Ggrp, \Hgrp \in \classGroups$, an automorphism $\gamma:\Hgrp\to\Hgrp$ with $\gamma^{2}=\id_{\Hgrp}$, and $\mel\geq1$ such that
\begin{equation}\label{equ:pi0Sprf_A_GHZ__MBand:in_the_proof}
    \PSTB{\MBand} \ \equiv \ \pi_0\StabilizerIsotId{\func,\partial\MBand} \ \cong \ \pi_1\Orbit{\func} \ \cong \ \Agrp \ \times  \ \WrGHZ{\Ggrp}{\Hgrp}{\gamma}{\mel}.
\end{equation}

\subsection{Eliminating the group $\Agrp$}
First we reduce the problem to the situation when $\Agrp=\UnitGroup$.
Put
\[
    \sg{\MBand} \ := \  \overline{\MBand\setminus\YYi{0}}
                \ =  \  \regU{\CrComp} \,\cup\, \mathop{\cup}\limits_{i=1}^{n}\YYi{i},
\]
see Figure~\ref{fig:cr_level_decomp}.
Evidently, $\sg{\MBand}$ is still a M\"obius band and $\partial\sg{\MBand}$ is a contour of $\func$, whence $\sg{\MBand}$ is $\func$-adapted subsurface, and therefore the restriction $\sg{\func}:=\func|_{\sg{\MBand}}$ belongs to $\FSpR{\sg{\MBand}}$.
The following lemma is easy.
It can be seen from Figure~\ref{fig:cr_level_decomp}, and we leave it for the reader:
\begin{lemma}\label{lm:reduct_sgB}
Let $\xi(\sg{\func})=\{\sg{\Uman}_{\sg{\CrComp}}, \sg{\Yman}_0,\sg{\Yman}_1,\ldots,\sg{\Yman}_{\sg{n}}\}$ be any \spd{\sg{\func}}\ of $\sg{\MBand}$.
Then $\CrComp = \sg{\CrComp}$, $\xi(\sg{\func})$ contains the same number of $2$-disks, i.e.\ $n = \sg{n}$, while the cylinder $\sg{\Yman}_0$ contains no critical points of $\sg{\func}$.
\end{lemma}

Notice that the regular contour $\partial\sg{\MBand} = \sg{\MBand} \cap \YYi{0}$ of $\func$
\begin{itemize}[label={$- $}]
\item cuts $\MBand$ into two subsurfaces $\YYi{0}$ and $\sg{\MBand}$, one of which (namely $\YYi{0}$) is a cylinder;
\item and is also invariant under $\Stabilizer{\func}$, since by Lemma~\ref{lm:unique_cr_level}, $\YYi{0}$ and $\sg{\MBand}$ are $\Stabilizer{\func}$-invariant.
\end{itemize}
Then, due to~\cite[Theorem~5.5(3)]{Maksymenko:TA:2020}, the latter two properties imply that the following homotopy equivalence holds:
\[
    \StabilizerIsotId{\func,\partial\MBand}
    \simeq
    \StabilizerIsotId{\func|_{\sg{\MBand}},\partial\sg{\MBand}}
    \times
    \StabilizerIsotId{\func|_{\YYi{0}},\partial\YYi{0}},
\]
whence we get an isomorphism of the corresponding $\pi_0$-groups:
\[
    \pi_0\StabilizerIsotId{\func,\partial\MBand}
    \cong
    \pi_0\StabilizerIsotId{\func|_{\sg{\MBand}},\partial\sg{\MBand}}
    \times
    \pi_0\StabilizerIsotId{\func|_{\YYi{0}},\partial\YYi{0}}.
\]
Denote $\Agrp := \pi_0\StabilizerIsotId{\func|_{\YYi{0}},\partial\YYi{0}}$.
Then by Theorem~\ref{th:class:oriented_case}, $\Agrp = \pi_0\StabilizerIsotId{\func|_{\YYi{0}},\partial\YYi{0}} = \pi_1\Orbit{\func|_{\YYi{0}}}\in\classGroups$.

Therefore, replacing $\sg{\MBand}$ with $\MBand$ and taking to account Lemma~\ref{lm:reduct_sgB}, it remains to prove the following statement:
\begin{itemize}
\item[$- $]
\term{if the cylinder $\YYi{0}$ of the \spd{\func}\ contains no critical points of $\func$, then
\[
    \PST{\MBand}=\pi_0\StabilizerIsotId{\func|_{\MBand},\partial\MBand} \cong \WrGHZ{\Ggrp}{\Hgrp}{\gamma}{\mel}
\]
for some $\Ggrp,\Hgrp,\gamma,\mel$ as in~\eqref{equ:pi0Sprf_A_GHZ__MBand:in_the_proof}.}
\end{itemize}
Thus, we will assume further that $\YYi{0}$ contains no critical points of $\func$.

\subsection{Several subgroups of $\PSTB{\MBand}$}
Recall that we denoted by $\KerSAct$ the normal subgroup of $\STB{\MBand}$ consisting of diffeomorphisms $\dif$ such that $\dif(\YYi{i})=\YYi{i}$, for all $i=1,\ldots,n$, and $\dif$ also preserves orientation of $\YYi{i}$.
Consider also another two subgroups
\[
    \StabilizerNbh{\func,\regU{\CrComp}\cup\YYi{0}}
    \ \subset \
    \StabilizerNbh{\func,\regU{\CrComp}\cup\partial\MBand}
\]
of $\STB{\MBand}$ consisting of diffeomorphisms fixed respectively near $\regU{\CrComp}\cup\YYi{0}$ and near $\regU{\CrComp}\cup\partial\MBand$.
Equivalently, this means that such diffeomorphisms are supported respectively in $\mathop{\cup}\limits_{i=1}^{n}\Int{\YYi{i}}$ and $\mathop{\cup}\limits_{i=0}^{n}\Int{\YYi{i}}$.
Hence, they are contained in $\KerSAct$, and thus we have the following inclusions:
\begin{equation}\label{equ:subgroups_of_STfB}
    \StabilizerNbh{\func,\regU{\CrComp}\cup\YYi{0}}
    \ \subset \
    \StabilizerNbh{\func,\regU{\Kman}\cup\partial\MBand}
    \ \subset \
    \KerSAct
    \ \subset \
    \STB{\MBand}.
\end{equation}
\begin{lemma}[{\rm\cite[Lemma~8.2]{MaksymenkoKuznietsova:PIGC:2019}}]\label{lm:pi0Ker}
The inclusion $\kappa:\StabilizerNbh{\func,\regU{\Kman}\cup\partial\MBand} \subset \KerSAct$ is a homotopy equivalence.
In particular, it yields an isomorphism
\begin{equation}\label{equ:pi0SfUK_pi0SfY}
    \kappa_0:\pi_0\StabilizerNbh{\func,\regU{\Kman}\cup\partial\MBand} \cong \pi_0\KerSAct.
\end{equation}
\end{lemma}
\begin{corollary}\label{cor:pi0Ker}
We have the following commutative diagram in which the arrows $\alpha_0$ are isomorphisms:
\[
\xymatrix{
    \ \pi_0\StabilizerNbh{\func,\regU{\CrComp}\cup\YYi{0}}       \ \ar@{^(->}[r] \ar[d]^-{\alpha_0}_-{\cong} &
    \ \pi_0\StabilizerNbh{\func,\regU{\Kman}\cup\partial\MBand}  \ \ar@{=}[r]^-{\kappa_0}_-{\cong} \ar[d]^-{\alpha_0}_-{\cong} &
    \ \pi_0\KerSAct                                              \ \ar@{^(->}[d] \ar@{=}[ld] \\
%%%%%%%
    \ \myprod_{i=1}^{n}\PST{\YYi{i}} \times 0               \   \ar@{^(->}[r] \ar@/_15pt/[rr]_-{\lambda} &
    \ \bigl(\myprod_{i=1}^{n}\PST{\YYi{i}}\bigr) \times \bZ \   \ar@{^(->}[r] &
    \ \pi_0\STB{\MBand}
}
\]
\end{corollary}
\begin{proof}
Evidently, $\KerSAct$ contains the identity path component $\StabilizerId{\func,\partial\MBand}$ of $\STB{\MBand}$ being therefore the identity path component of $\KerSAct$.
Hence, the inclusion $\KerSAct \subset \STB{\MBand}$ induces a monomorphism
\[
\pi_0\KerSAct = \frac{\KerSAct}{\StabilizerId{\func,\partial\MBand}}
    \monoArrow
    \frac{\STB{\MBand}}{\StabilizerId{\func,\partial\MBand}}
    =
    \PSTB{\MBand}
\]
represented by the right vertical arrow of the diagram.

Further notice, that we have a natural restriction homomorphism:
\[
\alpha:\StabilizerNbh{\func,\regU{\Kman}\cup\partial\MBand} \to \myprod_{i=0}^{n}\STN{\YYi{i}},
\qquad
\alpha(\dif) =
\bigl( \dif|_{\YYi{0}}, \ldots, \dif|_{\YYi{n}} \bigr),
\]
which is evidently an \term{isomorphism} of topological groups.
Hence, due to Lemma~\ref{lm:Snb_S}, we get the following isomorphisms:
\[
    \alpha_0: \pi_0\StabilizerNbh{\func,\regU{\Kman}\cup\partial\MBand}  \cong  \myprod_{i=0}^{n}\PSTN{\YYi{i}} \cong  \myprod_{i=0}^{n}\PST{\YYi{i}}.
\]
Since $\YYi{0}$ contains no critical points of $\func$, we have that $\PST{\YYi{0}} \cong \bZ$ by~\cite[Theorem~5.5(1b)]{Maksymenko:TA:2020}, and this group is generated by an isotopy class of a Dehn twist $\tau\in\ST{\YYi{0}}$ fixed near $\partial\YYi{0}$.
This gives the isomorphism $\alpha_0$.

Finally, note that $\alpha\bigl(\StabilizerNbh{\func,\regU{\Kman}\cup\YYi{0}}\bigr) = \myprod_{i=1}^{n}\STN{\YYi{i}}$.
This implies the left vertical arrow is an isomorphism, while the left upper horizontal arrow is a monomorphism.
\end{proof}

\begin{remark}\rm
For simplicity, we will further identify each $\PFST{\YYi{i}} \subset \PST{\YYi{i}}$ with some subgroups of $\PSTB{\MBand}$ via a monomorphism $\lambda$.
For that reason let us explicitly describe this identification.
Let $\dif\in\ST{\YYi{i}}$ be a representative of some element $[\dif]\in\PST{\YYi{i}}$.
Then $\lambda([\dif])$ is obtained as follows: make $\dif$ fixed near $\partial\YYi{i}$ by any $\func$-preserving isotopy fixed on $\partial\YYi{i}$, extend it further by the identity on all of $\MBand$, and finally take the isotopy class of that extension in $\PSTB{\MBand}$.
\end{remark}

Since $\YYi{0}$ contains no critical points of $\func$, every $\dif\in\ST{\YYi{i}}$ leaves invariant every contour of $\func$ in $\YYi{0}$, whence
\begin{equation}\label{equ:lambda_Deltai__L_cap_Delta}
    \PFST{\YYi{i}} \ = \ \PFSTB{\MBand} \, \cap \, \PST{\YYi{i}}.
\end{equation}

In Section~\ref{sect:proof:th:g_eta} we will prove the following Theorem~\ref{th:g_eta} being a M\"obius band counterpart of the construction described in~\cite[Section~12]{Maksymenko:TA:2020} for $2$-disks and cylinders.
It will allow to describe an algebraic structure of $\PSTB{\MBand}$ using Lemmas~\ref{lm:3x3_G_m:char} and~\ref{lm:3x3_G_H_gamma_m:char}, see Section~\ref{sect:decuct:th:class:Mobius_band}.
\begin{theorem}\label{th:g_eta}
Suppose the cylinder $\YYi{0}$ contains no critical points of $\func$.
Then
\begin{enumerate}[label={\rm(\alph*)}]
\item\label{enum:th:g_eta:ker_eta}
$\ker(\eta) = \myprod_{i=1}^{n}\PST{\YYi{i}}$;

\item\label{enum:th:g_eta:ega_g_1}
there exists $\gdif\in\PSTB{\MBand}$ such that $\eta(\gdif)=1$, $\gdif^{\sgOrd}\in \pi_0\FolStabilizer{\func,\partial\MBand}$, and $\gel^{\sgOrd}$ also commutes with $\ker(\eta)$.
\end{enumerate}
\end{theorem}

\begin{corollary}\label{cor:stab_3x3_diagram}
We have the following exact $(3\times3)$-diagram in which all rows and columns are exact, the upper row is a product of Bieberbach sequences for $(\func|_{\YYi{i}},\partial\YYi{i})$, $i=1,\ldots,k$, while the middle row is the Bieberbach sequence of $(\func,\partial\MBand)$:
\begin{equation}\label{equ:diagram:eta_pi0SfdB}
    \begin{gathered}
    \xymatrix@C=1.5em@R=1.7em{
        \ \myprod_{i=1}^{n}\PFST{\YYi{i}}  \ \ar@{^(->}[r] \ar@{^(->}[d]^-{\lambda} &
        \ \myprod_{i=1}^{n}\PST{\YYi{i}}   \ \ar@{->>}[r]  \ar@{^(->}[d]^-{\lambda} &
        \ \myprod_{i=1}^{n}\GST{\YYi{i}}   \ \ar@{^(->}[d]    \\
        \ \PFSTB{\MBand}                   \ \ar@{^(->}[r] \ar@{->>}[d]  &
        \ \PSTB{\MBand}                    \ \ar@{->>}[d]^-{\eta}\ar@{->>}[r] &
        \ \GSTB{\MBand}                    \ \ar@{->>}[d] \\
        \ \sgOrd\bZ                        \ \ar@{^(->}[r]   &
        \ \bZ                              \ \ar@{->>}[r] &
        \ \bZ_{\sgOrd}                     \
    }
    \end{gathered}
\end{equation}
\end{corollary}
\begin{proof}
For simplicity, denote $\Bgrp = \PSTB{\MBand}$, $\Lgrp = \myprod_{i=1}^{n}\pi_0\Stabilizer{\func|_{\YYi{i}},\partial\YYi{i}}$, $\Agrp = \pi_0\FolStabilizer{\func,\partial\MBand}$, and $\Kgrp:=\Agrp\cap\Lgrp$.
Then by~\eqref{equ:lambda_Deltai__L_cap_Delta}, $\Kgrp = \myprod_{i=1}^{n}\pi_0\FolStabilizer{\func|_{\YYi{i}},\partial\YYi{i}}$.
Hence, by Lemma~\ref{lm:3x3-diagram} we get an exact $(3\times3)$-diagram~\eqref{equ:3x3-diagram}.
Moreover, by definition, $\myprod_{i=1}^{n}\GST{\YYi{i}} = \Lgrp/\Kgrp$ and $\GSTB{\MBand} = \Bgrp / \Agrp$, so corresponding right upper and middle terms of~\eqref{equ:diagram:eta_pi0SfdB} and~\eqref{equ:3x3-diagram} agree.

The identification of bottom rows follows from Theorem~\ref{th:g_eta}.
Indeed, since $\eta$ is an epimorphism and $\Lgrp=\ker(\eta)$, we see that $\Bgrp/\Lgrp=\bZ$.

Moreover, since $\Agrp = \PFSTB{\MBand} \subset \pi_0\KerSAct=\eta^{-1}(\sgOrd)$, we have that $\eta(\Agrp) \subset \sgOrd\bZ$.
Conversely, since $\gdif^{\sgOrd}\in\Agrp$ and $\eta(\sgOrd) = 1$, $\sgOrd = \eta(\gdif^{\sgOrd}) \in \eta(\Agrp)$, and thus $\sgOrd\bZ\subset\eta(\Agrp)$.
\end{proof}

\section{Deduction of Theorem~\ref{th:class:Mobius_band} from Theorem~\ref{th:g_eta}}\label{sect:decuct:th:class:Mobius_band}
Recall that if $\Xman$ is either a M\"obius band or a $2$-disk, and $\func\in\FSpR{\Xman}$, then we have the following isomorphisms:
\begin{align*}
    &\pi_0\FolStabilizerIsotId{\func,\partial\Xman}
        \stackrel{\eqref{equ:STf_STprf_disk_cyl_mb}}{\cong}
    \PFSTB{\Xman}, &
    &\pi_1\OrbitComp{\func}{\func}
        \stackrel{\eqref{equ:j_piOf_pi0Stprf}}{\cong}
    \pi_0\StabilizerIsotId{\func,\partial\Xman}
        \stackrel{\eqref{equ:STf_STprf_disk_cyl_mb}}{\cong}
    \PSTB{\Xman}.
\end{align*}
Therefore, in what follows it will be more convenient to use groups $\PSTB{\Xman}$.

Now let $\func\in\FSpR{\MBand}$.
Keeping notations from the previous Section, denote by $\cntOne$ and $\cntTwo$ the numbers of orbits of the \term{non-free} $\bZ_{\sgOrd}$-action on the set $\CompSet$ of disks of types~\ref{dt1} and~\ref{dt2} respectively.
Let also $\sg{\cntOne}$ and $\sg{\cntTwo}$ be the numbers of orbits of types~\ref{dt1} and~\ref{dt2} of the \term{free} $\bZ_{\sgOrd}$-action on the set $\PlCompSet = \CompSet\times\{\pm1\}$.

Since the action of $\bZ_{\sgOrd}$ on $\PlCompSet$ is free, we have that $2n = \sgOrd\,(\sg{\cntOne} + \sg{\cntTwo})$.
Moreover, due to Lemma~\ref{lm:action_g_on_hY},  $\YYi{i}$ is a disk of type~\ref{dt1} (resp.\ of type~\ref{dt2}) iff $(\YYi{i},1)$ and $(\YYi{i},-1)$ belong to distinct orbits (resp.\ the same orbit).
This implies that $\sg{\cntOne} = 2\cntOne$, and $\sg{\cntTwo} = \cntTwo$.
Hence,
\begin{equation}\label{equ:n_d_e2_b}
n = \sgOrd\,(\cntOne + \cntTwo/2).
\end{equation}
Note that this number is always integer, since $\sgOrd$ is even whenever $\cntTwo>0$.
Consider the following three cases.

\subsection*{Case (A)}
Suppose all disks are of type~\ref{dt1}, i.e.\ $\cntTwo=0$.
Then $\sgOrd$ is odd due to Lemma~\ref{lm:b_is_odd}.
\begin{lemma}\label{lm:all_disks_T1}
In the case {\rm(A)} the action of $\bZ_{\sgOrd}$ on disks $\CompSet$ is free, so $n = \sgOrd\cntOne$, and for some $\Ggrp\in\classGroups$ we have an isomorphism:
\[
    \pi_1\OrbitComp{\func}{\func} \ \cong \ \PSTB{\MBand} \ \cong \ \WrGZ{\Ggrp}{\sgOrd}.
\]
\end{lemma}
\begin{proof}
The assumption that all disks are of type~\ref{dt1} evidently means that if $\dif(\YYi{i})=\YYi{i}$ for some $i=1,\ldots,n$ and $\dif\in\STB{\MBand}$, then $\dif$ also preserves orientation of $\YYi{i}$.
But since the action of $\bZ_{\sgOrd}$ on $\PlCompSet$ is free, it implies that $\dif$ leaves invariant all other disks $\YYi{i'}$ and preserves their orientations.
Hence, the action of $\bZ_{\sgOrd}$ on $\PlCompSet$ is free.

Therefore, one can enumerate disks in $\CompSet$ as follows:
\begin{equation}\label{equ:disks_T1}
\begin{array}{cccc}
    \DD{1}{0} & \DD{1}{1} & \cdots & \DD{1}{\sgOrd-1} \\
    \DD{2}{0} & \DD{2}{1} & \cdots & \DD{2}{\sgOrd-1} \\
    \vdots    & \vdots    & \vdots & \vdots      \\
    \DD{\cntOne}{0} & \DD{\cntOne}{1} & \cdots & \DD{\cntOne}{\sgOrd-1} \\
\end{array}
\end{equation}
so that $\gdif$ will cyclically shift columns to the \term{right}, i.e.\ $\gdif(\DD{j}{i})=\DD{j}{i+1\bmod \sgOrd}$ for all $i,j$.
An example is shown in Figure~\ref{fig:case_a:b3:n9}.
Let us point out that in that figure $\DD{j}{0}, \DD{j}{2}, \DD{j}{1}$, $j=1,2,3$, is the ``natural geometric'' ordering of those disks ``along'' $\MBand$, and it differs from their ordering in the $j$-row of~\eqref{equ:disks_T1}.

Also note that shifting to the right of columns in~\eqref{equ:disks_T1} is opposite to the behavior of the automorphism $\ahom$ defining the groups $\WrGZ{\Ggrp}{\sgOrd}$ from Section~\ref{sect:grp:GwrmZ}.
We are going to apply Lemma~\ref{lm:3x3_G_m:char} and therefore will now introduce notations agreeing with that lemma.

\begin{figure}[htbp]
\includegraphics[width=0.9\textwidth]{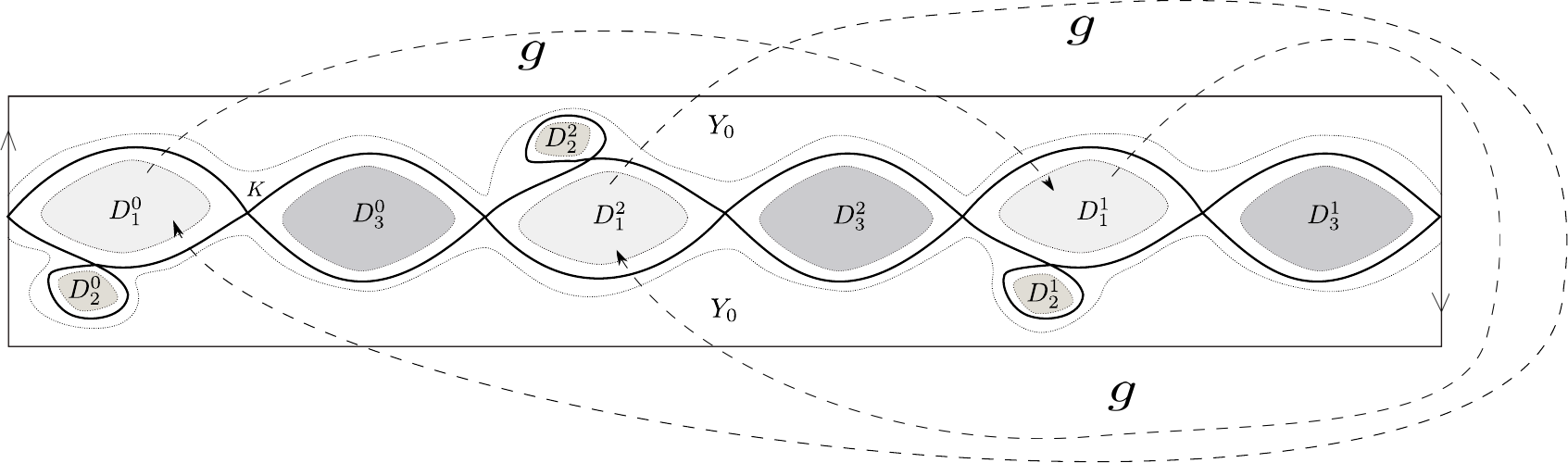}
\mycaption{Case (A): $\cntOne=3$, $\cntTwo=0$, $\sgOrd=3$, $n=\sgOrd\cntOne=9$}
{\PSTB{\MBand}      \, \cong  \, \WrGZ{\Bigl(\myprod_{j=1}^{3}\PST{\DD{j}{0}}\Bigr)}{3}}
{\pi_1\Orbit{\func} \, \cong  \, \WrGZ{\Bigl(\myprod_{j=1}^{3} \pi_1\Orbit{\func|_{\DD{j}{0}}} \Bigr)}{3}}
\label{fig:case_a:b3:n9}
\end{figure}

For $i=0,\ldots,\sgOrd-1$ denote
\begin{align*}
    \Ggrp_i &:= \myprod_{j=1}^{\cntOne} \PST{\DD{j}{i}},  &
    \Pgrp_i &:= \myprod_{j=1}^{\cntOne} \PFST{\DD{j}{i}}.
\end{align*}
Let also $\Ggrp:=\Ggrp_0$ and $\Pgrp:=\Pgrp_0$.
As $\gdif^{i}(\DD{j}{0})=\DD{j}{i}$, we get from Lemma~\ref{lm:action_g_on_hY}\ref{enum:lm:action_g_on_hY:conj_by_g} that $\Ggrp_i:=\gdif^{i}\Ggrp\gdif^{-i}$ and $\Pgrp_i:=\gdif^{i}\Pgrp\gdif^{-i}$ as in Lemma~\ref{lm:3x3_G_m:char}.
Hence, due to diagram~\eqref{equ:diagram:eta_pi0SfdB}, the kernel of $\eta$ is a direct product of subgroups $\Ggrp_i$ conjugated to $\Ggrp$:
\[
    \Lgrp := \ker(\eta)
    \stackrel{\eqref{equ:diagram:eta_pi0SfdB}}{=} 
    \myprod_{k=1}^{n}\!\PST{\YYi{k}}
    = \myprod_{i=0}^{\sgOrd-1}\myprod_{j=1}^{\cntOne}\!\PST{\DD{j}{i}}
    = \Ggrp_0\times\cdots\times\Ggrp_{\sgOrd-1}.
\]

Moreover, denote $\Agrp :=\PFSTB{\MBand}$.
Then~\eqref{equ:lambda_Deltai__L_cap_Delta} implies that $\Pgrp_i = \Agrp \cap \Ggrp_i$.
Therefore, by diagram~\eqref{equ:diagram:eta_pi0SfdB}, $\Kgrp := \Agrp\cap\Lgrp = \Pgrp_0\times\cdots\times\Pgrp_{\sgOrd-1}$ is generated by groups $\Pgrp_i$.
All other requirements of Lemma~\ref{lm:3x3_G_m:char}:
\begin{itemize}
\item$\eta(\gdif)=1$, $\eta(\Agrp) = \sgOrd\bZ$, $\gdif^{\sgOrd}\in\Agrp$, and that $\gdif^{\sgOrd}$ commutes with $\Lgrp$,
\end{itemize}
are contained in the statement of Theorem~\ref{th:g_eta}.

Therefore, by Lemma~\ref{lm:3x3_G_m:char}, the diagram~\eqref{equ:diagram:eta_pi0SfdB} is isomorphic to the following one:
{\small%
\begin{equation*}
    \begin{gathered}
    \xymatrix@R=5ex@C=3ex{
        \Bigl(\myprod_{j=1}^{\cntOne}\!\PFST{\DD{j}{0}}\Bigr)^{\sgOrd} \!\!\!\times\! 0  \ar@{^(->}[r] \ar@{^(->}[d] &
        \Bigl(\myprod_{j=1}^{\cntOne}\!\PST{\DD{j}{0}}\Bigr)^{\sgOrd}  \!\!\!\times\! 0  \ar@{->>}[r]  \ar@{^(->}[d] &
        \Bigl(\myprod_{j=1}^{n}\GST{\DD{j}{0}}\Bigr)^{\sgOrd}        \!\!\!\times\! 0                \ar@{^(->}[d] \\
        %%%%%%%%%%%%%
        \Bigl(\myprod_{j=1}^{\cntOne}\!\PFST{\DD{j}{0}}\Bigr)^{\sgOrd} \!\!\!\times\! \sgOrd\bZ   \ar@{^(->}[r] \ar@{->>}[d]  &
        \WrGZ{\Bigl(\myprod_{j=1}^{\cntOne}\!\PST{\DD{j}{0}}\Bigr)}{\sgOrd}             \ar@{->>}[d]^-{\eta}\ar@{->>}[r] &
        \WrGZm{\Bigl(\myprod_{j=1}^{\cntOne}\!\GST{\DD{j}{0}}\Bigr)}{\sgOrd}            \ar@{->>}[d] \\
        \sgOrd\bZ                                                               \ar@{^(->}[r]   &
        \bZ                                                                     \ar@{->>}[r] &
        \bZ_{\sgOrd}
    }
\end{gathered}
\end{equation*}}

Now, by Theorem~\ref{th:class:oriented_case}, each group $\PST{\DD{j}{0}} = \pi_1\Orbit{\func|_{\DD{j}{0}}}$ belongs to $\classGroups$, whence their product $\Ggrp =\myprod_{j=1}^{\cntOne}\PST{\DD{j}{0}} = \myprod_{j=1}^{\cntOne}\pi_1\Orbit{\func|_{\DD{j}{0}}}\in\classGroups$ as well.
It remains to note that the last diagram contains the statement that $\PSTB{\MBand} \cong \WrGZ{\Ggrp}{\sgOrd}$.
\end{proof}

\begin{remark}\rm
If $\sgOrd=1$, so $\STB{\MBand}=\KerSAct$ and thus the action is in fact trivial, we have that $\Ggrp = \myprod_{i=1}^{n}\PST{\YYi{i}}$, and by Lemma~\ref{lm:all_disks_T1},
\[
    \pi_0\KerSAct \cong \PSTB{\MBand} \ \cong \ \WrGZ{\Ggrp}{1} \cong \Ggrp\times\bZ = \myprod_{i=1}^{\cntOne}\PST{\YYi{i}} \ \times \ \bZ,
\]
which agrees with isomorphism $\kappa_0$ from Lemma~\ref{lm:pi0Ker}, see example in Figure~\ref{fig:case_a:b1:n3}.
\begin{figure}[htbp]
\begin{tabular}{ccc}
\includegraphics[height=2.5cm]{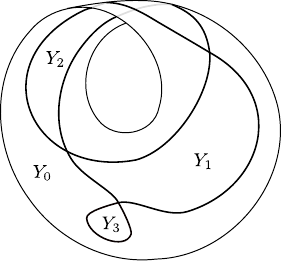} & \qquad\qquad &
\includegraphics[height=2.5cm]{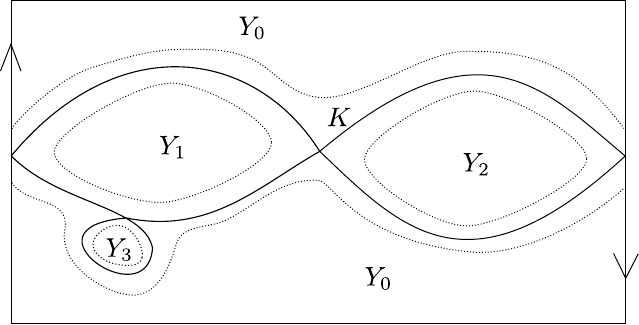}
\end{tabular}
\mycaption{Case~{\rm(A)}: $\cntOne=3$, $\cntTwo=0$, $b=1$, $n=\sgOrd\cntOne=3$}
{\PSTB{\MBand}      \, \cong \, \Bigl(\myprod_{i=1}^{3}\PST{\YYi{i}}\Bigr) \times\bZ}
{\pi_1\Orbit{\func} \, \cong \, \Bigl(\myprod_{j=1}^{3} \pi_1\Orbit{\func|_{\YYi{j}}} \Bigr) \times\bZ.}
\label{fig:case_a:b1:n3}
\end{figure}
\end{remark}

\subsection*{Case (B)}
Suppose all disks are of type~\ref{dt2}, i.e.\ $\cntOne=0$.
Then $\sgOrd=2\mel$ for some $\mel\geq1$, whence $n = \sgOrd\cntTwo/2 = \mel\cntTwo$, see~\eqref{equ:n_d_e2_b}.
\begin{lemma}\label{lm:all_disks_T2}
In the case~{\rm(B)} there exist $\Hgrp\in\classGroups$ and its automorphism $\gamma:\Hgrp\to\Hgrp$ with $\gamma^2=\id_{\Hgrp}$ such that we have an isomorphism:
\[
    \pi_1\Orbit{\func} \ \cong \ \PSTB{\MBand} \ \cong \ \WrGHZ{\UnitGroup}{\Hgrp}{\gamma}{\mel}.
\]
\end{lemma}
\begin{proof}
The proof is similar to Lemma~\ref{lm:all_disks_T1}.
Since all orbits are of type~\ref{dt2}, we can also enumerate disks in $\CompSet$ as follows:
\begin{equation}\label{equ:disks_T2}
\begin{array}{cccc}
    \EE{1}{0} & \EE{1}{1} & \cdots & \EE{1}{\mel-1} \\
    \EE{2}{0} & \EE{2}{1} & \cdots & \EE{2}{\mel-1} \\
    \vdots    & \vdots    & \vdots & \vdots      \\
    \EE{\cntTwo}{0} & \EE{\cntTwo}{1} & \cdots & \EE{\cntTwo}{\mel-1} \\
\end{array}
\end{equation}
so that $\gdif$ will cyclically shift the columns to the right, i.e.\ $\gdif(\EE{j}{i})=\EE{j}{i+1\bmod \mel}$ for all $i,j$.
In particular, $\gdif^{\mel}(\EE{j}{i})=\EE{j}{i}$, and the restriction $\gdif^{\mel}:\EE{j}{i}\to\EE{j}{i}$ reverses orientation, see Figure~\ref{fig:case_b:b1:n1}.
\begin{figure}[htbp!]
\includegraphics[height=2.5cm]{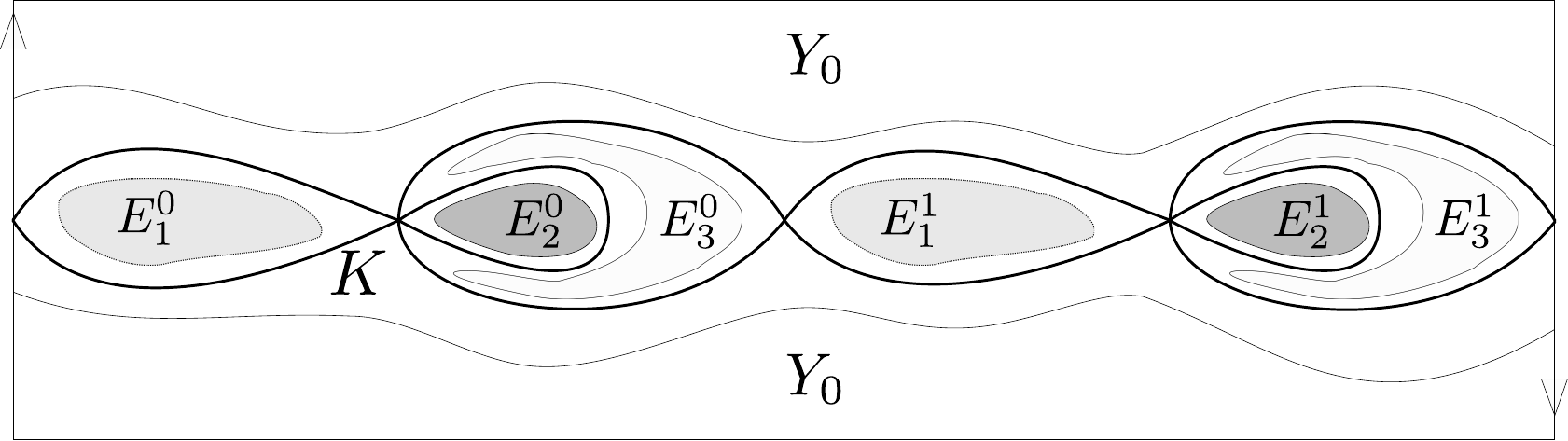}
\mycaption{Case (B): $\cntOne=0$, $\cntTwo=3$, $\sgOrd=4$, $\mel=\sgOrd/2 = 2$, $n=\mel\cntTwo=6$}
{\PSTB{\MBand} \cong \WrGHZ{\UnitGroup}{\myprod_{j=1}^{3}\PST{\EE{j}{0}}}{2}{\gamma}}
{\pi_1\Orbit{\func} \cong \WrGHZ{\UnitGroup}{\myprod_{j=1}^{3}\pi_1\Orbit{\func|_{\EE{j}{0}}}}{2}{\gamma}}
\label{fig:case_b:b4:n6}
\end{figure}

\begin{figure}[htbp!]
\begin{tabular}{ccc}
    \includegraphics[height=2.5cm]{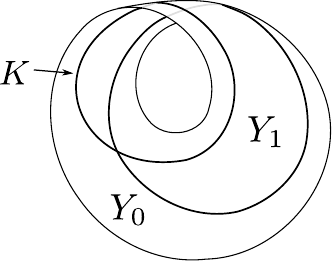} & \qquad\qquad &
    \includegraphics[height=2.5cm]{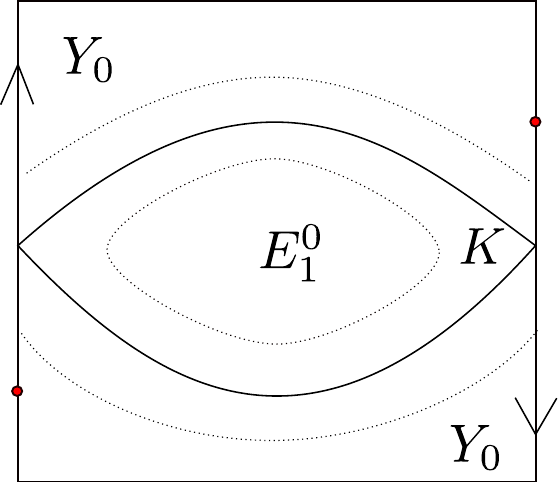}
\end{tabular}
\mycaption{Case (B): $\cntOne=0$, $\cntTwo=1$, $\sgOrd=2$, $\mel=\sgOrd/2 = 1$, $n=\mel\cntTwo=1$}
{\PSTB{\MBand} \cong \WrGHZ{\UnitGroup}{\PST{\EE{1}{0}}}{1}{\gamma}}
{\pi_1\Orbit{\func} \cong \WrGHZ{\UnitGroup}{\pi_1\Orbit{\func|_{\EE{1}{0}}}}{1}{\gamma}}
\label{fig:case_b:b1:n1}
\end{figure}

For $i=0,\ldots,\mel-1$ denote
\begin{align*}
    \Hgrp_i &:= \myprod_{j=1}^{\cntTwo} \PST{\EE{j}{i}},  &
    \Qgrp_i &:= \myprod_{j=1}^{\cntTwo} \PFST{\EE{j}{i}}.
\end{align*}
Let also $\Hgrp:=\Hgrp_0$ and $\Qgrp:=\Qgrp_0$.
Since $\gdif^{i}(\EE{j}{0})=\EE{j}{i}$, we get from Lemma~\ref{lm:action_g_on_hY}\ref{enum:lm:action_g_on_hY:conj_by_g} that $\Hgrp_i:=\gdif^{i}\Hgrp\gdif^{-i}$ and $\Pgrp_i:=\gdif^{i}\Pgrp\gdif^{-i}$ as in Lemma~\ref{lm:3x3_G_H_gamma_m:char}.
Hence, due to diagram~\eqref{equ:diagram:eta_pi0SfdB}, the kernel of $\eta$ is a direct product of subgroups $\Hgrp_i$ conjugated to $\Hgrp$:
\[
    \Lgrp := \ker(\eta) \stackrel{\eqref{equ:diagram:eta_pi0SfdB}}{=} \myprod_{k=1}^{n}\PST{\YYi{k}}
    = \myprod_{i=0}^{\sgOrd-1} \myprod_{j=1}^{\cntTwo} \PST{\EE{j}{i}}
    = \Hgrp_0\times\cdots\times\Hgrp_{\sgOrd-1}.
\]
Moreover, denote $\Agrp :=\PFSTB{\MBand}$.
Then~\eqref{equ:lambda_Deltai__L_cap_Delta} implies that $\Pgrp_i = \Agrp \cap \Hgrp_i$.
Therefore, by diagram~\eqref{equ:diagram:eta_pi0SfdB}, $\Kgrp := \Agrp\cap\Lgrp = \Pgrp_0\times\cdots\times\Pgrp_{\sgOrd-1}$ is generated by groups $\Pgrp_i$.
All other requirements of Lemma~\ref{lm:3x3_G_H_gamma_m:char} are contained in the statement of Theorem~\ref{th:g_eta}.
Namely,
\begin{align*}
&\eta(\gdif)=1,           &
&\eta(\Agrp) = 2\mel\bZ,  &
&\gdif^{2\mel}\in\Agrp,   & 
&\gdif^{2\mel} \ \text{commutes with} \  \Lgrp.
\end{align*}
Moreover, $\gamma:\Hgrp\to\Hgrp$, $\gamma(a)=\gdif^{\mel} a \gdif^{-a}$, is an automorphism of $\Hgrp$ with $\gamma^2=\id_{\Hgrp}$.

Therefore, by Lemma~\ref{lm:3x3_G_H_gamma_m:char}, the diagram~\eqref{equ:diagram:eta_pi0SfdB} is isomorphic to the following one:
\begin{equation*}
    \begin{gathered}
    \xymatrix@C=2.5ex{
        \UnitGroup^{2\mel} \times \Delta_0^{\mel}       \times 0  \ar@{^(->}[r] \ar@{^(->}[d] &
        \UnitGroup^{2\mel} \times \mathcal{S}_0^{\mel}  \times 0  \ar@{->>}[r]  \ar@{^(->}[d] &
        \UnitGroup^{2\mel} \times \mathbf{G}_0^{\mel}   \times 0                \ar@{^(->}[d] \\
        %%%%%%%%%%%%%
        \UnitGroup^{2\mel} \times \Delta_0^{\mel} \times 2\mel\bZ   \ar@{^(->}[r] \ar@{->>}[d]  &
        \WrGHZ{\UnitGroup}{\mathcal{S}_0}{\gamma}{\mel}   \ar@{->>}[d]^-{\eta}\ar@{->>}[r] &
        \WrGHZ{\UnitGroup}{\mathbf{G}_0 }{\gamma}{\mel}   \ar@{->>}[d] \\
        %%%%%%%%%%%%%
        2\mel\bZ                                          \ar@{^(->}[r]   &
        \bZ                                               \ar@{->>}[r] &
        \bZ_{2\mel}
    }
\end{gathered}
\end{equation*}
where 
\begin{align*}
\Delta_0      &= \myprod_{j=1}^{\cntTwo}\PFST{\EE{j}{0}}, &
\mathcal{S}_0 &= \myprod_{j=1}^{\cntTwo}\PST{\EE{j}{0}}, &
\mathbf{G}_0  &= \myprod_{j=1}^{\cntTwo}\GST{\EE{j}{0}}.
\end{align*}

Now, by Theorem~\ref{th:class:oriented_case}, each group $\PST{\EE{j}{0}} \cong \pi_1\Orbit{\func|_{\EE{j}{0}}}$ belongs to $\classGroups$, whence their product $\mathcal{S}_0 \cong \myprod_{j=1}^{\cntTwo}\pi_1\Orbit{\func|_{\EE{j}{0}}} \in\classGroups$ as well.
Moreover, the last diagram contains the statement that $\pi_1\Orbit{\func} \cong \PSTB{\MBand} \cong \WrGHZ{\UnitGroup}{\mathcal{S}_0}{\gamma}{\mel}$.
\end{proof}

\subsection*{Case (C)}
Suppose that disks of both types~\ref{dt1} and~\ref{dt2} are presented.
Then again, since we have disks of type~\ref{dt2}, $\sgOrd$ must be even, and we also put $\mel = \sgOrd/2$.
\begin{lemma}\label{lm:there_are_disks_T1_and_T2}
In the case~{\rm(C)} there exist two groups $\Ggrp,\Hgrp\in\classGroups$ and an automorphism $\gamma:\Hgrp\to\Hgrp$ with $\gamma^2=\id_{\Hgrp}$ such that we have an isomorphism:
\[
    \pi_1\Orbit{\func} \ \cong \ \PSTB{\MBand} \ \cong \ \WrGHZ{\Ggrp}{\Hgrp}{\gamma}{\mel}.
\]
\end{lemma}
\begin{proof}
In this case the disks of type~\ref{dt1} can be enumerated as~\eqref{equ:disks_T1}, while the disks of type~\ref{dt2} can be enumerated as~\eqref{equ:disks_T2}.
Now one can define groups $\Ggrp$ and $\Hgrp$ exactly as in previous two lemmas and by means of Lemma~\ref{lm:3x3_G_H_gamma_m:char} show existence not only of an isomorphism $\PSTB{\MBand} \ \cong \ \WrGHZ{\Ggrp}{\Hgrp}{\gamma}{\mel}$, but also of an isomorphism of the diagram~\eqref{equ:diagram:eta_pi0SfdB} onto $(3\times3)$-diagram from that lemma.
We leave the details for the reader, see also Figures~\ref{fig:case_c:b2:n:4} and~\ref{fig:case_c:b6:n:12}.
\end{proof}

\begin{figure}[htbp!]
\begin{tabular}{ccc}
\includegraphics[height=2.5cm]{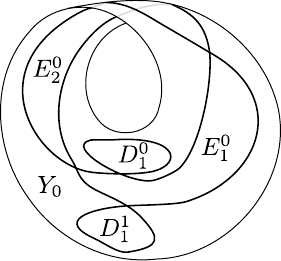} & \qquad\qquad &
\includegraphics[height=2.5cm]{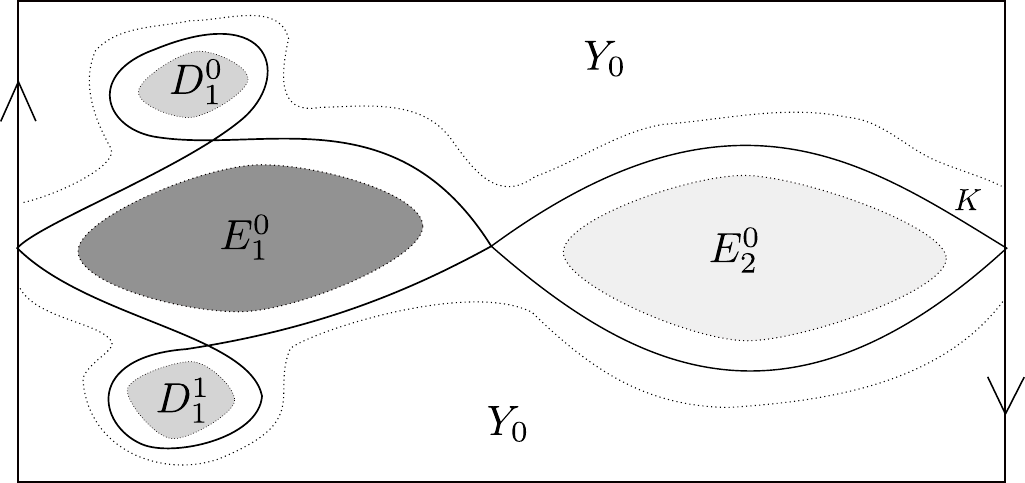}
\end{tabular}
\mycaption{Case (C): $\cntOne=1$, $\cntTwo=2$, $\sgOrd = 2$, $\mel=1$, $n = \sgOrd\, (\cntOne + \cntTwo/2) = 4$}
{\PSTB{\MBand} \cong \WrGHZ{\PST{\DD{1}{0}}}{ \  \PST{\EE{1}{0}} \times \PST{\EE{2}{0}} \, }{1}{\gamma}}
{}
\label{fig:case_c:b2:n:4}
\end{figure}
\begin{figure}[htbp!]
\includegraphics[height=2.4cm]{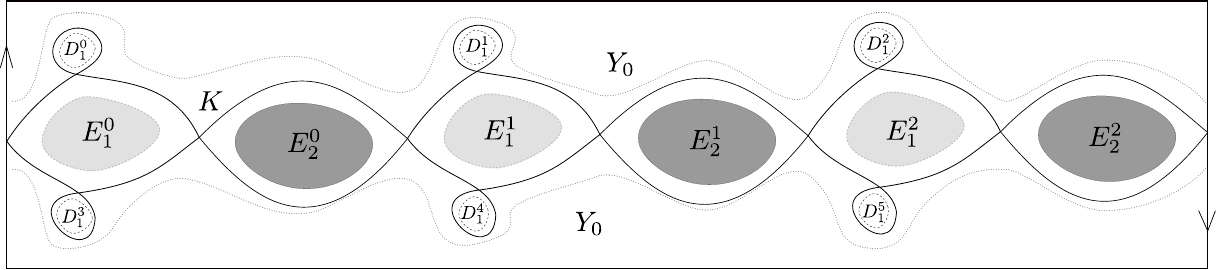}
\mycaption{{Case (C): $\cntOne=1$, $\cntTwo=2$, $\sgOrd = 6$, $\mel=3$, $n = \sgOrd\, (\cntOne + \cntTwo/2) = 12$}}
{\PSTB{\MBand} \cong \WrGHZ{\PST{\DD{1}{0}}}{ \  \PST{\EE{1}{0}} \times \PST{\EE{2}{0}} \, }{3}{\gamma}}
{}
\label{fig:case_c:b6:n:12}
\end{figure}

\section{Proof of Theorem~\ref{th:g_eta}}\label{sect:proof:th:g_eta}
Suppose the cylinder $\YYi{0}$ contains no critical points of $\func$.

\ref{enum:th:g_eta:ker_eta}
Let us check that $\ker(\eta) = \myprod_{i=1}^{n}\PST{\YYi{i}}$.
Suppose $\dif\in\STB{\MBand}$ is supported in $\mathop{\cup}\limits_{i=1}^{n}\YYi{i}$, i.e.\ fixed on $\regU{\CrComp}\cup\YYi{0} \supset \QQi{0}$.
In other words, $\dif|_{\QQi{0}}=\id_{\QQi{0}}$, whence its unique lifting $\pdif:\tQman_{0}\to\tQman_{0}$ fixed on $\bR\times0$ must be the identity, so $\pdif=\id_{\tQman_{0}}$.
Hence, $\pdif(J_k)=J_k$ for all $k\in\bZ$, and therefore $\eta(\dif)=0$.
This implies that $\myprod_{i=1}^{n}\PST{\YYi{i}} \subset \ker(\eta)$.

Conversely, since $\KerSAct = \eta^{-1}(\sgOrd\bZ)$, we have that $\ker(\eta) \subset \KerSAct$.
Moreover, by Corollary~\ref{cor:pi0Ker}, there is an isomorphism 
\[ \pi_0\KerSAct \cong \bigl(\myprod_{i=1}^{n}\PST{\YYi{i}}\bigr) \times \bZ,\]
where the generator of the multiple $\bZ$ corresponds to the Dehn twist $\tau$ along $\partial\MBand$ supported in $\YYi{0}$.
But, due to~\eqref{equ:eta_tau_b}, $\eta(\tau)=\sgOrd$, whence $\ker(\eta) \subset \myprod_{i=1}^{n}\PST{\YYi{i}}$.

\ref{enum:th:g_eta:ega_g_1}
We need to find $\gdif\in\STB{\MBand}$ such that
\begin{enumerate}[label={\alph*)}, leftmargin=8ex]
\item\label{enum:g_prop:1} $\eta(\gdif)=1$,
\item\label{enum:g_prop:gb_in_FS} $\gdif^{\sgOrd}\in \pi_0\FolStabilizer{\func,\partial\MBand}$,
\item\label{enum:g_prop:gb_comm_ker_eta} $\gel^{\sgOrd}$ also commutes with $\ker(\eta)$.
\end{enumerate}
The proof resembles~\cite[Lemma~13.1(3)]{Maksymenko:TA:2020}, and also uses the result from the paper~\cite{KuznietsovaMaksymenko:PIGC:2020} which was especially written to include all technicalities needed for the proof of the existence of $\gdif$.

If $\sgOrd=1$ we can take $\gdif=\id_{\MBand}$ and all conditions will trivially hold.

Thus suppose $\sgOrd\geq2$.
Take any $\qdif\in\STB{\MBand}$ such that $\eta(\qdif)=1$.
We will change $\qdif$ so that conditions \ref{enum:g_prop:1}-\ref{enum:g_prop:gb_comm_ker_eta} will be satisfied.

Recall that the disks $\YYi{i}$, $i=1,\ldots,n$, can be divided into disks of types~\ref{dt1} and~\ref{dt2}, and also enumerated in accordance with~\eqref{equ:disks_T1} and~\eqref{equ:disks_T2},
so that $\qdif$ will cyclically shift the columns to the right.
Also if disks of type~\ref{dt2} are presented, then $\sgOrd$ is even and in this case we put $\mel=\sgOrd/2$.

It will also be convenient to denote by
\begin{align*}
    \DDLast&:= \mathop{\cup}\limits_{j=1}^{\cntOne} \DD{j}{\sgOrd-1},
    &
    \EELast&:= \mathop{\cup}\limits_{j=1}^{\cntTwo} \EE{j}{\mel-1},
\end{align*}
the unions of disks standing in the last columns of \eqref{equ:disks_T1} and~\eqref{equ:disks_T2}.
These sets might be empty, if the disks of the corresponding types are not presented.

Now we need to construct two additional diffeomorphisms $\dif$ and $\kdif$ depending on existence of disks of those types.

(1) Suppose there are disks of type~\ref{dt1}.
Note that $\qdif^{\sgOrd}\in\KerSAct=\eta^{-1}(\sgOrd\bZ)$, since $\eta(\qdif^{\sgOrd})=\sgOrd$.
Recall that due to Lemma~\ref{lm:pi0Ker} the inclusion 
\[j:\StabilizerNbh{\func,\regU{\Kman}\cup\partial\MBand} \subset \KerSAct\]
is a homotopy equivalence, whence $\qdif^{\sgOrd}$ is isotopic in $\KerSAct$ to some diffeomorphism $\dif_0\in\StabilizerNbh{\func,\regU{\Kman}\cup\partial\MBand}$.
Moreover, ``releasing $\partial\MBand$'' we see that $\dif_0$ is in turn isotopic in $\Stabilizer{\func}$ to a diffeomorphism $\dif$ fixed on $\regU{\CrComp}\cup\YYi{0}$ via an isotopy supported in $\YYi{0}$.

Hence, $\dif^{-1}\circ\qdif^{\sgOrd}\in\StabilizerId{\func}$.
Then for each $i\in\{0,1,\ldots,\sgOrd-1\}$ we have that
\begin{enumerate}[label={\rm(h\arabic*)}]
\item\label{enum:h_prop:1} $\qdif^{i}\circ\dif^{-1}\circ\qdif^{\sgOrd-i} \in \StabilizerId{\func}$, since $\StabilizerId{\func}$ is a normal subgroup of $\Stabilizer{\func}$;
\item\label{enum:h_prop:3} $\dif^{-1}\circ\qdif = \qdif:\DDLast\to\DDX{0}$, since in particular, $\dif$ is fixed near $\partial\DDLast$.
\end{enumerate}

(2) Suppose now that there are disks of type~\ref{dt2}.
Then for each $i,j$ we have a reversing orientation diffeomorphism $\qdif^{\mel}:\EE{j}{i}\to\EE{j}{i}$.
In particular, by~\cite[Theorem~3.5]{KuznietsovaMaksymenko:PIGC:2020}, applied to each connected components of $\EELast$, there exists a diffeomorphism $\kdif:\EELast\to\EELast$ such that $\kdif=\qdif^{\mel}$ near $\partial\EELast$, and $\kdif^2\in\StabilizerId{\func|_{\EELast}}$.
It follows that
\begin{enumerate}[leftmargin=7ex, label={\rm(k\arabic*)}]
\item\label{enum:k_prop:1} $\kdif^{-1}\circ\qdif^{1-\mel} = \qdif: \EEX{\mel-1}\to\EEX{0}$ near some neighborhood of $\partial\EELast$.
\end{enumerate}

Now define the map $\gdif:\MBand\to\MBand$ by the following rule:
\[
\gdif(\px) =
\begin{cases}
    \dif^{-1}\circ\qdif,           & \px \in \DDLast, \\
    \kdif^{-1}\circ\qdif^{1-\mel}, & \px \in \EELast, \\
    \qdif(\px),                    & \px\in\MBand\setminus(\DDLast\cup\EELast).
\end{cases}
\]
Due to~\ref{enum:h_prop:3} and~\ref{enum:k_prop:1}, $\gdif$ is a diffeomorphism.
It is easy to see that $\gdif$ is also $\func$-preserving, i.e.\ $\gdif\in\STB{\MBand}$.
We claim that $\gdif$ satisfies conditions \ref{enum:g_prop:1}-\ref{enum:g_prop:gb_comm_ker_eta}.

\ref{enum:g_prop:1}
Since $\gdif=\qdif$ on $\MBand\setminus(\DDLast\cup\EELast) \supset \QQi{0}$, it follows that $\eta(\dif)=\eta(\qdif)=1$.

\ref{enum:g_prop:gb_in_FS}
Since $\gdif$, and therefore $\gdif^{\sgOrd}$, are fixed on $\partial\MBand$, it suffices to prove that $\gdif^{\sgOrd}\in\FolStabilizer{\func}$.

Moreover, $\gdif^{\sgOrd} = \qdif^{\sgOrd} \in \KerSAct$, and therefore it leaves invariant each $\YYi{i}$, $i=0,\ldots,n$, and $\regU{\CrComp}$.
Therefore, due to Lemma~\ref{lm:h_in_FST}, if $A$ denotes either of the surfaces $\YYi{i}$ or $\regU{\CrComp}\cup\YYi{0}$, then we need to check that the restriction of $\gdif^{\sgOrd}|_{A}$ belongs to the corresponding group $\FolStabilizer{\func|_{A}}$.

Since $\gdif^{\sgOrd} \in \KerSAct$, it follows that $\gdif^{\sgOrd}$ leaves invariant all contours of $\func$ contained in $\regU{\CrComp}\cup\YYi{0}$, i.e.\ $\gdif^{\sgOrd}|_{\regU{\CrComp}\cup\YYi{0}}\in\FolStabilizer{\func|_{\regU{\CrComp}\cup\YYi{0}}}$.

Furthermore, suppose there are disks of type~\ref{dt1}.
Then for each $i\in\{0,\ldots,\sgOrd-1\}$ the map $\gdif^{\sgOrd}|_{\DDX{i}}:\DDX{i}\to\DDX{i}$ is the following composition:
\[
   \qdif^{i}\circ\dif^{-1}\circ\qdif^{\sgOrd-i}:
    \DDX{i}
        \underbrace{
            \xrightarrow{\qdif} \DDX{i+1} \xrightarrow{\qdif} \cdots \xrightarrow{\qdif}
        }_{\sgOrd-i-1 \ \text{arrows}}
    \DDLast
    \xrightarrow{\dif^{-1}\circ\qdif}
    \DDX{0}
    \underbrace{
        \xrightarrow{\qdif} \cdots \xrightarrow{\qdif}
        }_{i \ \text{arrows}}
    \DDX{i},
\]
which by~\ref{enum:h_prop:1} belongs to $\StabilizerId{\func|_{\DDX{i}}} \subset \FolStabilizer{\func|_{\DDX{i}}}$.

Similarly, suppose that there exist disks of type~\ref{dt2}.
Then for each $i\in\{0,\ldots,\mel-1\}$ the restriction map $\gdif^{\mel}:\EEX{i}\to\EEX{i}$ is a composition of the following maps:
\[
    \qdif^{i}\circ\kdif^{-1}\circ\qdif^{-i}:
    \EEX{i}
        \underbrace{
            \xrightarrow{\qdif} \EEX{i+1} \xrightarrow{\qdif} \cdots \xrightarrow{\qdif}
        }_{\mel-i-1 \ \text{arrows}}
    \EELast
    \xrightarrow{\kdif^{-1}\circ\qdif^{1-\mel}}
    \EEX{0}
    \underbrace{
        \xrightarrow{\qdif} \cdots \xrightarrow{\qdif}
        }_{i \ \text{arrows}}
    \EEX{i}.
\]
Whence $\gdif^{\sgOrd}=\gdif^{2\mel}=\qdif^{i}\circ\kdif^{-2}\circ\qdif^{-i} \in\StabilizerId{\func|_{\EEX{i}}} \subset\FolStabilizer{\func|_{\EEX{i}}}$, since $\kdif^{2}\in\StabilizerId{\func|_{\EEX{0}}}$.

\ref{enum:g_prop:gb_comm_ker_eta}
It remains to check that the isotopy class of $\gdif^{\sgOrd}$ in $\PSTB{\MBand}$ commutes with $\ker(\eta) = \myprod_{i=1}^{n}\PST{\YYi{i}}$.
Fix any $\func$-regular neighborhood $\Uman'$ of $\regU{\CrComp}$.

Let $\gamma \in \ker(\eta) = \myprod_{i=1}^{n}\PST{\YYi{i}} \subset\PSTB{\MBand}$ be any element and $\dif\in\STB{\MBand}$ be any representative of $\gamma$.
Then one can assume that $\dif$ is fixed on $\Uman'\cup\YYi{0}$.
We will show that $\gdif^{\sgOrd}$ is isotopic in $\STB{\MBand}$ to a diffeomorphism supported in $\Uman'\cup\YYi{0}$, and therefore commuting with $\dif$.
This will imply that the isotopy classes of $\gdif^{\sgOrd}$ and $\dif$ commute in $\PSTB{\MBand}$.

Indeed, for $i=1,\ldots,n$ denote $\Uman_i:=\YYi{i} \cap (\Uman'\setminus\regU{\CrComp})$.
Then $\Uman_i$ is evidently a $\func|_{\YYi{i}}$-regular neighborhood (even a \term{collar}) of $\partial\YYi{i}$.
As noted in the proof of~\ref{enum:g_prop:gb_in_FS}, $\gdif^{\sgOrd}|_{\YYi{i}} \in\StabilizerId{\func|_{\YYi{i}}}$, $i=1,\ldots,n$.
Hence, by Lemma~\ref{lm:h_Sidf__SidfX}, $\gdif^{\sgOrd}|_{\YYi{i}}$ is isotopic relatively some neighborhood of $\partial\YYi{i}$ to a diffeomorphism supported in $\Uman_i$.
That isotopy extends by the identity to an isotopy of all $\gdif^{\sgOrd}$.

Applying these arguments for all $i=1,\ldots,n$, we deform $\gdif^{\sgOrd}$ in $\STB{\MBand}$ to a diffeomorphism supported in $\Uman'\cup\YYi{0}$.

This completes the proof of Theorem~\ref{th:g_eta}.

\subsection*{Acknowledgments}
% This work was supported by a grant from the Simons Foundation (1030291, 1290607, I.V.K. and S.I.M.).
The authors are also grateful to Dmitrii Pasechnik and Dave Benson for very useful explanations and discussions on MathOverflow of the finiteness of a codimension of ideals in algebras of polynomials, \cite{Pasechnik:MO:FinCodim, Maksymenko:MO:IdFinCodim}.

%\bibliographystyle{plainurl} %amsalpha %apsrev %plain %nature %gost780u .....
%\bibliography{biblio}
\def\cprime{$'$} \def\cprime{$'$} \def\cprime{$'$} \def\cprime{$'$}
  \def\cprime{$'$} \def\cprime{$'$} \def\cprime{$'$} \def\cprime{$'$}
  \def\cprime{$'$} \def\cprime{$'$} \def\cprime{$'$} \def\cprime{$'$}
  \def\cprime{$'$} \def\cprime{$'$}

\end{document}